\documentclass[reqno,12pt]{amsart}
\usepackage{amsmath, amssymb, amsthm, epsfig, tikz, caption, sagetex}
\usepackage{hyperref, latexsym, times, xurl, color, xcolor, stackrel, setspace, array, booktabs}
\usetikzlibrary{arrows.meta,decorations.pathreplacing,calc,positioning}
\usetikzlibrary{decorations.markings, arrows,angles,quotes}
\usepackage{wrapfig,lipsum}
\usepackage[shortlabels]{enumitem}
\usepackage[mathscr]{euscript}
\usepackage[left=2.5cm, right=2.5cm, bottom=2cm, top=2cm]{geometry}
\def\today{\ifcase\month\or
  January\or February\or March\or April\or May\or June\or
  July\or August\or September\or October\or November\or December\fi
  \space\number\day, \number\year}
\usepackage{dsfont}
\newtheorem{proposition}{Proposition}[section]
\newtheorem{lemma}[proposition]{Lemma}
\newtheorem{corollary}[proposition]{Corollary}
\newtheorem{theorem}[proposition]{Theorem}

\newtheorem*{theorem*}{Theorem}
\theoremstyle{definition}

\theoremstyle{remark}
\newtheorem*{remark}{Remark}
\numberwithin{equation}{section}

\DeclareMathOperator{\sgn}{\mathrm{sgn}}

\DeclareMathOperator*{\Res}{Res}
\DeclareMathOperator{\TV}{TV}
\DeclareMathOperator{\li}{li}
\DeclareMathOperator{\err}{err}
\DeclareMathOperator{\tri}{tri}

\DeclareMathOperator{\sech}{sech}

\DeclareMathOperator{\Ei}{Ei}

\newcommand{\R}{\mathbb{R}}

\newcommand{\Z}{\mathbb{Z}}

\colorlet{darkred}{red!70!black}
\colorlet{darkblue}{blue!70!black}
\colorlet{darkgreen}{black!60!white}
\tikzset{
  contour/.style={
    postaction={decorate},
    decoration={markings, mark=at position 0.5 with {\arrow{latex}}}
  },
  contourend/.style={->, >=latex},
  motion/.style={->},
  motiondotted/.style={dotted, ->},
}

\begin{document}

\title{Optimal bounds for sums of bounded arithmetic functions}
\author[Chirre and Helfgott]{Andr\'{e}s Chirre and  Harald Andr\'{e}s Helfgott}

\date{\today}

\address{Departamento de ciencias - secci\'on matem\'aticas, Pontificia Universidad Cat\'olica del Per\'u, Lima, Per\'u}
\email{cchirre@pucp.edu.pe}

\address{IMJ-PRG, Université Paris Cité, Bâtiment Sophie Germain, 8 Place Aurélie Nemours, 75205 Paris Cedex 13, France.} 
\email{harald.helfgott@gmail.com}

\allowdisplaybreaks
\numberwithin{equation}{section}


\begin{abstract}
Let $A(s) = \sum_n a_n n^{-s}$ be a Dirichlet series with meromorphic continuation. Say we are given information on the poles of $A(s)$ with $|\Im s| \leq T$ for some large constant $T$. What is the best way to use such finite spectral data to give explicit estimates on sums $\sum_{n\leq x} a_n$?

  The problem of giving explicit bounds on the Mertens function
  $M(x) = \sum_{n\leq x} \mu(n)$ illustrates how open this basic
  question was.
  Bounding $M(x)$ might seem equivalent to estimating
  $\psi(x) = \sum_{n\leq x} \Lambda(n)$ or the number of primes $\leq x$.
  However, we have long had fairly good explicit bounds on prime counts,
  while bounding $M(x)$ remained a notoriously stubborn problem.

  We prove a sharp, general result on sums $\sum_{n\leq x} a_n n^{-\sigma}$ for $a_n$ bounded, giving an
  optimal way to use information on the poles of $A(s)$ with $|\Im s|\leq T$ and no data on the poles above. 
 Our bounds on $M(x)$ are stronger than previous ones by many orders of magnitude.
  (Similar results for $\psi(x)$ are given in a companion paper.) Using rigorous residue computations by D. Platt, we obtain, for $x\geq 1$,  \[|M(x)|\leq \frac{3}{\pi\cdot 10^{10}}\cdot x + 11.39 \sqrt{x}.\]
This is a corollary of our main result, essentially an explicit formula with the contribution of each pole clearly stated; we shall discuss how this finer structure can be
useful.
  
  Our proof mixes a Fourier-analytic approach in the style of Wiener--Ikehara with contour-shifting, using optimal approximants of Beurling--Selberg type (Carneiro--Littmann, 2013);  for $\sigma=1$, the approximant in  (Vaaler, 1985) reappears. While we proceed independently of existing 
  explicit work on $M(x)$ and $\psi(x)$,  
 our method has an important step in common with work on another problem
 by (Ramana--Ramaré, 2020).
\end{abstract}

\maketitle

  \section{Introduction}

  \subsection{Basic problem}

  Many problems in analytic number theory involve
  estimating sums $\sum_{n\leq x} a_n$ of arithmetic functions.
  Here ``arithmetic function'' means ``a sequence $\{a_n\}_{n=1}^\infty$ that number 
  theorists study'' or, most often, a sequence $\{a_n\}$
  such that the Dirichlet
  series $\sum_n a_n n^{-s}$ converges absolutely for $\Re s> \sigma_0$  and has
  meromorphic continuation to a function $A(s)$ on $\mathbb{C}$.

  Two basic examples to keep in mind are:
  \begin{itemize}
    \item $a_n=\Lambda(n)$, where
      $\Lambda$ is the von Mangoldt function;
      then $A(s) = -\zeta'(s)/\zeta(s)$;
    \item $a_n = \mu(n)$, where $\mu$ is the Möbius function; then
      $A(s) = 1/\zeta(s)$.
  \end{itemize}
These examples are both paradigmatic and bread-and-butter, in the sense that the need for estimates on $\sum_{n\leq x} \mu(n)$ and $\sum_{n\leq x} \Lambda(n)$ is everywhere in analytic number theory.
  \subsection{Dichotomy}

In the case $a_n = \Lambda(n)$, there were several sorts of useful estimates, though they used spectral data suboptimally.
Our result on $\Lambda(n)$ is the subject of the companion paper \cite{Nonnegart}.

For $a_n = \mu(n)$, the situation was far worse. There were
  no correct, direct, explicit analytic bounds on the Mertens function $M(x) = \sum_{n\leq x} \mu(n)$ in the literature. (See \cite{RamEtatLieux} for a survey.)

This situation gives one answer to the question ``why care about explicit estimates?'' We know that the Prime Number Theorem (that is, $\sum_{n\leq x} \Lambda(n) = (1+o(1)) x$) and $M(x) = o(x)$ are equivalent, yet the two problems turn out to be qualitatively different once we ask for explicit bounds.
 The underlying analytical issue is that the residues of $1/\zeta(s)$ appear in the explicit formula for $M(x)$, and one cannot give a general bound for them; bounding them would require showing that zeros of $\zeta(s)$ cannot be extremely close together, and that may lie deeper than the Riemann Hypothesis.

 What we can do is find all poles of $A(s)=1/\zeta(s)$ with $0<|\Im s|\leq T$
 for some large constant $T$
 (that is, all zeros of $\zeta(s)$ in that region)
 and compute the residue $1/\zeta'(s)$ of $A(s)$ at each pole. These computations can be done rigorously (\S \ref{subs:compinp}).
 The question is then how to best use this finite information
 to estimate $M(x)$. One can ask oneself the same question about 
 any other $A(s)$: shouldn't knowledge of finitely many poles be enough to give good bounds
 on partial sums?

A naïve student might set out to solve this problem by looking for a weight function whose Mellin transform is compactly supported. There is
no such thing, but, as we will see, there is a conceptually clean way to proceed
that essentially fulfills that dream.
 \subsection{Results}

 We show how to use information on the poles of $A(s)$ with $|\Im s|\leq T$ optimally, without using zero-free regions or any other information for
 $|\Im s|> T$.  




 \subsubsection{Estimates for $a_n$ bounded}

For our most general result, we will need a very mild technical condition. We will ask for a function to be bounded on a 
``ladder'', that is, a union of segments
\begin{equation}\label{eq:sapli}
S = ((-\infty,1] \pm i T)\cup
\bigcup_n (\sigma_n+i[-T,T])\;\;\;\;
\text{for some 
$\{\sigma_n\}_{n=0}^\infty$ with $\sigma_0 = 1$ and $\sigma_n\to -\infty$,}
\end{equation}
that we can use to shift a contour to $\Re s = -\infty$.
``Bounded'' here implies ``regular and bounded''.
The sum 
$\sum_{\rho\in \mathcal{Z}_{A}(T) \cup \{\sigma\}}$ in \eqref{eq:lilece} should be read as $\lim_{n\to\infty}
\sum_{\rho\in \mathcal{Z}_{A}(T) \cup \{\sigma\}: \Re \rho > \sigma_n}$. See Fig.~\ref{fig:ladder}.

 \begin{theorem}\label{thm:mainthmA}
  Let $A(s)=\sum_n a_n n^{-s}$ extend meromorphically to 
  $\mathbb{C}$. Assume
   $a_\infty = \sup_n |a_n|<\infty$. 
  Let $T\geq 4\pi$. Assume
  $A(s) T^s$ is bounded on some $S$ as in \eqref{eq:sapli}.
%
Then, for any $\sigma\in \mathbb{R}$, $x>e^2 T$, 
   \begin{equation}\label{eq:lilece}\begin{aligned}\frac{1}{x^{1-\sigma}} \sum_{n\leq x}
   \frac{a_n}{n^{\sigma}} = 
a_\infty \iota_{\delta,\sigma} +  \delta
   \sum_{\rho\in \mathcal{Z}_{A}(T) \cup \{\sigma\}} \Res_{s=\rho} \left(
   w_{\delta,\sigma}(s) A(s) x^{s-1}\right)  + \varepsilon(x,T) 
   ,\end{aligned}\end{equation}
   for $\delta = \frac{\pi}{2 T}$. Here $\mathcal{Z}_{A}(T)$ is the set of poles $\rho$
   of $A(s)$ with $|\Im \rho|\leq T$, and
   \begin{equation}\label{eq:mainsmooth}
   |\iota_{\delta,\sigma}|\leq \delta\, \mathrm{tanhc}((\sigma-1) \delta)\leq \delta,\;\;\;\; \;\;
   w_{\delta,\sigma}(s) =
    \coth(\delta (s-\sigma)) - \tanh(\delta (1-\sigma)),
   \end{equation}
   where $\mathrm{tanhc}(x)$ equals $\frac{\tanh x}{x}$ for $x\ne 0$ and $1$ for $x=0$. 
   We bound the error term $\varepsilon(x,T)$ by
\[\left|\varepsilon(x,T)\right| \leq \frac{\pi}{4}\, \frac{\frac{a_\infty}{L} + \frac{a_\infty}{L^2} + 
I}{T^2} + \frac{2 a_\infty}{x},\;\;\text{where}\;\;
L= \log \frac{x}{T},\;\;
I = \frac{1}{2} \sum_{\xi = \pm 1}\int_0^\infty t |A(1-t +i\xi T)| x^{-t} dt.\]
 \end{theorem}

 We can read \eqref{eq:lilece} as follows. The term $\iota_{\delta,\sigma}$ is inherent to using data on $A(s)$ only up to $T$. By \eqref{eq:mainsmooth}, $|\iota_{\delta,0}|\leq \tanh \frac{\pi}{2 T}$ and
 $|\iota_{\delta,1}|\leq \frac{\pi}{2 T}$; these bounds are sharp,
as will be shown in Proposition \ref{prop:counterex}.

The next term in \eqref{eq:lilece} is the contribution $\sum_\rho$ of poles up to $T$.
The weight $w_{\delta,\sigma}$ is optimal for our support: as we shall show,  $\mathds{1}_{[-1,1]}(t) \cdot w_{\delta,\sigma}(1 + i T t) $
is the Fourier transform of
 an extremal function in the sense of Beurling--Selberg found by
Carneiro--Littmann \cite{zbMATH06384942}.
That restriction is continuous, since
$w_{\delta,\sigma}(1\pm i T) = 0$. We can see $w_{\delta,\sigma}$ is a shifted, rotated cotangent plus a constant.

The error term $\varepsilon(x,T)$ is tiny in theory ($o(1)$ as $x\to \infty$) and practice (small compared to $1/T^2$).

 \begin{corollary}\label{cor:mertens} 
 Assume that all zeros of $\zeta(s)$ with $|\Im s|\leq T$ are simple,
 where $T\geq 4\pi$. Let $\sigma\geq -1$.
      Then, for any $x\geq e^2 T$ such that $\max_{r\leq 1} 1/|\zeta(r\pm i T)|\leq \log^2 x$,
\begin{equation*}
\left|\sum_{n\leq x} \frac{\mu(n)}{n^{\sigma}} -
\left(\delta
   \sum_{\rho\in \mathcal{Z}_*(T)} \frac{w_{\delta,\sigma}(\rho)}{\zeta'(\rho)} x^{\rho-\sigma} + 
   \frac{1}{\zeta(\sigma)}\right)\right|
\leq \frac{\pi/2}{T-1}\cdot x^{1-\sigma} + \frac{2}{x^\sigma},
   \end{equation*}
where $\delta = \frac{\pi}{2 T}$,
$\mathcal{Z}_*(T)$ is the set of non-trivial zeros $\rho$
of $\zeta(s)$ with $|\Im \rho|\leq T$, and
 $w_{\delta,\sigma}$ is as in \eqref{eq:mainsmooth}. 
\end{corollary}
Thanks to residue computations for $1/\zeta(s)$
 with $|\Im s|\leq T = 10^{10}+1$, carried out by D. Platt, and to an easy computation of 
 $\min_{-1/64\leq \sigma\leq 1} |\zeta(\sigma\pm i T)|$ for the same $T$, we 
 can apply Corollary \ref{cor:mertens} immediately. However, while
 Thm.~\ref{thm:mainthmA} is optimal in general, we can exploit the fact that  $\mu$ is
 supported on square-free integers. After deriving new estimates on square-free integers (\S \ref{sec:mainsqfr}), we will
 be able to improve the leading constant in our bounds on $M(x)$ by a factor of $1/\zeta(2)$.

 \begin{corollary}\label{cor:mertensimple}
 Let $M(x) = \sum_{n\leq x} \mu(n)$ and $m(x) = \sum_{n\leq x} \mu(n)/n$.
Then, for $x\geq 1$,
\[|M(x)|\leq \frac{3}{\pi\cdot 10^{10}}\cdot x + 11.39 \sqrt{x},\;\;\;\;\;\;\;\;\;\;
   |m(x)|\leq \frac{3}{\pi\cdot 10^{10}} + \frac{11.39}{\sqrt{x}}.
\]
\end{corollary}
 Here $11.39$ comes mainly from $\delta
   \sum_{\rho\in \mathcal{Z}_*(10^{10}+1)}
\left|\frac{\coth(\delta \rho)}{\zeta'(\rho)}\right| =11.35051\dotsc$. Later, in \S \ref{subs:companal}, we will discuss what one might do to obtain cancellation in the sum over $\rho$.




\begin{figure}[th]
\scalebox{.6}{
\begin{tikzpicture}[>=latex,scale=0.6,
                    horizaux/.style={line width=1pt},   
                    vertaux/.style={gray,line width=0.25pt},
                    unionarrow/.style={-Latex,line width=0.8pt}]
\def\T{4}
\def\Left{-24}
\def\CrossSz{0.08}
\def\RightScale{1.6}

\draw[vertaux,->] (0,-5) -- (0,5);

\draw[vertaux] (\Left,0) -- (0,0);                 

\foreach \y in {\T,-\T}
  \draw[horizaux] (0,\y) -- (\Left,\y);      

\foreach \x in {-1,-4.5, -19}
  \draw[horizaux] (\x,-\T) -- (\x,\T);        

\filldraw[black] (-1, 0) circle (1.2pt) node[above right] {\(\sigma_1\)};
\filldraw[black] (-4.5, 0) circle (1.2pt) node[above right] {\(\sigma_2\)};
\filldraw[black] (-19, 0) circle (1.2pt) node[above right] {\(\sigma_3\)};

\begin{scope}[xscale=\RightScale,clip]

  \filldraw[black] (1, 0) circle (1.2pt) node[above right] {\(\sigma_0=1\)};

  \draw[vertaux,->] (0,0) -- (2.2,0);               

  \draw[horizaux] (1,-\T) -- (1,\T);      

  \foreach \y in {\T,-\T}{
    \begin{scope}[shift={(1,\y)},xscale={1/\RightScale}]
      \fill (0,0) circle[radius=1pt];
    \end{scope}
  }

  \node[below right=2pt] at (1,-\T) {$1-iT$};
  \node[above right=2pt] at (1,\T)  {$1+iT$};

  \foreach \y in {\T,-\T}
    \draw[horizaux] (1,\y) -- (0,\y);

  \coordinate (SUnionTarget) at (1,2.2);
\end{scope}

\node[font=\Large\bfseries] (Slabel) at ($(SUnionTarget)+(6,1)$) {$S$};
\draw[unionarrow] (Slabel.west) to[out=160,in=-20] (SUnionTarget);

\end{tikzpicture}
}
\caption{A ladder $S$: climb it leftwards to shift a contour to $\Re s = -\infty$.}\label{fig:ladder}
\end{figure}
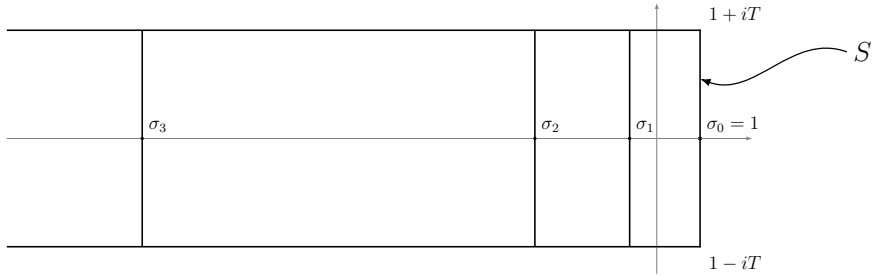

\subsection{Context and methods}
\subsubsection{Existing results on the Mertens function}\label{subs:predif}
 Estimating $M(x)$, even in the non-explicit regime, seemed to be intrinsically challenging.\footnote{Thus Titchmarsh: ``The finer theory of $M(x)$ is
 extremely obscure, and the results are not nearly so precise as the corresponding ones in the prime-number problem'' \cite[\S 14.26]{MR882550}. One could, however, counter that $M(x)$ is typical, and the prime-number problem is a happy exception: $\psi(x)$ is governed by $-\zeta'(s)/\zeta(s)$, and
the residue of a function of the form $-f'(s)/f(s)$ at a zero $s=\rho$ of $f(s)$ is just the zero's multiplicity.}
 What explicit bounds existed were of one of these types:
 \begin{enumerate}[(a)]
 \item Brute-force computational bounds for small $x$ ($x\leq 10^{16}$),
 \item\label{it:interM} Elementary bounds -- essentially, complicated variants of Chebyshev (\S \ref{subs:Mxbounds}),
 \item\label{it:iterato} Bounds obtained by means of combinatorial identities
    from bounds on $|\psi(x)-x|$ combined with elementary bounds on $M(x)$ as in
   \ref{it:interM}. The identities are often iterated to derive results of the form 
   $|M(x)|\leq c x/\log x$ from results of the form $|M(x)|\leq \epsilon x$ and $|\psi(x)-x|\leq C x (\log x)^\beta \exp(-c \sqrt{\log x})$; see \S \ref{subs:iterprodhist} and \S \ref{subs:iterprod}. The method is related to classic, elementary proofs of the equivalence of PNT and $M(x)=o(x)$.
   \end{enumerate}
 

Until very
recently, the best bound of the form $|M(x)|\leq \epsilon x$ was
\begin{equation}\label{eq:CEM}|M(x)|\leq \frac{x}{4345}\;\;\;\;\;\;\;\;\;
\text{for $x\geq 2160535$,}
\end{equation}
proved in \cite{zbMATH05257415} by a method of type \ref{it:interM} above.
There is now
a preprint \cite{Daval2} showing that
\begin{equation}\label{eq:daval}|M(x)|\leq \frac{x}{160383}
\;\;\;\;\;\;\;\;\;
\text{for $x\geq 8.4\cdot 10^9$}
\end{equation}
by a method of type \ref{it:iterato}, using \eqref{eq:CEM} as an input.
Such is the context for our bounds on $M(x)$.



\subsubsection{Strategy}
It is by now a commonplace observation that it is often best to estimate sums
$\sum_{n\leq x} a_n$ by first approximating them by smoothed sums
$\sum_{n\leq x} a_n \eta(n/x)$, where $\eta$ is continuous. One way to proceed then is to take Mellin transforms to obtain 
\begin{equation}\label{eq:smooperr}\sum_{n\leq x} a_n \eta(n/x) = \int_{\sigma-i\infty}^{\sigma+i\infty} A(s) x^s M\eta(s) ds,\end{equation} where $A(s)$ is the Dirichlet series $\sum_n a_n n^{-s}$. (Perron's formula is essentially the same,
but for unsmoothed sums, and with an error term; it is a little more complicated to prove because
the Fourier transform $\widehat{\mathds{1}_{[0,1]}}$ is not in $L^1(\mathbb{R})$.) Now,  one can convince oneself that the restriction of $M\eta(s)$ to a vertical
line cannot be compactly supported -- and indeed it cannot: a holomorphic or meromorphic function cannot equal a compactly supported function on a line.\footnote{Not
even on the border of a strip of holomorphy, by the Schwarz reflection principle.}

Matters are clearer if, instead of defining $\eta$, we choose a weight
$\varphi:\mathbb{R}\to \mathbb{R}$ and work out
\begin{equation}\label{eq:schokolade}\int_{\sigma-i\infty}^{\sigma+i\infty} A(s) x^s \varphi\left(
\frac{s-\sigma}{i}\right) ds.\end{equation} The integral in \eqref{eq:schokolade} can be expressed as a sum involving
$a_n$ and $\widehat{\varphi}$ (Lemma \ref{lem:basicfour}). It is clear that we can
take $\varphi$ to be compactly supported and still have $\widehat{\varphi}$ be in
$L^1(\mathbb{R})$. This is not a new insight; it underlies 
the first half of the proof of the Wiener--Ikehara theorem (see, e.g., \cite[p. 43--44]{zbMATH05203418}).

While working on a different problem, Ramana and Ramaré \cite{zbMATH07182422}
not only gave a statement for general $\varphi$, but realized that, since their
$\varphi$ was piecewise polynomial, they could shift the contour for
each piece to the left. We observe, more generally, that it is enough
for a function $\varphi$ supported on a compact interval $\mathbf{I}$ to equal a holomorphic
or meromorphic function $\Phi$ {\em on $\mathbf{I}$} (as opposed to: on all of $\mathbb{R}$). We can then replace $\varphi$ by $\Phi$ in $\int_{\sigma + i \mathbf{I}}
A(s) x^s \varphi\left(\frac{s-\sigma}{i}\right)$, and then shift the contour.

Our way of estimating the resulting terms differs from that in
\cite{zbMATH07182422}. It leads us to an optimization problem -- how to best approximate a given function by a band-limited function. 
This is a problem of a kind first solved by Beurling,
and later by Selberg; depending on the function being approximated, the solution
can be that found by Beurling (and Selberg), or one given by Vaaler
\cite{zbMATH03919007}, or Graham and Vaaler \cite{zbMATH03758875}, or Carneiro
and Littmann \cite{zbMATH06384942}. 

Why this approach was not found before is a bit of a mystery. Selberg's rediscovery of Beurling's work dates to 1974 -- but even before then, a
non-optimal, compactly supported $\varphi$ chosen empirically could have given something close to optimal. 
By 1968, the first $3{\,}500{\,}000$ zeros of $\zeta(s)$ had been computed, with serious considerations of rigor \cite{zbMATH03304440}. If one goes that far up with a residue computation and then proceeds as we do, one already obtains
a substantially stronger bound than \eqref{eq:daval}. In fact,
even if one just computes residues up to
$|\Im s|\leq 1468$, which is how high verifications of RH had got
 \cite{zbMATH02526114} before the invention of programmable computers, one obtains a bound much stronger than the result $\limsup_{x\to \infty} 
|M(x)|/x\leq 1/105$ obtained in 1980
\cite{zbMATH03749090}.




\subsection{Structure of the paper}\label{subs:structu}
We start with Fourier-based replacements for Perron's formula (\S \ref{subs:perfour}) and bound the effect of changing a
sharp truncation to a weight (\S \ref{subs:diffweights}). 

 Section \ref{sec:twoasides} is an excursus, not needed for the main argument: we give a brief proof of PNT (\S \ref{subs:yapa}) using \S \ref{sec:coeur}, 
and also construct examples showing our results to be optimal in a strong sense (\S \ref{subs:diriconstr}): there exist series with no poles $\rho$ with $|\Im \rho|\leq T$ for which our bounds are arbitrarily close to sharp.

 We describe the solutions to
our optimization problems in \S \ref{sec:modidon}. 
In \S \ref{sec:contour}, we shift contours, as we have just explained.
Proving our main results is then rather easy (\S  \ref{subs:mainboud}); we estimate our sums $\sum_n a_n/n^\sigma$,
dealing with $\sigma=1$ by passing to a limit. We then apply our results to
$a_n=\mu(n)$ (\S \ref{subs:mainmu}), having bounded the contribution of
trivial zeros. We then derive clean corollaries by means of rigorous computational results.
In \S \ref{sec:mainsqfr}, we prove results on the distribution of square-free numbers, and in \S \ref{sec:sqsupp}, we show how to use them to improve our own results for $\mu$ and other functions of square-free support.

We discuss work past and future in \S \ref{sec:finremark} -- including a sketch of how to use our bounds for $M(x)$ and $\psi(x)$ as main ingredients for proving explicit bounds of the form $|M(x)|\leq c x/(\log x)^k$.
\subsection{Notation} We define the Fourier transform 
$\widehat{f}(x) = \int_{-\infty}^\infty f(t) e^{-2\pi i x t} dt$ for $f\in L^1(\R)$,
extended to $f\in L^2(\R)$ in the usual way (e.g., Thm.~9.13 in \cite{zbMATH01022658}, which, however, puts the factor of $2\pi$ elsewhere). We write $\|f\|_1$ and $\|f\|_{\infty}$ for the $L^1$ norm and the $L^\infty$ norm, respectively, and $\|f\|_{\TV}$ for the total variation of a function $f:\mathbb{R}\to \mathbb{C}$.

We write $O^*(R)$ to mean a quantity of absolute value at most
$R$. This convention goes back to Ramaré, in the early 90s; it is not unusual but has not yet become standard. Tao uses $O_\leq$. Both conventions are of course based on the standard asymptotic notation $O(\dotsc)$.

As is common, we write $O_{N,g}(\dotsc)$ (say) to mean that the implied constant depends on $N$ and on the definition of $g$, and nothing else. We use $C_{N,g}$, $\ll_{N,g}$ and $o_{N,g}(\dotsc)$ analogously.

When we write $\sum_{n}$, we mean a sum over positive integers; we use
$\sum_{n\in \mathbb{Z}}$ for a sum over all integers. We write $\mathbb{Z}_{>0}$ for the set of positive integers. We let $\mathds{1}_S$ be the characteristic function of a set $S\subset \mathbb{R}$, that is, $\mathds{1}_S(x)=1$ for $x\in S$, $\mathds{1}_S(x)=0$ for $x\not\in S$.

\subsection{Acknowledgements}
We are very much obliged to David Platt, on whose residue calculations we rely crucially; he has been extraordinarily responsive.
We are also very grateful to Fedor Nazarov, who showed us how to apply the
Whittaker--Shannon interpolation formula to obtain a function with sign changes at the half-integers \cite{484674} (thus rediscovering work by Vaaler).
Thanks are also due to Danylo Radchenko, for comments on \S \ref{subs:companal}, to K. Soundararajan, who made a suggestion that
was central to a previous approach of ours to the same problem
\cite{403629}, to Fedor Petrov, for two elegant remarks, and to several other contributors,
often anonymous, to MathOverflow and
Mathematics Stack Exchange.
 We also thank Kevin Ford, Andrew Granville, Habiba Kadiri, Nathan Ng and Olivier Ramaré for their feedback and encouragement.

       




  
  

  


  \section{From a complex integral to an $L^1(\R)$ approximation problem}\label{sec:coeur}
  \subsection{A smoothed Perron formula based on the Fourier transform}\label{subs:perfour} The phrase
    ``Smoothed Perron formula'' brings to mind an identity
    of the form \eqref{eq:smooperr}. What we mean is something
    slightly different, reflecting our strategy: we want to work with a fairly arbitrary
  weight function $\varphi$
  on a vertical integral, in place of $M\eta$ in \eqref{eq:smooperr};
  then we will work out what will happen on the
  side of the sum, knowing that the Fourier transform $\widehat{\varphi}$ will appear.


  The following proposition is close to several in the literature; it is a natural starting point
  for the Wiener--Ikehara Tauberian method, and it is also in some sense akin to the Guinand-Weil formula.
  Statements like Lemma \ref{lem:basicfour}
  are often given
  with $\varphi$ and $\widehat{\varphi}$ switched.
  Curiously, a statement in 
  the formal-proof project ``Prime Number Theorem and\dots''
  \cite{PNT+} is very close to Lemma~\ref{lem:basicfour}.\footnote{Indeed, it has now
  become equivalent to it, since we contacted the project participants to show them that one of the assumptions of their Lemma~1 was superfluous (March 3, 2025).}
  See also \cite[Thm. 2.1]{zbMATH07182422}. At any rate, we give a proof from scratch, as it is brief and straightforward.
  
  \begin{lemma}\label{lem:basicfour}
    Let $A(s) = \sum_n a_n n^{-s}$ be a Dirichlet series converging absolutely for
    $\Re s = \sigma$.
    Let $\varphi:\mathbb{R}\to \mathbb{C}$ be in $L^1(\R)$. Then, for any $x>0$ and any $T>0$,
    \[\frac{1}{2\pi iT} \int_{\sigma-i\infty}^{\sigma+i\infty} \varphi\left(\frac{\Im s}{T}\right)
    A(s) x^s ds = \dfrac{1}{2\pi}\sum_{n} a_n \left(\frac{x}{n}\right)^\sigma
    \widehat{\varphi}\left(\frac{T}{2\pi} \log \frac{n}{x}\right).\]
  \end{lemma}
  \begin{proof}
    By the dominated convergence theorem,
    $$
    \int_{\sigma-i\infty}^{\sigma+i\infty} \varphi\bigg(\frac{\Im s}{T}\bigg)
    A(s) x^s ds = 
        \int_{\sigma-i\infty}^{\sigma+i\infty} \varphi\bigg(\frac{\Im s}{T}\bigg)
        \sum_{n} a_n n^{-s} x^s ds =
\sum_{n} a_n \int_{\sigma-i\infty}^{\sigma+i\infty} \varphi\bigg(\frac{\Im s}{T}\bigg)
\left(\frac{x}{n}\right)^s ds$$
since $\varphi\in L^1(\R)$
and $\sum_{n} |a_n n^{-s}|\leq \sum_n |a_n| n^{-\sigma}$. Clearly,
$$\frac{1}{i T} \int_{\sigma-i\infty}^{\sigma+i\infty} \varphi\left(\frac{\Im s}{T}\right)
\left(\frac{x}{n}\right)^s ds =
\left(\frac{x}{n}\right)^\sigma \int_{-\infty}^\infty
\varphi(t) e^{i T t \log \frac{x}{n}} dt =
\left(\frac{x}{n}\right)^\sigma \widehat{\varphi}\left(\frac{T}{2\pi} \log \frac{n}{x}\right).
$$
  \end{proof}

  It will be useful to be able to integrate on the very edge of the region of absolute convergence of $A(s)$, assuming that $A(s)$ extends continuously to the edge.



  
    \begin{lemma}\label{lem:edge}
	Let $A(s) = \sum_n a_n n^{-s}$ be a Dirichlet series converging absolutely for
	$\Re s > 1$, extending continuously to the segment $1 + i [-T,T]$.
        Let $\varphi:\mathbb{R}\to \mathbb{C}$ be in $L^1(\R)$ and supported on $[-1,1]$ with $\widehat{\varphi}(y) = O(1/y^{\theta})$ for some
        $\theta\geq 1$ as $y\to +\infty$. Assume that
         $\sum_{n>1} \frac{|a_n|}{n(\log n)^{\theta}} < \infty$.
Then, for any $x>0$,
	\begin{equation}\label{eq:perreq2}
		\frac{1}{2\pi iT} \int_{1-i T}^{1+i T} \varphi\left(\frac{\Im s}{T}\right)
	A(s) x^s ds =  \dfrac{1}{2\pi}\sum_{n} a_n\,\frac{x}{n}\,
	\widehat{\varphi}\left(\frac{T}{2\pi} \log \frac{n}{x}\right).\end{equation}
\end{lemma}
If $a_n$ is bounded, then the condition $\sum_{n>1} \frac{|a_n|}{n(\log n)^{\theta}} < \infty$
         holds for every $\theta>1$, and so it is enough to assume
         that $\widehat{\varphi}(y) = O(1/y^\theta)$ for some $\theta>1$
         as $y\to +\infty$.
\begin{proof}
We can apply Lemma \ref{lem:basicfour} for $\Re s=1+\epsilon$, $\epsilon>0$ arbitrary. Then
  \begin{equation}\label{eq:perreq122}\frac{1}{2\pi iT} \int_{1+\epsilon-i T}^{1+\epsilon+i T} \varphi\left(\frac{\Im s}{T}\right)
	A(s) x^s ds = \dfrac{1}{2\pi}\sum_{n} a_n \left(\frac{x}{n}\right)^{1+\epsilon}
	\widehat{\varphi}\left(\frac{T}{2\pi} \log \frac{n}{x}\right)\end{equation}
    since $\varphi$ is supported on 
    $[-1,1]$.
Now let $\epsilon\to 0^+$. Then
the left side of \eqref{eq:perreq122} tends to the left side of \eqref{eq:perreq2} (by continuity of $A$ and because $\varphi$ is compactly supported).
The right side of \eqref{eq:perreq122} tends to the right side of \eqref{eq:perreq2} by 
dominated convergence, since
\[
    \sum_n \frac{|a_n|}{n} \left|\widehat{\varphi}\left(\frac{T}{2\pi} \log \frac{n}{x}\right)\right|\ll_{T,x} 1 + \sum_{n>e} \frac{|a_n|}{n} 
\frac{1}{(\log n)^{\theta}} < \infty.\]
\end{proof}

\subsection{From sums to $L^1$ norms.}\label{subs:diffweights}
In \S \ref{subs:perfour}, we expressed
sums of the form 
$\sum_n a_n \frac{x}{n} \widehat{\varphi}\left(\frac{T}{2\pi} \log \frac{x}{n}\right)$ in terms of integrals
of $\varphi$. Our aim is actually to estimate
$\sum_{n\leq x} a_n/n^\sigma$.
We can write $\sum_{n\leq x} a_n (x/n)^\sigma$ as $\sum_n a_n \frac{x}{n} I\left(\frac{T}{2\pi} \log \frac{n}{x}\right)$ for
$I(y) = \mathds{1}_{(-\infty,0]}(y)\cdot e^{2\pi (1-\sigma) y/T}$.

We will bound the difference between $\sum_n a_n \frac{x}{n} \widehat{\varphi}\left(\frac{T}{2\pi} \log \frac{x}{n}\right)$ and $\sum_n a_n \frac{x}{n} I\left(\frac{T}{2\pi} \log \frac{n}{x}\right)$ 
by applying Prop.~\ref{prop:sumwiz} with $f = \widehat{\varphi}-I$.
The main term in \eqref{eq:mocambo} is
  $\frac{2\pi}{T} \|f\|_1$, so, in effect, we are reducing our problem to that of minimizing
  $\|\widehat{\varphi}-I\|_1$.


\subsubsection{General bound}
Let us first prove a simple tail bound.
\begin{lemma}\label{lem:adartail}
 Let $x,T>0$, $y_0\leq -T/\pi$, and define $\omega_0 = e^{\frac{2\pi}{T} y_0} x$. Then
\begin{equation}\label{eq:twiddled}\frac{\{\omega_0\}/\omega_0}{\left(\frac{T}{2\pi} \log \frac{\omega_0}{x}\right)^2} + \sum_{n\leq \omega_0} \frac{1/n}{\left(\frac{T}{2\pi} \log \frac{n}{x}\right)^2}\leq 
\frac{2\pi}{T |y_0|}.\end{equation}
\end{lemma}
\begin{proof}
We bound the left side of \eqref{eq:twiddled} by
$$
\frac{(2\pi)^2 }{T^2}
\left(\frac{\{\omega_0\}}{\omega_0 \log^2 \frac{\omega_0}{x}} + \sum_{n\leq \lfloor \omega_0\rfloor} \frac{1}{n \log^2 \frac{n}{x}}\right)\leq
\frac{(2\pi)^2 }{T^2}
\int_0^{\omega_0} \frac{d\omega}{\omega \log^2 \frac{\omega}{x}} = 
\frac{(2\pi)^2}{T^2 \log \frac{x}{\omega_0}} = \frac{2\pi}{T |y_0|},$$
since $1/(\omega \log^2(\omega/x))$ is decreasing for $\omega\leq x/e^2$,
and $\omega_0\leq e^{-\frac{2\pi}{T} \frac{T}{\pi}} x = x/e^2$.
\end{proof}

Now we show how to bound a sum by an $L^1$ norm, as promised.
\begin{proposition}\label{prop:sumwiz}
Let $f:\mathbb{R}\to \mathbb{C}$ be a function in $L^1(\mathbb{R})$. 
Let $T>0$. Assume there are $y_0\leq -T/\pi$,
$\kappa>0$ such that (i)
$|f(y)|\leq \kappa/y^2$ for all $y\leq y_0$
and (ii) $f$ has bounded variation on
$[y_0,\infty)$.
    
    Then, for any $x>0$,
\begin{equation}
\label{eq:mocambo}
\sum_{n} \frac{1}{n}
  \left|f\left(\frac{T}{2\pi} \log \frac{n}{x}\right)\right| \leq
\frac{2\pi}{T} \left(\|f\|_1 + \frac{\kappa}{|y_0|}\right)
+ \frac{1}{2 x}
\left\|e^{-\frac{2\pi}{T} y} f(y)\right\|_{\text{$\TV$ on $[y_0,\infty)$}}.
\end{equation}
\end{proposition}

\begin{proof}
     For any integrable $g$, any positive integer $n$ and any 
     $r_-,r_+\geq 0$,
	\begin{equation}\label{eq:ordon}\int_{n-r_-}^{n+r_+} g(\omega) d\omega = \left(r_- +r_+\right) g(n) + 
	O^*\bigg(r_- \sup_{\omega\in [n-r_-,n]} |g(\omega)-g(n)| +
	r_+ \sup_{\omega\in [n,n+r_+]} |g(\omega)-g(n)|\bigg)\end{equation}
	and so, if $g$ has bounded variation on $[\omega_0,\infty)$ for some $\omega_0$, 
    \[
    \sum_{n\geq \omega_0}\nolimits^* g(n) =
    \int_{\lceil \omega_0\rceil - r}^\infty g(\omega) d\omega  +
    \frac{1}{2} O^*\left(
	\|g\|_{\text{$\TV$ on $[\lceil \omega_0\rceil,\infty)$}}\right)
    + r O^*\left(
	\|g\|_{\text{$\TV$ on $[\omega_0,\lceil \omega_0\rceil]$}}\right)
    ,\]
    where $r = \min(1/2,\lceil \omega_0\rceil - \omega_0)$ and
    $\sum^*$ means that $n=\lceil \omega_0\rceil$ is weighted by $1/2+r$.
    We can ``top up'' that weight: 
    $g(\lceil \omega_0\rceil) = g(\omega_0) + O^*\left(
	\|g\|_{\text{$\TV$ on $[\omega_0,\lceil \omega_0\rceil]$}}\right)$, and so
    \begin{equation}\label{eq:desche}\begin{aligned}\sum_{n\geq \omega_0} g(n) &= \mathop{\sum\nolimits^*}_{n\geq \omega_0} g(n) +
    \left(\frac{1}{2} - r\right) g(\lceil \omega_0\rceil) \\ &=
    \int_{\lceil \omega_0\rceil - r}^\infty g(\omega) d\omega  +
    \frac{1}{2} O^*\left(
	\|g\|_{\text{$\TV$ on $[\omega_0,\infty)$}}\right) +
    \left(\frac{1}{2} - r\right) g(\omega_0).
    \end{aligned}\end{equation}
Let $\omega_0 = e^{\frac{2\pi}{T} y_0} x$. Define 
 $ g(\omega) = \frac{1}{\omega}
  \left|f\left(\frac{T}{2\pi} \log \frac{\omega}{x}\right)\right|$.
By Lemma \ref{lem:adartail} and $|f(y)|\leq \kappa/y^2$ for $y\leq y_0$,
\begin{equation}\label{eq:leather}\begin{aligned} 
\{\omega_0\} g(\omega_0) + \sum_{n\leq \omega_0} g(n) =
\frac{\{\omega_0\}}{\omega_0}\left|f\left(\frac{T}{2\pi} \log \frac{\omega_0}{x}\right)\right| &+ \sum_{n\leq \omega_0} \frac{1}{n}
\left|f\left(\frac{T}{2\pi} \log \frac{n}{x}\right)\right|\leq
\frac{2\pi \kappa}{T |y_0|}.\end{aligned}\end{equation}

From \eqref{eq:desche} and \eqref{eq:leather}, by
$1/2-r = \max(0, \{\omega_0\}-1/2) \leq \{\omega_0\}$ and $\lceil \omega\rceil- r\geq \omega_0$, 
$$\sum_n g(n) \leq
\int_{\omega_0}^\infty g(\omega) d\omega  +
    \frac{1}{2} O^*\left(
	\|g\|_{\text{$\TV$ on $[\omega_0,\infty)$}}\right) + 
    \frac{2\pi \kappa}{T |y_0|}.
$$
By a change of variables $y = \frac{T}{2\pi} \log \frac{\omega}{x}$,
$$\int_{\omega_0}^\infty g(\omega) d\omega = \frac{2\pi}{T}
\int_{y_0}^\infty|f(y)| dy \leq
 \frac{2\pi}{T} \|f\|_1.$$
Since $\TV$ is invariant under a change of variables,
$\|g\|_{\text{$\TV$ on $[\omega_0,\infty)$}} = \frac{1}{x} \left\|e^{-\frac{2\pi}{T} y} f(y)\right\|_{\text{$\TV$ on $[y_0,\infty)$}}$.
\end{proof}
{\em Remark.}
In general, for $g(y)$ decreasing and non-negative,
and $f(y)$ of bounded variation and going to $0$ as $y\to \infty$,
we can bound\footnote{Sketch of proof (F. Petrov): for $y_0\leq x_1\leq x_2\leq \dotsc$,
define $f_i = f(x_i)$, $g_i = g(x_i)$, $s_i = \sum_{j\geq i} |f_j-f_{j+1}|$.
Then $|g_i f_i - g_{i+1} f_{i+1}|\leq f_i s_i - f_{i+1} s_{i+1}$. Sum both
sides over $i$; the right-hand sum telescopes.}
$\|g\cdot f\|_{\text{$\TV$ on $[y_0,\infty)$}}\leq g(y_0) \|f\|_{\text{$\TV$ on $[y_0,\infty)$}}$.
Thus, to apply Prop.~\ref{prop:sumwiz}, it is enough to have a bound
on $\|f\|_{\text{$\TV$ on $[y_0,\infty)$}}$, rather
than on 
$\|e^{-\frac{2\pi}{T} y} f(y)\|_{\text{$\TV$ on $[y_0,\infty)$}}$. However,
we shall find it better to prove the latter kind of bound directly.

\subsubsection{Sums with support on square-free numbers}
Let us now prove a variant of Prop.~\ref{prop:sumwiz} specifically
for sequences with support on square-free numbers.
As is customary, we write $Q(x)$ for the number of square-free numbers $\leq x$, 
and define $R(x) = Q(x) - \frac{6}{\pi^2} x$. 

\begin{lemma}\label{lem:thirno}
Let $f:(y_0,\infty)\to \mathbb{C}$ be a function in $L^1$
with bounded variation. Let $x,T>0$ and 
$\omega_0 = e^{\frac{2\pi}{T} y_0} x$. Then, for 
 any $K\in \mathbb{C}$,
$$\begin{aligned}\sum_{n>\omega_0} \frac{\mu^2(n)}{n}  
\left|f\left(\frac{T}{2\pi} \log \frac{n}{x}\right)\right| &=
\frac{2\pi}{T} \cdot \frac{6}{\pi^2} \|f\|_1
\\ &+ 
\frac{|f(y_0^+)|}{\omega_0} (K-R(\omega_0))+
\frac{1}{x}
\int_{y_0^+}^\infty \left(K-R\left(e^{\frac{2\pi y}{T}} x\right)\right)
d\frac{|f(y)|}{e^{\frac{2\pi y}{T}}}.\end{aligned}$$
\end{lemma}
\begin{proof}

Let $g:(\omega_0,\infty)\to \mathbb{C}$ be integrable and of bounded variation, with $g(\omega) = o(1/\omega)$ for $\omega\to \infty$. By integration by parts,
\[\begin{aligned}
\sum_{n> \omega_0} \mu^2(n) g(n) &= \int_{\omega_0^+}^\infty g(\omega) dQ(\omega) =
\int_{\omega_0^+}^\infty g(\omega) d\left(\frac{6}{\pi^2} \omega\right) +
\int_{\omega_0^+}^\infty g(\omega) dR(\omega)\\
&= \frac{6}{\pi^2} \int_{\omega_0^+}^\infty g(\omega) d\omega 
- g(\omega_0^+) (R(\omega_0)-K) - \int_{\omega_0^+}^\infty (R(\omega) -K) dg(\omega)
\end{aligned}\]
Now let  $ g(\omega) = \frac{1}{\omega}
\left|f\left(\frac{T}{2\pi} \log \frac{\omega}{x}\right)\right|$. 
By a change of variables $\omega = e^{\frac{2\pi y}{T}} x$,
$$\int_{\omega_0^+}^\infty (K-R(\omega)) dg(\omega) = \frac{1}{x}
\int_{y_0^+}^\infty \left(K-R\left(e^{\frac{2\pi y}{T}} x\right)\right)
d\left(e^{-\frac{2\pi y}{T}} |f(y)|\right)$$
and of course
$\int_{\omega_0^+}^\infty g(\omega) d\omega = \frac{2\pi}{T}
\int_{y_0^+}^\infty|f(y)| dy =
 \frac{2\pi}{T} \|f\|_1$.
\end{proof}

In principle, the best way to apply Lemma \ref{lem:thirno} may be to set
$K = R(x)$, and then use bounds on $R(x)-R(\omega)$, i.e., 
bounds on square-free numbers on intervals. However,
in order not to depend on
the explicit literature (see \S \ref{subs:sqfrdiscuss}), we will assume only a bound of the simple form
 $|R(x)|\leq c\sqrt{x}$.
 \begin{proposition}\label{prop:sumwizsqf}
 Let $f:\mathbb{R}\to \mathbb{C}$ be a function in $L^1(\R)$. 
Let $T>0$. Assume that there are $y_0\leq -T/\pi$,
$\kappa>0$ such that (i)
$|f(y)|\leq \kappa/y^2$ for all $y\leq y_0$
and (ii) $f$ has bounded variation on
$[y_0,\infty)$. Let $x>0$. Assume that for some $c>0$, $|R(\omega)|\leq c \sqrt{\omega}$ for all $\omega\geq \omega_0$,
where $\omega_0 = e^{\frac{2\pi}{T} y_0} x$. Then
\[\begin{aligned}\sum_{n} \frac{\mu^2(n)}{n}  
\left|f\left(\frac{T}{2\pi} \log \frac{n}{x}\right)\right| &\leq
\frac{2\pi}{T}   \left(\frac{6}{\pi^2} \|f\|_1  
+ \frac{\kappa}{|y_0|} + \frac{c}{2\sqrt{x}}
\|e^{-\frac{\pi y}{T}} f(y)|_{(y_0,\infty)}\|_1\right)
\\
&+ \frac{c}{\sqrt{x}} \left(
\frac{|f(y_0^+)|}{e^{\frac{\pi}{T} y_0}} 
+ \|e^{-\frac{\pi y}{T}} f(y)\|_{\text{$\TV$ on $(y_0,\infty)$}} \right).\end{aligned}\]
 \end{proposition}
\begin{proof}
Apply Lemma \ref{lem:thirno} with $K=0$.
By the assumption $|R(\omega)|\leq c \sqrt{\omega}$,
$$\frac{|f(y_0^+)|}{\omega_0} |R(\omega_0)|\leq 
c \frac{|f(y_0^+)|}{\sqrt{\omega_0}} =
\frac{c}{\sqrt{x}} \frac{|f(y_0^+)|}{e^{\frac{\pi}{T} y_0}},
$$
$$\left|\int_{y_0^+}^\infty R\left(e^{2\pi y/T} x\right)
d\frac{|f(y)|}{e^{\frac{2\pi y}{T}}}\right|\leq
c \sqrt{x} \int_{y_0^+}^\infty e^{\frac{\pi y}{T}} \left|d\frac{|f(y)|}{e^{\frac{2\pi y}{T}}}\right|.
$$
Since $d|e^{-\frac{\pi y}{T}} f(y)| = d(e^{\frac{\pi y}{T}} e^{-\frac{2\pi y}{T}} |f(y)|) = 
\frac{\pi}{T} e^{-\frac{\pi y}{T}} |f(y)| + e^{\frac{\pi y}{T}} d(e^{-\frac{2\pi y}{T}} |f(y)|)$,
\[
 \int_{y_0^+}^\infty e^{\frac{\pi y}{T}} \left|d\frac{|f(y)|}{e^{\frac{2\pi y}{T}}}\right| \leq
 \int_{y_0^+}^\infty d|e^{-\frac{\pi y}{T}} f(y)| +
 \frac{\pi}{T} \int_{y_0^+}^\infty e^{-\frac{\pi y}{T}} |f(y)| dy.
\]
Finally, we apply Lemma \ref{lem:adartail} to bound the tail terms.
\end{proof}
\subsection{Conclusions: sums, integrals, and an $L^1$ error term}
Let us now state more generally and carefully what we sketched at
the beginning of \S \ref{subs:diffweights}.
For $\{a_n\}_{n=1}^\infty$ and $\sigma\in \mathbb{R}\setminus \{1\}$, 
let
\begin{equation}\label{eq:sotodef}S_\sigma(x) = \sum_{n\leq x} \frac{a_n}{n^\sigma}\;\;\;\text{if $\sigma<1$,}\;\;\;\;\;\;\; \;\;\;\;\;\;\; S_\sigma(x)= \sum_{n\geq x} \frac{a_n}{n^\sigma}\;\;\;\text{if $\sigma>1$.}\end{equation}
Our task is to estimate these sums.

 For $\lambda\in \mathbb{R}\setminus \{0\}$, we define $I_\lambda$ to be the truncated exponential 
\begin{equation}\label{eq:truncexp}I_\lambda(y) = \mathds{1}_{[0,\infty)}(\sgn(\lambda) y)\cdot e^{-\lambda y}.\end{equation} 
The motivation for this definition is that, for any $\sigma\neq 1$ and $x\geq 1$:
\begin{equation}\label{eq:sombrerero}
S_\sigma(x) = x^{-\sigma} \sum_{n} a_n \frac{x}{n} I_{\lambda}\left(\frac{T}{2\pi}\log \frac{n}{x}\right),
\end{equation}
where $T>0$, and $\lambda = 2\pi (\sigma- 1)/T$. Now we assemble our results from so far.

\begin{proposition}\label{prop:summsec2}
Let $\{a_n\}_{n=1}^\infty$ be such that $\sup_n |a_n|\leq 1$.
Assume $A(s)=\sum_{n}a_n n^{-s}$ extends continuously to $1+i [-T,T]$ for some $T>0$.
Let $\varphi:\mathbb{R}\to \mathbb{C}$ be in $L^1(\mathbb{R})$ and supported on $[-1,1]$.
Assume $\widehat{\varphi}$ has bounded variation and
$\widehat{\varphi}(y)=O(1/|y|^\theta)$ as $y\to \infty$ for some $\theta>1$.

Let $S_\sigma$ be as in \eqref{eq:sotodef} for $\sigma\ne 1$.
Let $I_\lambda$ be as in \eqref{eq:truncexp} with $\lambda = \frac{2\pi (\sigma-1)}{T}$. Assume that there are $y_0\leq -T/\pi$,
$\kappa>0$ such that $|\widehat{\varphi}(y)-I_\lambda(y)|\leq \kappa/|y|^2$ for all $y\leq y_0$.
Then, for any $x>0$,
\begin{equation}\label{eq:mirino}\begin{aligned}
 S_\sigma(x) &= 
\frac{x^{-\sigma}}{i T} \int_{1-i T}^{1+i T} \varphi\left(\frac{\Im s}{T}\right)
	A(s) x^s ds +
     2\pi \frac{x^{1-\sigma}}{T} O^*\left(\|\widehat{\varphi}-I_\lambda\|_1 \right)\\
&+   O^*\left(2\pi \kappa \cdot \frac{x^{1-\sigma}}{T |y_0|}
+ \frac{1}{2 x^\sigma} \left\|e^{-\frac{2\pi}{T} y} (\widehat{\varphi}(y)-I_\lambda(y))\right\|_{\text{$\TV$ on $[y_0,\infty)$}} 
\right).
\end{aligned}\end{equation}
\end{proposition}
\begin{proof}
By \eqref{eq:sombrerero}, the triangle inequality and $|a_n|\leq 1$,
$$x^\sigma S_\sigma(x) =
\sum_{n} a_n \frac{x}{n} \widehat{\varphi}\left(\frac{T}{2\pi}\log \frac{n}{x}\right)
+  O^*\left(x \sum_n  \frac{1}{n} \left|f\left(\frac{T}{2\pi}\log \frac{n}{x}\right)\right|\right)
$$
for $f = \widehat{\varphi}-I_\lambda$. We apply Lemma \ref{lem:edge} to the
first sum, and estimate the second sum by Prop.~\ref{prop:sumwiz}.
The conditions of Lemma \ref{lem:edge} hold
by $\sum_{n>1} 1/(n \log^\theta n) < \infty$; those of
Prop.~\ref{prop:sumwiz} hold because $\widehat{\varphi}$ and $I_\lambda$
are of bounded variation.
\end{proof}

Finally, we give a variant of Prop.~\ref{prop:summsec2} for sequences supported on square-free numbers.
As one can see, the main difference is that the main $L^1$-term becomes smaller by a factor of $6/\pi^2$,
at the cost of an increase in some of the error terms.
\begin{proposition}\label{prop:summsec3}
Let all conditions of Prop.~\ref{prop:summsec2} hold. Assume that
$a_n=0$ whenever $\mu(n)=0$. 
Assume as well that $|R(\omega)|\leq c \sqrt{\omega}$ for all $\omega\geq \omega_0$ and some $c>0$,
where $\omega_0 = e^{\frac{2\pi}{T} y_0} x$ and $x>0$. Then
$$\begin{aligned}
 S_\sigma(x) &= 
\frac{x^{-\sigma}}{i T} \int_{1-i T}^{1+i T} \varphi\left(\frac{\Im s}{T}\right)
	A(s) x^s ds +
     \frac{12}{\pi}\cdot \frac{x^{1-\sigma}}{T} O^*\left( \|\widehat{\varphi}-I_\lambda\|_1 \right)\\
&+  \frac{2\pi x^{1-\sigma}}{T}  O^*\left(\frac{\kappa}{|y_0|} + \frac{c}{2\sqrt{x}}
\|e^{-\frac{\pi y}{T}}  (\widehat{\varphi}(y)-I_\lambda(y))  |_{(y_0,\infty)}\|_1\right)\\
&+
c x^{\frac{1}{2} - \sigma} O^*\left({{e^{-\frac{\pi}{T} y_0}}|\widehat{\varphi}(y_0^+)-I_\lambda(y_0^+)|} + \left\|e^{-\frac{\pi y}{T}} (\widehat{\varphi}(y)-I_\lambda(y))
\right\|_{\text{$\TV$ on $(y_0,\infty)$}} 
\right).
\end{aligned}$$
\end{proposition}
\begin{proof}
    Proceed as in the proof of Prop.~\ref{prop:summsec2}, but apply
    Prop.~\ref{prop:sumwizsqf} instead of Prop.~\ref{prop:sumwiz}.
\end{proof}
\section{Two asides}\label{sec:twoasides}
\subsection{The prime number theorem}\label{subs:yapa}
What we have so far is more than 
enough to prove that $M(x) = o(x)$ with nearly no extra work, given the input that $\zeta(s)$ does not vanish on $\Re s = 1$. In turn,
$M(x) = o(x)$ implies the Prime Number Theorem elementarily (Axer-Landau;
see \cite[\S 8.1]{MR2378655} or \cite{zbMATH02625234}), so we have
in effect obtained a proof of PNT as a side result. This is
unsurprising, given that we have so far followed a strategy related to the Wiener--Ikehara proof of PNT. Indeed the following can be seen as a variant of the Wiener--Ikehara Tauberian theorem; we could state it for any bounded sequence $\{a_n\}$ such that $A(s) = \sum_n a_n n^{-s}$ extends continuously to $\Re s = 1$.
\begin{proposition}\label{prop:MPNT}
Assume that $\zeta(s)\ne 0$ for every $s\in \mathbb{C}$ with $\Re s = 1$.
Then $M(x) = o(x)$.
\end{proposition}
\begin{proof}
Let $\varphi:\mathbb{R}\to \mathbb{C}$ be an absolutely continuous function supported on $[-1,1]$ such that $\varphi'$ has bounded variation
(meaning: the distributional derivative $D\varphi'$ of $\varphi'$
is a finite signed measure).
It follows easily\footnote{Or use, say, the standard example
$\varphi(t) = \mathds{1}_{[-1,1]}(t)\cdot (1-|t|)$, for which these properties 
are immediate (since $\widehat{\varphi}(t) = (\sin(\pi t)/(\pi t))^2$) and well known.
The properties are just as easy to prove for a more general $\varphi$, however.}
that $\varphi\in L^1(\mathbb{R})$,
$\widehat{\varphi}(t)=O(1/t^2)$ as $t\to \pm \infty$, $\widehat{\varphi}\in L^1(\mathbb{R})$,
and $\widehat{\varphi}$ has bounded variation (since 
$\widehat{\varphi}'(t)$ is the Fourier transform of $- 2\pi i t \varphi(t)$, 
whose derivative has bounded variation, and so $\widehat{\varphi}'(t)$ is $O(1/t^2)$ and must hence be integrable). Let $A(s) = 1/\zeta(s)$.

Then, for any $T>0$, $x>0$,
by Prop.~\ref{prop:summsec2} with $\sigma = 0$, $\theta=2$ and
(say) $y_0 = -T$,
$$ M(x) = \sum_{n\leq x} \mu(n)
= \frac{1}{i T} \int_{1-i T}^{1+i T} \varphi\left(\frac{\Im s}{T}\right)\frac{x^s}{\zeta(s)} ds + O_\varphi\left( \frac{x}{T} + 1\right).$$

We let $g(t) = \varphi(t)/\zeta(1+it T)$.
Since $g$ is continuous and compactly supported, it is in $L^1(\mathbb{R})$.
Hence, by the Riemann-Lebesgue lemma,
$$\int_{1-i T}^{1+i T} \varphi\left(\frac{\Im s}{T}\right) \frac{x^s}{\zeta(s)}
ds = i x T \int_{-\infty}^\infty g(t) e^{i t T \log x}  dt = i x T\cdot \widehat{g}\left(-\frac{T}{2\pi} \log x\right) = o_{\varphi, T}(x)$$
for any fixed $T$ as $x\to \infty$.

Thus, $|M(x)|\leq  C\cdot (x/T + 1) + o_{\varphi,T}(x)$ for some constant $C=C_\varphi$,
and so $|M(x)|\leq 2 C x/T$ for any $x$ larger than some $N=N_{\varphi,T}$. Since $T$ is arbitrarily large, we obtain $M(x) = o(x)$.
\end{proof}

Proceeding a little differently,
we could have worked with a non-negative sequence $\{a_n\}$ instead of a bounded
sequence, thus obtaining a proof of PNT in the form $\psi(x) = x$. We do not,
in part because Graham and Vaaler have already done exactly that \cite[Thm.~10]{zbMATH03758875}.

The interest of the approach here lies in its economy: we develop a generally useful framework (\S \ref{sec:coeur}) and then prove Prop.~\ref{prop:MPNT} extremely quickly, before the sections
using optimal approximants.



\subsection{Dirichlet series showing $\tanh(\pi/2 T)$ is optimal}\label{subs:diriconstr}
We will now show that the leading term of Theorem \ref{thm:mainthmA} is tight. The construction is inspired by the well-known example $A(s) = (\zeta(s+i)+\zeta(s-i))/2 = \sum_n \cos(\log n)\cdot n^{-s}$, often used to show that one cannot derive asymptotics for $\sum_{n\leq x} a_n$ just from the behavior of $\sum_n a_n n^{-\sigma}$ for real $\sigma\to 1^+$.

We will be working with approximations $f(t)$ to the square wave $\sgn \cos t$.
\begin{lemma}\label{lem:kernos}
  Let $f:\mathbb{R}/2\pi\mathbb{Z}\to \mathbb{C}$ be bounded, with
  bounded total variation and $\int_0^{2\pi} f(t) dt = 0$. Let $g(t) = \sgn \cos t$. Then, for $T\geq 1$, $x\geq 2$, $N\in \mathbb{Z}_{>0}$ such that
  $T \log x = 2\pi N + \pi/2$,
  \begin{equation}\sum_{n\leq x} f(T \log n) = x \tanh \frac{\pi}{2 T}+ 
O\left(\frac{x}{T} \cdot \|f-g\|_1\right) + O(T \|f\|_{\TV} \log x  + \|f\|_\infty + 1).\end{equation}
Moreover, if $L = \lim_{\sigma\to 1^+} \sum_{n}  \frac{f(T \log n)}{n^\sigma}$ exists,
  \begin{equation}\label{eq:agout}\sum_{n\leq x} \frac{f(T \log n)}{n} = L +
\frac{\pi}{2 T} + O\left(\frac{\|f-g\|_1}{T}\right) + 
O\left(\frac{T \|f\|_{\TV}+\|f\|_\infty}{x}\right)  
  .\end{equation}
\end{lemma}
The argument below can be modified easily to give bounds on $\sum_{n\leq x} f(T \log n) n^{-\sigma}$. We work out $\sigma = 1$ partly because it is a tricky case.
\begin{proof}
Write $F:\mathbb{R}\to \mathbb{C}$ for the pull-back of $f$, i.e., $F(t)$ is just 
$f(t \bmod 2\pi)$ (or, as we write colloquially, $f(t)$). By \eqref{eq:ordon}, 
$$\begin{aligned}\sum_{n\leq x} f(T \log n) 
&=
\int_1^x f(T \log y) dy + O(\|F(T \log y)\|_{\text{$\TV$ on $[1,x]$}} + \|f\|_\infty).
\end{aligned}$$
For $\sigma>1$, we bound instead
$$\sum_{n>x} f(T \log n) n^{-\sigma} = 
\int_x^\infty f(T \log y) y^{-\sigma} dy 
+ O(\|F(T \log y) y^{-\sigma}\|_{\text{$\TV$ on $[1,x]$}}+ \|f\|_\infty x^{-\sigma}).$$
We note that
$$\sum_{n\leq x} \frac{f(T \log n)}{n} = 
\lim_{\sigma\to 1^+} \sum_{n\leq x}  \frac{f(T \log n)}{n^\sigma}
= \lim_{\sigma\to 1^+} \sum_{n}  \frac{f(T \log n)}{n^\sigma}
- \lim_{\sigma\to 1^+} \sum_{n>x}  \frac{f(T \log n)}{n^\sigma}.
$$
Total variation being invariant under changes of variables,
\[\|F(T \log y)\|_{\text{$\TV$ on $[1,x]$}} = 
\|F(t)\|_{\text{$\TV$ on $[0,T \log x]$}}
\ll \|f\|_{\TV} (T\log x + 1),\]
\[\left\|\frac{F(T \log y)}{y^{\sigma}}\right\|_{\text{$\TV$ on $[x,\infty)$}} = 
\left\|\frac{F(t)}{e^{\sigma t/T}}\right\|_{\text{$\TV$ on $[T \log x,\infty)$}} 
\leq \|f\|_{\TV} \cdot x^{-\sigma} \sum_{n=0}^{\infty} e^{-2\pi \sigma n/T},\]
since $e^{-\sigma t}$ is decreasing and tends to $0$.
We bound the sum here by
$1/(1-e^{-2\pi \sigma/T}) < T/2\pi \sigma + 1$.
 
Clearly, $\int_1^x f(T \log y) dy = 
\frac{1}{T}\int_0^{T \log x} e^{\frac{t}{T}} f(t) dt$. Since $\int_0^{2\pi} f(t) dt = 0$ and
$T \log x = 2\pi N + \frac{\pi}{2}$,
$$
\int_0^{T \log x} e^{\frac{t}{T}} f(t) dt = 
\sum_{n=0}^{N-1}
\int_{2\pi n}^{2\pi (n+1)} \left(e^{\frac{t}{T}} - e^{\frac{2\pi}{T} n}\right) f(t) dt
+ \int_{2\pi N}^{2\pi N + \frac{\pi}{2}} e^{\frac{t}{T}} f(t) dt,$$
and, since $\int_0^{2\pi} g(t) dt = 0$, the same holds with $g$ instead of
$f$.
Now
$$\sum_{n=0}^{N-1} \int_{2\pi n}^{2\pi (n+1)} \left(e^{\frac{t}{T}} - e^{\frac{2\pi}{T} n}\right) |f(t) - g(t)| dt\leq 
  \|f-g\|_1\sum_{n=0}^{N-1} \left(e^{\frac{2\pi}{T} (n+1)} -
 e^{\frac{2\pi}{T} n}\right) < x \|f-g\|_1,
$$
$$\int_{2\pi N}^{2\pi N + \frac{\pi}{2}} e^{\frac{t}{T}} |f(t)-g(t)| dt
\leq x \|f-g\|_1.$$
It remains to note that
\[\begin{aligned}\int_0^{T \log x} e^{\frac{t}{T}} g(t) dt &=
\int_0^{2\pi N+\frac{\pi}{2}} e^{\frac{t}{T}} dt - 2
\sum_{n=0}^{N-1} \int_{2\pi n +  \frac{\pi}{2}}^{2\pi n + \frac{3 \pi}{2}} e^{\frac{t}{T}} dt\\
&= T\cdot (e^{\frac{2\pi N+\frac{\pi}{2}}{T}} - 1) - 2 T\cdot
\left(e^{\frac{\pi}{T}} - 1\right)
\sum_{n=0}^{N-1} e^{\frac{2\pi n + \frac{\pi}{2}}{T}}\\
&= T \left(e^{\frac{2\pi N+\frac{\pi}{2}}{T}}
- 1 - 2  \frac{ 
  e^{\frac{2\pi N+\pi/2}{T}} - 
e^{ \frac{\pi}{2 T}}}{e^{\frac{\pi}{T}} + 1}\right)
= T \left(x \tanh \frac{\pi}{2 T}+ O(1)\right).\end{aligned}\]
We conclude that
$$\sum_{n\leq x} f(T \log n) = x \tanh \frac{\pi}{2 T}+ O(1) + \|f-g\|_1 O\left(\frac{x}{T}\right) 
+ O\left(  \|f\|_{\TV} (T\log x + 1) + \|f\|_\infty\right).
$$

If $\sigma>1$,
$\int_x^\infty \frac{f(T \log y)}{y^{\sigma}} dy = 
\frac{1}{T}\int_{T\log x}^\infty  \frac{f(t)}{e^{(\sigma-1) t/T}} dt$, and
$$\frac{1}{T}\int_{T\log x}^\infty \frac{f(t)}{e^{(\sigma-1) t/T}}  dt= 
- \frac{1}{T} \int_{2\pi N}^{2\pi N +\pi/2} \frac{f(t)}{e^{(\sigma-1) t/T}}  dt +
\frac{1}{T} \sum_{n=0}^\infty \int_{2\pi (N+n)}^{2\pi (N+n+1)} \frac{f(t)}{e^{(\sigma-1) t/T}} dt.$$
Since $f$ has period $2\pi$ and $\int_0^{2\pi} f(t) dt = 0$,
$$\lim_{\sigma \to 1^+} \sum_{n=0}^\infty
\int_{2\pi (N+n)}^{2\pi (N+n+1)} \frac{f(t)}{e^{(\sigma-1) t/T}} dt = 
\lim_{\sigma\to 1^+}
\sum_{n=0}^\infty e^{\frac{(1-\sigma)\cdot 2\pi (N+n)}{T}}
\int_0^{2\pi}  \frac{f(t)}{e^{(\sigma-1) t/T}} dt = 0
$$
by dominated convergence, as the integral here is bounded by
$\|f\|_\infty  (1-a)$ for $a = e^{(1-\sigma) 2\pi/T}$, and
$\sum_{n=0}^\infty a^n (1-a) = 1 < \infty$. Of course
$$\lim_{\sigma\to 1^+}
\int_{2\pi N}^{2\pi N +\pi/2} \frac{f(t)}{e^{(\sigma-1) t/T}}  dt = 
\int_{0}^{\pi/2} f(t) dt = \frac{\pi}{2} + O(\|f-g\|_1).
$$
Hence
$$-\lim_{\sigma\to 1^+} \sum_{n>x}  \frac{f(T \log n)}{n^\sigma} = 
\frac{\pi}{2 T} + O\left(\frac{\|f-g\|_1}{T}\right) + 
O\left(\frac{(T+1) \|f\|_{\textrm{TV}}+\|f\|_\infty}{x}\right).
$$
\end{proof}

\begin{proposition}\label{prop:counterex}
Let $T\geq 1$. For every $\epsilon>0$, there are $\{a_n\}_{n=1}^\infty$, $|a_n|\leq 1$, such that
$A(s) = \sum_n a_n n^{-s}$ has meromorphic continuation to $\mathbb{C}$ with no poles with $|\Im s|\leq T$, and
\begin{equation}\limsup_{x\to\infty} \frac{1}{x}\sum_{n\leq x} a_n > (1-\epsilon) \tanh \frac{\pi}{2 T},\end{equation}
\begin{equation}\label{eq:chepita}\limsup_{x\to\infty} \sum_{n\leq x} \frac{a_n}{n} 
> A(1) + (1-\epsilon) \frac{\pi}{2 T}.\end{equation}
\end{proposition}
It goes without saying that we can obtain a sequence $\{a_n\}$ such that 
$\liminf_{x\to\infty} \frac{1}{x}\sum_{n\leq x} a_n > -(1-\epsilon) \tanh \frac{\pi}{2 T}$ simply
by flipping signs. Alternatively, we can take the same $\{a_n\}$ as above, and let
$x =\exp\left(\frac{2\pi N-\pi/2}{T_+}\right)$ instead of $x = \exp\left(\frac{2\pi N+\pi/2}{T_+}\right)$
in the proof below, and in Lem.~\ref{lem:kernos}.
\begin{proof}
Let $\sigma_K(t) = \sum_{k=0}^K c_k \cos(k t)$ be the $K$th Fej\'er sum for the square wave $\sgn \cos t$, that is, $$c_k = \begin{cases} \frac{4}{\pi} \left(\frac{1}{k} - \frac{1}{K+1}\right)
&\text{if $k$ odd,}\\ 0 &\text{if $k$ even.}\end{cases}$$
By, say, \cite[Lem.~2.2 (i) and (iii)]{zbMATH07503892}, $|\sigma_K(t)|\leq 1$ for all $t$.
(In other words, we have avoided the Gibbs phenomenon by using Fej\'er sums.) Let $T_+>T$ and define 
$a_n = \sigma_K(T_+\log n)$. Then
\begin{equation}\label{eq:dududu}\begin{aligned}\sum_n a_n n^{-s} &= \sum_n n^{-s} \sum_{k=0}^K \frac{c_k}{2} \left(e^{i k T_+\log n} + e^{- i k T_+ \log n}\right) = \sum_{k=0}^K \frac{c_k}{2} \left(\sum_n n^{-s+i k T_+} + \sum_n n^{-s-i k T_+}\right)
\\ &= \sum_{k=0}^K \frac{c_k}{2} (\zeta(s-i k T_+) + \zeta(s + i k T_+))\end{aligned}\end{equation}
for $\Re s >1$, and so, by meromorphic continuation, $A(s) = \sum_{k=0}^K \frac{c_k}{2} (\zeta(s-i k T_+) + \zeta(s + i k T_+))$ for all $s$. Since $c_0 = 0$, $A(s)$ has no poles with $|\Im s|< T_+$.

By Lemma \ref{lem:kernos} with $f = \sigma_K$ and $T_+$ instead of $T$, 
for $x = \exp\left(\frac{2\pi N+\pi/2}{T_+}\right)$, $N\in \mathbb{Z}_{>0}$,
\begin{equation}\label{eq:ubuth}\sum_{n\leq x} a_n = 
x \tanh \frac{\pi}{2 T_+}+ 
O\left(\frac{x}{T_+} \|\sigma_K- g\|_1\right) + O(T_+ \|\sigma_K\|_{\TV} \log x  + \|\sigma_K\|_\infty + 1),\end{equation}
where $g = \sgn \circ \cos :\mathbb{R}/2\pi \mathbb{Z}\to \mathbb{C}$.
We know that $\|\sigma_K\|_\infty\leq 1$. Clearly
$$\|\sigma_K\|_{\TV}
\leq \sum_{k=0}^K |c_k| \cdot\|\cos k t \|_{\TV} =
\sum_{k=0}^K
4 k |c_k| = \frac{16}{\pi} \sum_{\substack{0\leq k\leq K\\
    \text{$k$ odd}}} \left(1 - \frac{k}{K+1}\right) = O(K).$$
As $K\to \infty$,
$\|\sigma_K(t)-\sgn \cos t\|_1 \to 0$
(by \cite[Lem.~2.2 (ii) and (iii)]{zbMATH07503892} and $|\sigma_K(t)|\leq 1$).
Hence, we can choose $K$ such the term
$O\left((x/T) \|\sigma_K(t)-g \|_1 \right)$ in \eqref{eq:ubuth} is
 $O^*(\epsilon x/4 T)$. Choose $T_+>T$ such that $\tanh \frac{\pi}{2 T_+}
> \tanh \frac{\pi}{2 T} - \frac{\epsilon}{4 T}$. Then
$$\sum_{n\leq x} a_n = x \left(\tanh \frac{\pi}{2 T_+} + O^*\left(\frac{\epsilon}{4 T}\right)\right)
+ T_+ K \cdot O(\log x) + O(1) \geq (1-\epsilon) x \tanh \frac{\pi}{2 T}$$
once $x$ is larger than a constant depending on $T_+$, $K$ and $\epsilon$, or, what
is the same, on $T$ and $\epsilon$. 

To prove \eqref{eq:chepita}, apply Lemma \ref{lem:kernos} as before (with $f = \sigma_K$, $T_+$ instead of $T$, etc.) except we will use \eqref{eq:agout}. It is clear that
$\int_0^{2\pi} \sigma_K(t) dt = 0$. By \eqref{eq:dududu},
$\lim_{\sigma\to 1^+} \sum_n \frac{\sigma_K(T_+ \log n)}{n^\sigma}
= A(1)$. Hence
$$\sum_{n\leq x} \frac{a_n}{n} - A(1) = 
\frac{\pi}{2 T_+} + 
O\left(\frac{\|\sigma_K- g\|_1}{T_+}\right) + O\left(\frac{T_+ \|\sigma_K\|_{\TV}  + \|\sigma_K\|_\infty}{x}\right).$$
We proceed as before and obtain \eqref{eq:chepita}.

\end{proof}

For $\sigma>0$, we can proceed in the same way to show that the
same construction gives us $\limsup_{x\to \infty} \frac{1}{x^{1-\sigma}}
\sum_{n\leq x} a_n/n^\sigma > \frac{1-\epsilon}{\sigma-1} \tanh \frac{\pi (1-\sigma)}{2 T} $ if
$\sigma\ne 1$.
\section{Extremal approximants to the truncated exponential}\label{sec:modidon}
Our task is now to give band-limited approximations in $L^1$ norm
to a given function $I:\mathbb{R}\to \mathbb{C}$.
By ``band-limited'' we mean that our approximation is the Fourier transform
$\widehat{\varphi}$ of a function $\varphi$ supported on a compact interval (in our case, $[-1,1]$).



\vspace{0.1cm}

To be precise: let $I:\mathbb{R}\to \mathbb{C}$ be in $L^1(\R)$.
We want to find $\varphi:\mathbb{R}\to \mathbb{C}$ supported on $[-1,1]$, with 
$\varphi, \widehat{\varphi}\in L^1(\R)$, such that
\[
\|\widehat{\varphi}-I\|_1
\]
is minimal. 
The optimal $\widehat{\varphi}$ is sometimes called a {\em two-sided approximant}; the reason
is that there are one-sided approximants, namely, functions $\varphi$ satisfying the additional condition that $\widehat{\varphi}-I$ be non-negative or non-positive.
(We use one-sided approximants in \cite{Nonnegart}.)




Let $\lambda\in \mathbb{R}\setminus \{0\}$. We will consider the functions
$I = I_\lambda$ defined in
\eqref{eq:truncexp}. The problem of finding the optimal two-sided approximant for $I_\lambda$ was solved by Carneiro and Littmann \cite{zbMATH06384942}. 
Our task will be mainly to work out the rather nice Fourier transforms 
$\varphi_\lambda$ of
the approximants.

Results in the literature are often phrased in terms 
of {\em exponential type}. An entire function $F$ is of exponential type $2\pi\Delta$, with $\Delta>0$, if
$|F(z)| \ll_\epsilon e^{(2\pi\Delta + \epsilon) |z|}$. The Paley--Wiener
theorem states that, if $\varphi:\mathbb{R}\to\mathbb{C}$ is in $L^2(\mathbb{R})$
and supported in $[-\Delta,\Delta]$, then $\widehat{\varphi}$ is entire 
and of exponential type $2\pi\Delta$;
conversely, if $F$ is
exponential type $2\pi\Delta$, and $F|_\mathbb{R}$ lies
in $L^2(\mathbb{R})$, then $F=\widehat{\varphi}$ for some $\varphi\in L^2(\mathbb{R})$ supported in $[-\Delta,\Delta]$ (\cite[\S 5]{zbMATH03026314}, \cite[Ch. XVI, Thm.~7.2]{zbMATH01881986},  or
\cite[Thm.~19.3]{zbMATH01022658}).

{\em Remark.} As we shall see, when $\lambda\to 0^+$, the optimal two-sided
approximant to $I_\lambda$ tends to the optimal two-sided
approximant to $I_0=\mathds{1}_{[0,\infty)}$ found by Vaaler \cite{zbMATH03919007}.
  This is a ``cultural'' comment, in that we will keep
  $\lambda\ne 0$ throughout our work, letting $\lambda\to 0$ only at the end. 

\subsection{Carneiro--Littmann's approximant} 

As in \cite{zbMATH06384942}, we define
 $K_{\nu}$ for $\nu>0$ to be the entire function of exponential type $\pi$ given by
\begin{equation}\label{eq:ombroso}
K_{\nu}(z)=\dfrac{\sin\pi z}{\pi}\left\{\sum_{n}(-1)^n\left(\dfrac{e^{-\nu n}}{z-n}-\dfrac{e^{-\nu n}}{z}\right)\right\},
\end{equation}
and let
\[E_{\nu}(u) = \mathds{1}_{(0,\infty)}(u)\cdot e^{-\nu u} + \frac{1}{2} \mathds{1}_{\{0\}}(u).\]
Thus, $I_\lambda(u) = E_{|\lambda|}(\sgn(\lambda) u)$ for $u\ne 0$.
  Note that 
\begin{equation}\label{eq:apprinter}
K_\nu(n) = E_\nu(n)\;\;\;\text{for all $n\in \mathbb{Z}\setminus \{0\}$},\;\;\;\;\;\;\;\;\;\;\;\;K_{\nu}(0) = \frac{1}{e^{\nu} + 1}.
\end{equation}
\begin{remark}
In other words, $K_\nu$ is an {\em interpolant}; this is intimately related to
its optimality.
\end{remark}
\begin{proposition}\label{prop:carlitfou} Let $F(z)$ be an entire function of exponential type $2\pi$. Let $\lambda\in \mathbb{R}\setminus \{0\}$. Then
\[
\|F-I_\lambda\|_1\geq \dfrac{1-e^{-|\lambda|/2}}{|\lambda| (1+e^{-|\lambda|/2})} = 
\frac{\tanh(\lambda/4)}{\lambda}.\]
Equality is attained if and only if $F=\widehat{\varphi_\lambda}$, where 
\begin{align}  \label{eq:sonnenblum}
	\varphi_\lambda(t)=\mathds{1}_{[-1,1]}(t)\cdot \Phi_\lambda(t)
 \end{align}
and
\begin{align}
\Phi_\lambda(z) 
\label{eq:girasol3} &= \frac{\sgn(\lambda)}{4}\cdot \left(\coth\left(\frac{\pi z}{2 i}+\frac{\lambda}{4} \right) - \tanh \frac{\lambda}{4}\right)\\
\label{eq:girasol2}
&= \dfrac{i \sgn(\lambda)}{4 \cosh \frac{\lambda}{4}}\cdot 
\frac{\cos \frac{\pi z}{2}}{\sin\left(\frac{\pi z}{2}+\frac{\lambda}{4} i\right)}.
\end{align}
Note that $\widehat{\varphi_\lambda}(z) = K_{|\lambda|/2}(2 \sgn(\lambda) z)$.
\end{proposition}
It is easy to see that $\tanh(\lambda/4)/\lambda < 1/4$,
since $\tanh(0) = 0$ and $\tanh'(t) = \frac{1}{\cosh^2(t)}\leq 1$ for all $t\geq 0$,
with equality only at $t=0$. The Taylor series starts with $\tanh(\lambda/4)/\lambda = 
1/4 - \lambda^2/192 + \dotsc$.
\begin{proof} 
Apply \cite[Theorem~1]{zbMATH06384942} (with $c=0$ and $\delta=2$) to get
\begin{align*}  \label{L1CarneiroLittmann}
\|F-I_\lambda\|_1 & = \int_{-\infty}^\infty|F(\sgn(\lambda) u)-E_{|\lambda|}(u)|du  \geq \dfrac{1-e^{-|\lambda|/2}}{|\lambda| (1+e^{-|\lambda|/2})},
\end{align*} 
with equality if and only if $F(u)=K_{|\lambda|/2}(2 \sgn(\lambda) u)$.
We can work from now on with $\lambda>0$, as, for
$\lambda<0$, the optimal $\varphi_\lambda$ will be given by
$\varphi_\lambda(t) = \varphi_{-\lambda}(-t)$.

It is not hard to see from \eqref{eq:ombroso}
that, on $\mathbb{R}$, $K_{\lambda/2}$ is bounded. Since $I_\lambda$ is in $L^1(\R)$ and
$\|F-I_\lambda\|_1 < \infty$, $F$ is in $L^1(\R)$, and thus
$K_{\lambda/2}$ is also in $L^1(\R)$.
Hence, it is in $L^2(\R)$, and so,  since it is also of exponential type $\pi$,
the
Paley--Wiener theorem
gives us that $K_{\lambda/2}$ is the
Fourier transform (in the $L^2$ sense) of a function $f\in L^2(\mathbb{R})$ 
with support on $[-1/2,1/2]$.  Since $K_{\lambda/2}$ is in $L^1(\R)$, Fourier inversion
holds in the sense of \cite[Thm.~9.14]{zbMATH01022658}, that is,
$f=g$ pointwise a.e., where $g$ is the inverse Fourier transform of $K_{\lambda/2}$;
again because $K_{\lambda/2}$ is in $L^1(\R)$, $g$ is bounded and continuous.
Since $f$ is supported in $[-1/2,1/2]$ and $g$ is continuous, $g$ is supported
in $[-1/2,1/2]$.
We replace $f$ by $g$ from now on.
Since $f$ (that is, $g$) is bounded and compactly supported, 
$\widehat{f}$ is continuous, and, since $K_{\lambda/2}$ is continuous, it equals $\widehat{f}$ everywhere, not just almost everywhere.

Clearly, then, $F = \widehat{\varphi_\lambda}$ everywhere, where
$\varphi_\lambda(t) = f(t/2)/2$. It remains to determine $f$.

By the definition of Fourier transform, for
every $n\in \mathbb{Z}$, $K_{\lambda/2}(n) = \widehat{f}(n) = \int_{-1/2}^{1/2}
f(t) e^{-2\pi i n t}dt$, i.e., exactly the $n$th Fourier coefficient of a function 
$\bar{f}:\mathbb{R}/\mathbb{Z}\to \mathbb{C}$ such that $\bar{f}(t\bmod 1) = f(t)$ 
for $t\in [-1/2,1/2)$. (Here the fact that $K_{\lambda/2}=\widehat{f}$ everywhere matters.) 
By \eqref{eq:apprinter},
$K_{\lambda/2}|_\mathbb{Z}$ is in $\ell^1(\mathbb{Z})$. 
and, since $f$ is continuous, we have the Fourier series expansion 
\[\begin{aligned}f(t) = \overline{f}(t) &= \sum_{n\in \mathbb{Z}} \widehat{f}(n) e^{2 \pi i n t} = 
K_{\lambda/2}(0)+ \sum_n E_{\lambda/2}(n)e^{2\pi int}\\
	& = \dfrac{1}{e^{\lambda/2}+1} + \sum_n e^{-\lambda n/2 }e^{2\pi int} =
    \dfrac{1}{e^{\lambda/2}+1} + \dfrac{1}{e^{-2\pi it+\lambda/2}-1},\end{aligned}\]
for all $t\in [-1/2,1/2)$. (Recall $\sum_n$ means $\sum_{n=1}^\infty$.) Hence, for $\lambda>0$, \eqref{eq:sonnenblum} holds for 
$$\Phi_\lambda(z) = 
    \dfrac{1}{2}\left(\dfrac{1}{e^{\lambda/2}+1} +  
    \dfrac{1}{e^{\lambda/2 - \pi i z}-1}\right).$$
Thus, by $1/(e^a+1) = \frac{1}{2} (1-\tanh \frac{a}{2})$
and $1/(e^{b}-1) = \frac{1}{2}(\coth \frac{b}{2}-1)$ with $a=\lambda/2$, $b=\lambda/2-\pi i z$,
\[\Phi_\lambda(z) = \frac{1}{4} \left(-\tanh \frac{\lambda}{4} + \coth \left(\frac{\lambda}{4} + \frac{\pi z}{2 i}\right)\right)
\]
for $\lambda>0$. For $\lambda<0$, we let  $\Phi_\lambda(z) = \Phi_{-\lambda}(-z)$, and
obtain \eqref{eq:girasol3}.
 Lastly, we go from \eqref{eq:girasol3} to \eqref{eq:girasol2} by
 $\tanh s = - i \tan s i$, $\coth s = i \cot s i$, followed by
\[
 \tan \alpha i + \cot \beta  =
\frac{\sin \alpha i \sin \beta + \cos \alpha i \cos \beta}{\cos \alpha i \cdot \sin \beta} =
\frac{\cos(\beta- \alpha i)}{\cosh \alpha \cdot \sin \beta},\]
applied with $\alpha = \lambda/4$ and $\beta = \pi z/2 + \lambda i/4$.
\end{proof}


{\em Remark.} Vaaler finds the optimal two-sided approximant for $\sgn(x)$ of exponential type $2\pi$ \cite[Thm.~4]{zbMATH03919007}: it is $G(2 x)$, where $G(z)$ is as defined in \cite[Eq (2.3)]{zbMATH03919007}:
$$G(z) = 
\frac{\sin \pi z}{\pi} \left(\sum_{n\in \mathbb{Z}\setminus \{0\}} (-1)^n \sgn(n) \left(\frac{1}{z-n} + \frac{1}{n}\right) + 2 \log 2\right).$$
Then the optimal two-sided approximant for $\mathds{1}_{[0,\infty)}(x)$ is given by $(G(2 x)+1)/2$. 
Thanks to 
Abel's limit theorem \cite[\S 2.5, Thm.~3]{zbMATH07453874},
$\log 2 = \lim_{t\to 1^-} \log(1+t) = - \sum_n \frac{(-1)^n}{n}$, so
$2 \log 2 = - \sum_n \frac{(-1)^n \sgn(n)}{n}$.
Thus, by Euler's identity
$
\frac{\sin \pi s}{\pi}\sum_{n\in \Z}\frac{(-1)^n}{s-n}=1
$ with $s=2 x$, 
\begin{align*}
\dfrac{G(2 x)+1}{2}
&= \dfrac{\sin 2\pi x}{\pi}\left(\sum_{n}\dfrac{(-1)^n}{2 x-n} + \frac{1}{4 x}\right).
\end{align*}
On the other hand, for $\lambda > 0$, we know that the optimal two-sided approximation of exponential type $2\pi$ to
$\mathds{1}_{[0,\infty)}(x)\, e^{-\lambda x}$ is 
$$\widehat{\varphi_{\lambda}}(x)=\dfrac{\sin2\pi x}{\pi}\sum_{n}(-1)^n\left(\dfrac{e^{-\frac{\lambda}{2} n}}{2x-n}-\dfrac{e^{-\frac{\lambda}{2} n}}{2x}\right)= \dfrac{\sin2\pi x}{\pi}\left\{\sum_{n}\dfrac{(-1)^ne^{-\frac{\lambda}{2} n}}{2x-n} + \dfrac{(e^{\frac{\lambda}{2}} +1)^{-1}}{2 x}\right\}.
$$
Since $\sum_{n}\frac{(-1)^n}{2x-n}$ is convergent, another application of
Abel's limit theorem gives us that it equals $\lim_{\lambda\to 0^+}
\sum_{n}\frac{(-1)^ne^{-\frac{\lambda}{2} n}}{2x-n}$. Hence,
$\widehat{\varphi_\lambda}(x)$ tends to $\frac{G(2 x)+1}{2}$
for every $x$ (even if $2 x\in \mathbb{Z}$) as $\lambda\to 0^+$.

\begin{figure}[ht]
\begin{minipage}[c]{0.45\linewidth}
\centering
\includegraphics[width=\textwidth]{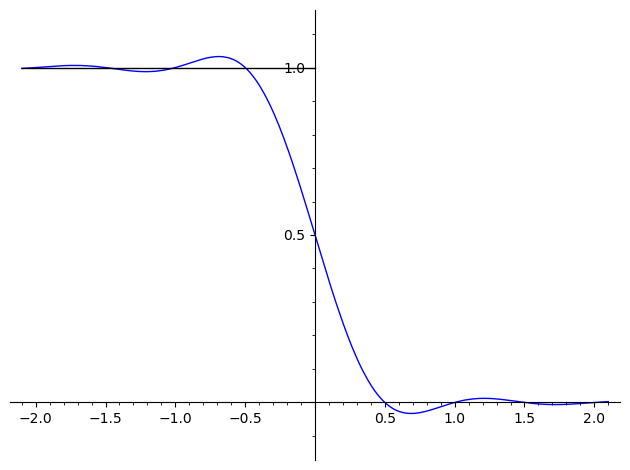}
\captionsetup{width=.85\linewidth}
\caption{Vaaler approximant to $\mathds{1}_{(-\infty,0]}(u)$}
\end{minipage}
\begin{minipage}[c]{0.45\linewidth}
\centering
\includegraphics[width=\textwidth]{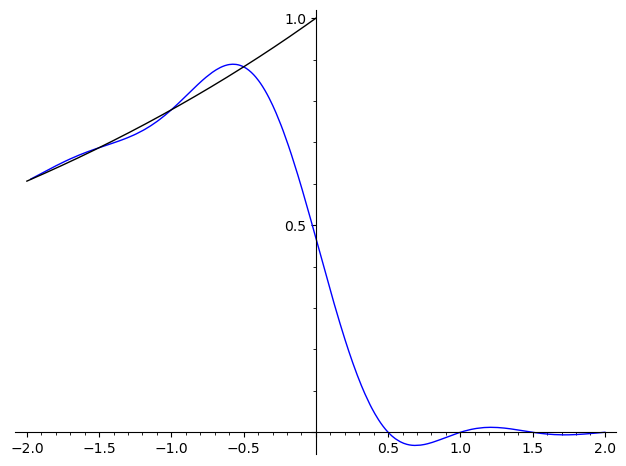}
\captionsetup{width=.9\linewidth}
\caption{Carneiro--Littmann approximant to $\mathds{1}_{(-\infty,0]}(u)\cdot e^{\lambda u}$ for $\lambda = 1/4$}
\label{fig:carnlittplot}
\end{minipage}
\end{figure}

\begin{lemma}\label{lem:benvenuto}
Let $\Phi_\lambda(z)$ be as in \eqref{eq:girasol3}--\eqref{eq:girasol2} for
$\lambda\in \mathbb{R}\setminus \{0\}$. Then $\Phi_\lambda(z)$ is a meromorphic function whose poles, all of
them simple, are at $2 n - \lambda i/2\pi$, $n\in \mathbb{Z}$; the residue at every pole
is $i \sgn(\lambda)/2\pi$. The zeros of $\Phi_\lambda(z)$ are at $2 n + 1$, $n\in \mathbb{Z}$. 
For any $c>-\lambda/2\pi$, $\Phi_\lambda(z)$ is bounded for $\Im z\geq c$.
For $t$ real, $|\Phi_\lambda(\pm 1 + i t)|\leq \frac{\pi}{8} |t|$.

\end{lemma}
\begin{proof}
The location of the zeros of $\Phi_\lambda$ is immediate from \eqref{eq:girasol2};
the statement on the residues of $\Phi_\lambda$ follows from \eqref{eq:girasol3} and
the fact that every residue of $\coth$ is $1$.
Boundedness holds for $\Im z\geq c$ by \eqref{eq:girasol3} and the fact 
that $\coth(s) = \frac{1+e^{-2 s}}{1-e^{-2 s}}$ is bounded on $\Re s\geq c$ for $c>0$ arbitrary.
Also by \eqref{eq:girasol3}, 
$$\left|\Phi_\lambda'(\pm 1+ i t)\right|=
\frac{\pi}{8} \tanh'\left(\frac{\pi t}{2} + \frac{\lambda}{4}\right) = 
\frac{\pi}{8} \sech^2 \left(\frac{\pi t}{2} + \frac{\lambda}{4}\right) \leq
\frac{\pi}{8}$$
for all real $t$,
since $\coth(\pm \frac{\pi}{2 i} + s) = \tanh s$ for all complex $s$ and 
$|\sech r|\leq 1$ for all real $r$. Since $\Phi_\lambda(\pm 1) = 0$, it follows that $|\Phi_\lambda(\pm 1 + i t)|\leq \frac{\pi}{8} |t|$ for $t$ real.
\end{proof}


\section{Shifting contours}\label{sec:contour}


It is now time to shift our contours of integration to $\Re s = -\infty$.

 The procedure is very straightforward: we have a meromorphic integrand. We start with a contour consisting of a straight path from $1+iT$ to $1-iT$, and shift it to the left. 
The result is a contour $\mathcal{C}_1$ consisting of a straight path from $1-i T$ to $-\infty -i T$ and another from $-\infty + i T$ to $1 + i T$.

``Shifting a contour to $\Re s = -\infty$'' truly means
shifting it first to a vertical line, and then to another further to the left, etc.,
taking care that the contour never goes through poles. We write 

    \begin{equation}\label{eq:defL}
      L \;=\; \bigcup_{n=1}^{\infty}\!\bigl(\sigma_n + i[-T,T]\bigr)
    \end{equation}
for these lines; here $1>\sigma_1>\sigma_2>\dotsc$ is a sequence of our choice, tending to $-\infty$ as $n\to \infty$.

We are of course integrating something, namely,
a meromorphic function on\footnote{By a meromorphic function on a set $R\subset\mathbb{C}$, we mean a meromorphic function on some open set containing $R$.}
a half-strip $R = (-\infty,1] + i [-T,T]$.
We keep track of the contribution of the poles $\rho$ that we do pick up; they are all the poles inside the region $R$ that we cover when we shift the contour to the left. 

\begin{lemma}\label{lem:kolobrz}
Let $G(s)$ be a meromorphic function on  
$R = (-\infty,1] + i [-T,T]$.  Assume 
that, for some $x_0\geq 1$, $G(s) x_0^s$ is bounded on $\partial R \cup L$, where
$L$ is as in \eqref{eq:defL}.
Then, for any $x>x_0$,
\[\frac{1}{2\pi i} \int_{1-iT}^{1+i T} G(s) x^s ds = \frac{1}{2\pi i}\int_{\mathcal{C}_\infty} G(s) x^s ds +  \sum_{\text{$\rho\in R$ a pole of $G$}} \Res\limits_{s=\rho} G(s) x^s,\]
where $\mathcal{C}_\infty$ is a straight path from $1-i T$ to $-\infty -i T$ and another from $-\infty + i T$ to $1 + i T$.
\end{lemma}
Here $\sum_{\text{$\rho\in R$ a pole of $G$}}$ means $\lim_{n\to\infty} \sum_{\text{$\rho\in R$ a pole of $G$}: \Re \rho> \sigma_n}$. The convergence of this sum in our applications will actually be absolute. On another note: being bounded on a contour such as $\partial R\cup L$
implies having no poles on that contour.
\begin{figure}[ht]
  \centering
  \begin{minipage}{0.45\textwidth}
    \centering
\begin{tikzpicture}[scale=0.6,
                    arrow/.style={postaction={decorate},
                  decoration={markings, mark=at position 0.5 with {\arrow{>}}}}]   
\useasboundingbox (-9,-6) rectangle (3,6);

\draw[->] (-8,0) -- (2,0) node[above] {};
\draw[->] (0,-5.2) -- (0,5.2) node[left] {};


\draw[darkblue,contour] (1,-4.5) -- (1,0);
\draw[darkblue,contour] (1,0) -- (1,4.5);

\draw[darkred,dotted] (-8,4.5) -- (-5.5,4.5);
\draw[darkred,dotted] (-8,-4.5) -- (-5.5,-4.5);

\draw[darkred,contour] (1,-4.5)  -- (-5.5,-4.5);
\draw[darkred,contour] (-5.5,-4.5) -- (-5.5,0);
\draw[darkred,contour] (-5.5,0) -- (-5.5,4.5);
\draw[darkred,contour] (-5.5,4.5) -- (1,4.5);
\draw[darkred,dotted,contour] (-3.5,0) -- (-3.5,4.5);
\draw[darkred,dotted,contour] (-3.5,-4.5) -- (-3.5,0);

\draw[darkred,dotted] (-5.5,4.5) -- (-8,4.5);
\draw[darkred,dotted] (-5.5,-4.5) -- (-8,-4.5);

\draw[darkred,motiondotted] (-4,-2.5) -- (-5,-2.5);
\draw[darkred,motiondotted] (-4,2.5) -- (-5,2.5);
\draw[darkred,motion] (-6,-2.5) -- (-7,-2.5);
\draw[darkred,motion] (-6,2.5) -- (-7,2.5);

\fill (1,4.5) circle (.07) node[above right] {\tiny $1+iT$};
\fill (1,-4.5) circle (.07) node[below right] {\tiny $1-iT$};
\fill (-3.5,4.5) circle (.07) node[above] {\textcolor{gray}{\tiny $\sigma_1+iT$}};
\fill (-3.5,-4.5) circle (.07) node[below] {\textcolor{gray}{\tiny $\sigma_1-iT$}};
\fill (-5.5,4.5) circle (.07) node[above] {\tiny $\sigma_2+iT$};
\fill (-5.5,-4.5) circle (.07) node[below] {\tiny $\sigma_2-iT$};

\foreach \x in {-2,-4} { \draw[darkgreen]
  ({\x+.08},.08) -- ({\x-.08},-.08) ({\x+.08},-.08) -- ({\x-.08},.08); }
\foreach \x in {-6} { \draw[very thin,darkgreen]
  ({\x+.08},.08) -- ({\x-.08},-.08) ({\x+.08},-.08) -- ({\x-.08},.08); }
\foreach \y in {1.413,2.0802,2.501,3.042,3.294,3.759,4.092,4.333} {
    \draw[darkgreen]
        (.58,{\y+.08}) -- (.42,{\y-.08}) (.42,{\y+.08}) -- (.58,{\y-.08})
        (.58,{-\y+.08}) -- (.42,{-\y-.08}) (.42,{-\y+.08}) -- (.58,{-\y-.08}); }
\foreach \y in {4.65,4.8,4.95} {
    \fill[darkgreen] (.5,\y) circle (.025) (.5,{-\y}) circle (.025); }
\node[rotate=270, anchor=south] at (1,2) {\scriptsize initial contour};
\node[rotate=270, anchor=south] at (1,-2.5) {\scriptsize initial contour};
\end{tikzpicture}
\end{minipage}
\hspace{40pt}
\begin{minipage}{0.45\textwidth}
\begin{tikzpicture}[>=latex,scale=0.6,
                    arrow/.style={postaction={decorate},
                  decoration={markings, mark=at position 0.5 with {\arrow{>}}}}]   
\useasboundingbox (-9,-6) rectangle (3,6);
\draw[->] (-8,0) -- (2,0) node[above] {};
\draw[->] (0,-5.2) -- (0,5.2) node[left] {};
\draw (-2,2.5) node {};
\draw[darkblue,arrow] (1,-4.5) -- (1,0);
\draw[darkblue,arrow] (1,0) -- (1,4.5);
\draw[darkred,arrow] (-8,4.5) -- (-5.5,4.5);
\draw[darkred,arrow] (-5.5,-4.5) -- (-8,-4.5);
\draw[darkred,arrow] (1,-4.5) -- (-5.5,-4.5);
\draw[darkred,arrow] (-5.5,4.5) -- (1,4.5);
\fill (1,4.5) circle (.07) node[above right] {\tiny $1+iT$};
\fill (1,-4.5) circle (.07) node[below right] {\tiny $1-iT$};
\foreach \x in {-2,-4,-6} { \draw[darkgreen]
  ({\x+.08},.08) -- ({\x-.08},-.08) ({\x+.08},-.08) -- ({\x-.08},.08); }
\foreach \y in {1.413,2.0802,2.501,3.042,3.294,3.759,4.092,4.333} {
    \draw[darkgreen]
        (.58,{\y+.08}) -- (.42,{\y-.08}) (.42,{\y+.08}) -- (.58,{\y-.08})
        (.58,{-\y+.08}) -- (.42,{-\y-.08}) (.42,{-\y+.08}) -- (.58,{-\y-.08}); }
\foreach \y in {4.65,4.8,4.95} {
    \fill[darkgreen] (.5,\y) circle (.025) (.5,{-\y}) circle (.025); }
\node[rotate=270, anchor=south] at (1,2) {\scriptsize initial contour};
\node[rotate=270, anchor=south] at (1,-2.5) {\scriptsize initial contour};
\node at (-3,5) {\scriptsize final contour};
\node at (-3,-5) {\scriptsize final contour};    
\end{tikzpicture}
\end{minipage}
\captionof{figure}{Proof of Lemma \ref{lem:kolobrz}: contour-shifting and result}\label{fig:kolobrz}
\label{fig:contour31}
\end{figure}
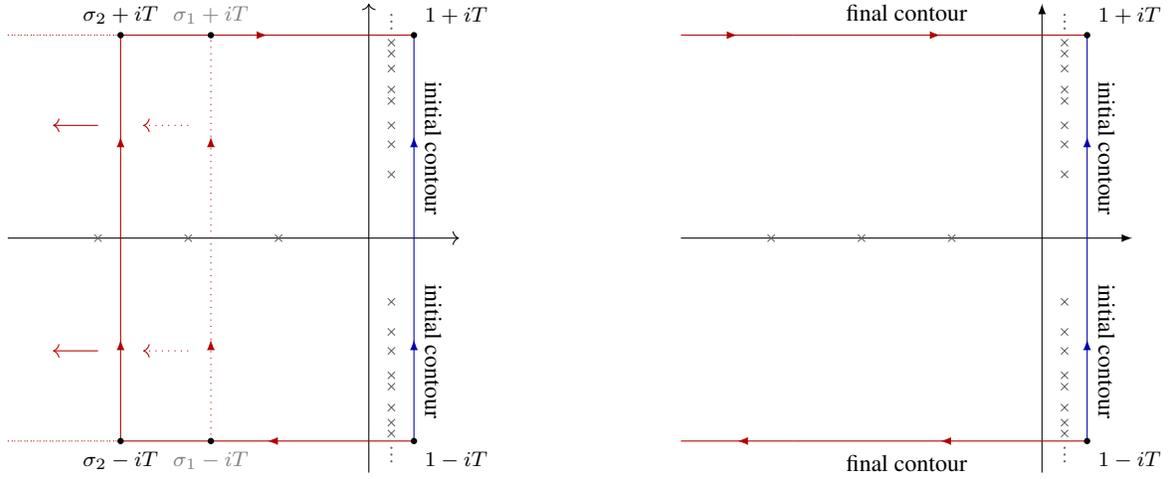

\begin{proof}
We simply shift the integral to the left:
$$\frac{1}{2\pi i}\int_{1-iT}^{1+i T} G(s) x^s ds = \frac{1}{2\pi i}\int_{\mathcal{C}_{n}} G(s) x^s ds + \sum_{\substack{\text{$\rho$ a pole of $G$}\\ \rho\in R,\;\Re \rho > \sigma_n}} \Res\limits_{s=\rho} G(s) x^s,$$
where $\mathcal{C}_{n}$ consists of the straight segments from $1-i T$ to $\sigma_n - i T$, from $\sigma_n- i T$ to $\sigma_n+i T$ and from $\sigma_n +i T$ to $1 + i T$. Let $n\to \infty$. Since $G(s) x_0^s$ is bounded on $L$,
$$\left|\int_{\sigma_n-i T}^{\sigma_n + i T} G(s) x^s ds\right|\ll (x/x_0)^{\sigma_n} T \to 0$$
as $n\to \infty$. Again by $|G(s) x^s|\leq |G(s) x_0^s (x/x_0)^s| \ll (x/x_0)^{\Re s}$, $$\lim_{n\to \infty} \int_{\sigma_n\pm i T} ^{1\pm i T} G(s) x^s ds = \int_{-\infty\pm i T} ^{1\pm i T} G(s) x^s ds.$$
\end{proof}

\begin{proposition}\label{prop:shiftbounded}
Let $F(s)$ be a meromorphic function on $R = (-\infty,1] + i [-T,T]$.
Assume that, for some $x_0\geq 1$, $F(s) x_0^s$ is bounded on $\partial R \cup L$, where
$L$ is as in \eqref{eq:defL}. Let $\Phi_\lambda$, $\lambda\ne 0$, be as in Prop.~\ref{prop:carlitfou}. Then, for any $x>x_0$,
\begin{equation}\label{eq:gruman}\begin{aligned}\frac{1}{2\pi i} \int_{1-iT}^{1+i T} \Phi_\lambda\left(\frac{s-1}{i T}\right) F(s) x^s ds &= \sum_{
\substack{\text{$\rho\in R$ a pole of $F(s)$}\\ \text{or $\rho = 1 + \frac{\lambda T}{2\pi}$ and $\lambda<0$}}} \Res\limits_{s=\rho} \left(\Phi_\lambda\left(\frac{s-1}{i T}\right) F(s) x^s \right) \\ &+ 
\frac{1}{16 T} \cdot O^*\left( \sum_{\xi=\pm 1} \int_0^\infty t |F(1 - t + i \xi  T)| x^{1-t} dt\right).
\end{aligned}\end{equation}
\end{proposition}
\begin{proof}
By Lemma \ref{lem:benvenuto}, $\Phi_\lambda\left(\frac{s-1}{i T}\right)$ is bounded on the part $L_0 = \bigcup_{n=n_0}^\infty (\sigma_n + i [-T,T])$ of $L$ given by all $n$ such that $\sigma_n < 1 + \lambda T/2\pi$; moreover, $\Phi_\lambda\left(\frac{s-1}{i T}\right) (1+\epsilon)^s$ is bounded on $\partial R$ for any $\epsilon>0$. 
Let $\epsilon>0$ be such that
$(1+\epsilon) x_0 < x$. Apply Lemma \ref{lem:kolobrz} with $G(s) = \Phi_\lambda\left(\frac{s-1}{i T}\right) F(s)$, $L_0$ instead of $L$, so as to avoid the pole at $1 + \lambda T/2\pi$,  and $(1+\epsilon) x_0$ instead of $x_0$. We obtain that the left side of \eqref{eq:gruman} equals
$$\sum_{\text{$\rho\in R$ a pole of $G$}} \Res\limits_{s=\rho} \left(\Phi_\lambda\left(\frac{s-1}{i T}\right) F(s) x^s\right) + 
\frac{1}{2\pi i}\int_{\mathcal{C}_\infty} \Phi_\lambda\left(\frac{s-1}{i T}\right) F(s) x^s ds.$$
Again by Lemma \ref{lem:benvenuto}, the poles of $G(s)$ in $R$ are just the poles 
of $F(s)$ in $R$, except that $1+\lambda T/2\pi$, being a pole of
$\Phi_\lambda\left(\frac{s-1}{i T}\right)$, 
is also a pole of $G(s)$, and of course it is in $R$ if and only if $\lambda<0$.

Yet again by Lemma \ref{lem:benvenuto},
$|\Phi_\lambda(\pm 1 + i t)|\leq \frac{\pi}{8} |t|$ for $t$ real.
Hence, 
$$\left|\frac{1}{2\pi i}\int_{\mathcal{C}_\infty} \Phi_\lambda\left(\frac{s-1}{i T}\right) F(s) x^s ds
\right|\leq \frac{1}{16 T} \sum_{\xi=\pm 1} \int_0^\infty t |F(1 - t + i \xi  T)| x^{1-t} dt.$$


\end{proof}

\section{Sums of bounded $a_n$. Sums of $\mu$.}

\subsection{The main result for $a_n$ bounded}\label{subs:mainboud}
It is now time to put everything together. We will start by applying Prop.~\ref{prop:summsec2}.
Since we want to minimize $\|\widehat{\varphi}-I_\lambda\|_1$ in
\eqref{eq:mirino}, we choose $\varphi=\varphi_\lambda$ as in Prop.~\ref{prop:carlitfou}.  For $z$ ranging in $[-1,1]$, as in our integral, $\varphi_\lambda(z) = \Phi_\lambda(z)$, and so we replace
$\varphi_\lambda$ by $\Phi_\lambda$. Then we use Prop.~\ref{prop:shiftbounded}, that is, we shift
the contour to the left, and are done, at least for $\sigma<1$.
If $\sigma>1$, we flip the sum $S_\sigma$. If $\sigma=1$, we work with $\sigma\to 1^-$.

Here, as always, $S_\sigma$ is as in \eqref{eq:sotodef} and
$I_\lambda$ is as in \eqref{eq:truncexp}.

\begin{proposition}\label{prop:gendau}
  Let $A(s)=\sum_n a_n n^{-s}$ extend meromorphically to 
  $\mathbb{C}$. 
  Let $T\geq 4\pi$. Assume 
   $a_\infty = \sup_n |a_n|<\infty$ and 
  $A(s) T^s$ is bounded on some $S$ as in \eqref{eq:sapli}.
For $\sigma\in \mathbb{R}\setminus \{1\}$, $x>e^2 T$, 
$$\begin{aligned}x^\sigma S_\sigma(x) &=
\delta \sgn(1-\sigma) \sum_{
\substack{\rho \in \mathcal{Z}_A(T)\\ \text{or $\rho = \sigma$ and $\lambda<0$}}} \Res\limits_{s=\rho} \left( w_{\delta,\sigma}(s) A(s) x^s \right)
+ O^*\left(\frac{\tanh \delta (\sigma-1)}{\sigma-1}\right)a_\infty \\&+ O^*\left(
\frac{\pi}{4} I +
\frac{\pi a_\infty}{4} \left(\frac{1}{L} + \frac{1}{L^2}\right)\right) \frac{ x}{T^2} + 
O^*\left(2 a_\infty\right),
\end{aligned}$$
where $\delta = \frac{\pi}{2T}$, $\mathcal{Z}_{A}(T)$ is the set of poles $\rho$
   of $A(s)$ with $|\Im \rho|\leq T$, 
   $$ w_{\delta,\sigma}(s) =
    \coth(\delta (s-\sigma)) - \tanh(\delta (1-\sigma)) 
    ,$$
    $$L = \log \frac{x}{T},\;\;\;\;I =  \frac{1}{2} \sum_{\xi=\pm 1} \int_0^\infty t |A(1 - t + i \xi  T)| x^{-t} dt.$$
\end{proposition}
\begin{proof}
Apply Prop.~\ref{prop:summsec2} with
$\varphi=\varphi_\lambda$ and  $A/a_\infty$ in place of $A$, where
$\varphi_\lambda$ is as in \eqref{eq:sonnenblum} and
$\lambda = 2\pi (\sigma-1)/T$. Let us examine each term from \eqref{eq:mirino}.

In the integral, $s$ goes in a straight line from $1- i T$ to $1 + i T$, and so
$\Im s = (s-1)/i$. We know from Proposition \ref{prop:carlitfou} that
$\|\widehat{\varphi} - I_\lambda\|_1 = \tanh(\lambda/4)/\lambda$.
Lemma \ref{lem:arborio} gives us 
\begin{equation}\label{eq:alicen1}
|\widehat{\varphi}(y)-I_\lambda(y)|\leq \kappa/y^2\end{equation} for $y\ne 0$ with $\kappa = 1/16\pi$, implying
$\widehat{\varphi}(y) = O(1/y^2)$ as $y\to \pm \infty$.
Lemma \ref{lem:artanor} bounds
\begin{equation}\label{eq:alicen2}\left\|e^{-\frac{2\pi}{T} y} (\widehat{\varphi}(y)-I_\lambda(y))\right\|_{\text{$\TV$ on $[y_0,\infty)$}} \leq
 c_1 + c_2 \frac{e^{-\alpha y_0}}{\alpha y_0^2}
\end{equation}
for $\alpha = 2\pi/T$, where
 $c_1 < \frac{17}{5} < 4$, $c_2 <\frac{2}{9} < \frac{1}{4}$.
 Therefore,
\begin{equation}\label{eq:raman}\begin{aligned}
 x^\sigma S_\sigma(x) &= 
\frac{1}{i T} \int_{1-i T}^{1+i T} \varphi\left(\frac{s-1}{i T}\right)
	A(s) x^s ds +O^*\left( \frac{2\pi}{T} \frac{\tanh(\lambda/4)}{\lambda}\right)\cdot
    a_\infty x\\
&+   \frac{a_\infty}{2}\cdot O^*\left(\frac{x}{4 T |y_0|}
+  c_2 \frac{e^{-\alpha y_0}}{\alpha y_0^2} + c_1
\right).
\end{aligned}\end{equation}

Letting $y_0 = -\rho/\alpha = -\rho T/2\pi$, we see that the last
expression within parentheses equals
$(\alpha x/4 T)/\rho + c_2 \alpha e^{\rho}/\rho^2 + c_1$, and so we should
make $K/\rho + e^\rho/\rho^2$ small for $K = x/ 4 c_2 T$. 
We choose $\rho = \log K$; then 
$K/\rho + e^\rho/\rho^2 = K/\log K + K/\log^2 K$, which is close to minimal.
We verify that condition $y_0\leq -T/\pi$ in Prop.~\ref{prop:summsec2} is fulfilled:
since $c_2<1/4$, $K>x/T\geq e^2$ and so $\rho=\log K>2$, meaning that $y_0 = -\rho T/2\pi < -T/\pi$.
We go on, estimating the last line of \eqref{eq:raman}:
$$\frac{a_\infty}{2}\cdot\left(\frac{x}{4 T |y_0|}
+  c_2 \frac{e^{-\alpha y_0}}{\alpha y_0^2}\right)\leq
\frac{a_\infty}{2} \cdot \frac{c_2 \alpha x}{4 c_2 T} \left(\frac{1}{\log K}
+ \frac{1}{\log^2 K}\right) = 
\frac{a_\infty \pi x}{4 T^2 \log \frac{x}{4 c_2 T}}
\left(1 
+ \frac{1}{\log \frac{x}{4 c_2 T}}\right).
$$


Since $0\leq \tanh(x)/x\leq 1$,
$\frac{2\pi}{T} \frac{\tanh(\lambda/4)}{\lambda} =
 \frac{\tanh(\lambda/4)}{\sigma-1}$.
By \eqref{eq:sonnenblum}, the function $\varphi$ in the integral in
\eqref{eq:raman} can be replaced by $\Phi=\Phi_\lambda$. Then we apply 
Proposition \ref{prop:shiftbounded} with $F(s) = A(s)$:
$$\begin{aligned}\frac{1}{i T} \int_{1-i T}^{1+i T} \varphi\left(\frac{s-1}{i T}\right)
	A(s) x^s ds &= \frac{2\pi}{T}
\sum_{
\substack{\text{$\rho\in R$ a pole of $A(s)$}\\ \text{or $\rho = 1 + \frac{\lambda T}{2\pi}$ and $\lambda<0$}}} \Res\limits_{s=\rho} \left(\Phi\left(\frac{s-1}{i T}\right) A(s) x^s \right) \\ &+ 
\frac{\pi}{8 T^2} \cdot O^*\left( \sum_{\xi=\pm 1} \int_0^\infty t |A(1 - t + i \xi  T)| x^{1-t} dt\right),
\end{aligned}
$$
where $R = (-\infty,1] + i [-T,T]$. Since $a_n$ is bounded and $A(s)$ is bounded on
$1 + i [-T,T]$, all poles of $A(s)$ with $\Im s\leq T$ lie to the left of $\Re s = 1$, i.e., they in $R$.
Clearly $1 + \lambda T/2\pi = \sigma$.
By \eqref{eq:girasol3} and $\frac{\pi}{2 i} \frac{s-1}{i T} + \frac{\lambda}{4} = - \delta\cdot (s-1) + \delta (\sigma-1) = \delta (\sigma-s)$,
$$\Phi\left(\frac{s-1}{i T}\right) = 
\frac{\sgn(\lambda)}{4} \cdot \left(\coth((\sigma-s) \delta)
- \tanh \frac{\lambda}{4}\right).$$
We finish by flipping the sign of both factors on the right; here
 $\sgn(\lambda) = \sgn(\sigma-1)$.
\end{proof}
\begin{remark}
We can write the weight $w_{\delta,\sigma}(s) =
    \coth(\delta (s-\sigma)) - \tanh(\delta (1-\sigma))$ in the form
    $$w_{\delta,\sigma}(s) =
    \coth(\delta (s-\sigma)) - \coth(\delta (1+ i T -\sigma)),$$
since
$\tanh((\sigma-1) \delta) = 
\tanh\left((\sigma- (1+ i T)) \delta + \frac{i \pi}{2}\right) = \coth((\sigma-(1+i T)) \delta)$.
\end{remark}
\begin{proof}[Proof of Theorem \ref{thm:mainthmA}] 
  
{\bf Case $\sigma<1$.} We apply Proposition~\ref{prop:gendau} and are done.

{\bf Case $\sigma >1$.} We 
want
to estimate $\sum_{n\leq x} a_n n^{-\sigma} = \sum_n a_n n^{-\sigma} - \sum_{n>x} a_n n^{-\sigma} = A(\sigma) - S_\sigma(x^+)$, so
we apply Proposition~\ref{prop:gendau} with $x^+$ (that is, a sequence of reals tending to $x$ from above) instead of $x$. The sign $\sgn(\sigma-1) = -1$
gets flipped: $-S_\sigma(x^+) = - \delta \sgn(1-\sigma) \sum_\rho \Res \dotsc = 
\delta \sum_\rho \Res \dotsc$. The error terms, being error terms, are unaffected by the change in sign.
Since the residue of $\coth z$ at $z=0$ is $1$, the residue of $w_{\delta,\sigma}(s)$ at $s=\sigma$ is $1/\delta$. Hence, we include the term
$A(\sigma)$ simply by including $\rho = \sigma$ in our sum over poles $\rho$.
(Since $\sigma>1$, $\sigma$ is not in $\mathcal{Z}_A(T)$ already.)

There is a subtlety regarding convergence here: the sum $\sum_\rho$ is generally
an infinite sum, and we do not know a priori that the limit of that sum of residues as $x_n\to x^+$ equals the sum of the limits of the residues. 
Recall that $\sum_{\rho \in \mathcal{Z}_A(T) \cup \{\sigma\}}$ here means
$\lim_{m\to \infty} \sum_{\rho \in \mathcal{Z}_A(T) \cup \{\sigma\} : \Re \rho > \sigma_m}$, where $\sigma_m\to -\infty$ (monotonically, it may be assumed). The difference between the sums for two consecutive
values is then
\begin{equation}\label{eq:adorno} \Delta_m = \sum_{\rho \in \mathcal{Z}_A(T) \cup \{\sigma\} : \sigma_m < \Re \rho\leq \sigma_{m+1}}
\Res_{s=\rho} (w_{\delta,\sigma}(s) A(s) x^{s-1}).\end{equation}
This sum equals $\frac{1}{2\pi}$ times the integral on the contour
$$\sigma_m - i T \rightarrow \sigma_m + i T \rightarrow \sigma_{m+1} + i T
\rightarrow \sigma_{m+1}- i T \rightarrow \sigma_m - i T.$$
On that contour, $A(s) T^{s-1}$ is uniformly bounded, and so is
$w_{\delta,\sigma}(s) = \coth(\delta(s-\sigma))-\tanh(\delta (1-\sigma))$.
So, by $x_n>x>e^2 T$, we see that $\Delta_m$
decays exponentially on $\sigma_m$, uniformly on $n$.
Since 
$\sum_{\rho \in \mathcal{Z}_A(T) \cup \{\sigma\} : \Re \rho > \sigma_M}$ is the sum of
the terms \eqref{eq:adorno} for $m<M$, then, by dominated convergence, the limit as
$x_n\to x^+$ of the limit as $M\to \infty$ equals the limit as $M\to \infty$ of the limit as $x_n\to x^+$.
In other words, the limit as $x_n\to x^+$ of the sum $\sum_\rho$ in Prop.~\ref{prop:gendau} is just the
sum $\sum_\rho$ for $x$.

The integral $I$ for $x\to x^+$ converges to its value for $x$, also by dominated convergence.

{\bf Case $\sigma=1$.} Apply Proposition \ref{prop:gendau} with $\sigma\to 1^-$.
Again, we have a limit of an infinite sum, but the same argument 
works as in the case $\sigma>1$.

\end{proof}
\begin{remark}
The condition in Theorem \ref{thm:mainthmA} that $A(s) T^s$ be bounded implies in particular
that $A(s)$ has no pole with $\Re s = 1$. It would be straightforward, given our framework,
to allow such a pole. In that case, that pole and the pole at $s=\sigma$ would collapse
into one pole for $\sigma=1$, but that would pose no issues. Cf. the situation in the companion paper \cite{Nonnegart}.

We have decided not to include the case of $A(s)$ having a pole at $s=1$ for the sake of simplicity. Given $A(s)$ with a pole with residue $a$ at $s=1$, we can always apply
Theorem \ref{thm:mainthmA} to $A(s) - a \zeta(s)$, which has no pole at $s=1$. (Since $\{a_n\}$ is bounded, the pole has to be simple.) This is the natural choice, as, in this situation, 
$\{a_n\}$ must in a sense be centered around $a$ rather than around $0$.
\end{remark}
\subsection{Summing over the trivial zeros for $A(s)=1/\zeta(s)$}
We will be looking at our weight
\begin{equation}\label{eq:fatso}
w_{\delta,\sigma}(s) = 
\coth(\delta (s-\sigma)) - \tanh(\delta (1-\sigma))
\end{equation}
\begin{lemma}\label{lem:paulin}
Let $\delta>0$, $\sigma$ real. Let $w_{\delta,\sigma}$ be as in \eqref{eq:fatso}. Then,
for $t<\min(\sigma,1)$ real, $w_{\delta,\sigma}(t)$ is negative, and $|w_{\delta,\sigma}(t)|$ is increasing in $t$. Moreover,
$$|w_{\delta,\sigma}(t)|\leq \frac{1}{\delta(\sigma-t)} + 2.$$
\end{lemma}
We can prove more precise inequalities, but we have no need for them here.
\begin{proof}
Since $\tanh x$ is increasing in $x$, and $\coth x < \tanh x$ for $x<0$,
$w_{\delta,\sigma}(t) <\coth(\delta(t-\sigma)) -\tanh(\delta(t-\sigma))<0$ for 
$t<\min(\sigma,1)$.
By $\coth'(x)<0$ for $x\ne 0$,
$w_{\delta,\sigma}(t)$ is decreasing for $t<\sigma$.

 Since $(\coth x - 1/x)' = 1/x^2 - 1/\sinh^2 x > 0$ for all $x>0$ and $\coth x - 1/x \to 1$ 
 as $x\to \infty$, we see that $\coth x - 1/x < 1$ for all $x>0$. 
 By $\tanh x > - 1$ for all real $x$,
$$-w_{\delta,\sigma}(t) = \coth(\delta(\sigma-t)) - \tanh(\delta(\sigma-1)) <
\frac{1}{\delta (\sigma-t)} + 2.$$
\end{proof}
\begin{lemma}\label{lem:pommedupe}
Let $\sigma>-2$ and $\delta>0$. 
Let $w_{\delta,\sigma}$ be as in \eqref{eq:fatso}. Then, for
$x\geq 2$,
$$ \delta  \sum_n \Res_{s= - 2 n} \frac{
   w_{\delta,\sigma}(s) x^{s-1}}{\zeta(s)} = 
   \left(\frac{1}{2+\sigma} + 2\delta\right) \cdot \frac{(2\pi)^2}{\zeta(3)} x^{-3}.$$
\end{lemma}
\begin{proof}
Since all trivial zeros of $\zeta$ are simple and $w_{\delta,\sigma}$ has no zeros to the left of $\sigma$, the residue of 
$\frac{w_{\delta,\sigma}(s) x^{s-1}}{\zeta(s)}$ at $s= -2 n$ equals
$\frac{w_{\delta,\sigma}(-2 n) x^{-2 n - 1}}{\zeta'(-2 n)}$.
By the functional equation \eqref{eq:funceq},
\begin{equation}\label{eq:koshmar}\frac{1}{\zeta'(-2 n)} = 
\frac{1}{(2\pi)^{s-1} \left(2 \sin\left(\frac{\pi s}{2}\right)\right)' \Gamma(1-s) \zeta(1-s)|_{s=-2n}} = 
\frac{(-1)^n(2 \pi)^{2n+1}}{ \pi \cdot(2 n)! \zeta(2 n + 1)}.
\end{equation}

We would like to show that we have an alternating sum, i.e., a sum of the form 
$S = \sum_n (-1)^n s_n$ with
$|s_n|$ decreasing and $s_n$ of constant sign, as then $|S|\leq |s_1|$.

By \eqref{eq:koshmar}, the ratio $|1/\zeta'(-2(n+1))|/|1/\zeta'(2n)|$ is 
$(2\pi)^2 \zeta(2n+1)/((2n+1) (2n+2) \zeta(2n+3))\leq 
(2\pi)^2 \zeta(3)/(3\cdot 4) = 3.95\dotsc < 4$.
We know from Lemma \ref{lem:paulin} that, for $s<\sigma$ real,
$w_{\delta,\sigma}(s)$ is real-valued and negative, and also that
$|w_{\delta,\sigma}(s)|$ is increasing in $s$; hence, 
$|w_{\delta,\sigma}(- 2 n)|$ is decreasing in $n$.
 Hence, for $x\geq 2$, our sum over trivial zeros is in fact an alternating sum, and so
$$\left|\sum_n\frac{w_{\delta,\sigma}(-2 n) x^{-2 n - 1}}{\zeta'(-2 n)}\right|
\leq \frac{|w_{\delta,\sigma}(-2)| x^{-3}}{\zeta'(-2)}
\leq \left(\frac{1}{\delta (\sigma+2)} + 2\right)
\frac{(2 \pi)^3}{ \pi \cdot 2 \zeta(3)} x^{-3}
$$ 
by Lemma \ref{lem:paulin} and \eqref{eq:koshmar}.
\end{proof}

\subsection{Results for $a_n = \mu(n)$ and $\sigma$ arbitrary}\label{subs:mainmu}
\begin{corollary}\label{cor:jolene}
Let $T\geq 4\pi$, $\sigma>-2$. Assume that all zeros of $\zeta(s)$ with $|\Im s|< T$ are simple, 
and that $\zeta(s)$ has no zeros with $|\Im s| = T$. Then, for $x\geq e^2 T$,
\[\sum_{n\leq x} \frac{\mu(n)}{n^\sigma} = 
(O^*(\delta) +  \varepsilon(x,T)) x^{1-\sigma} +  \delta
   \sum_{\rho\in \mathcal{Z}_*(T)} \frac{w_{\delta,\sigma}(\rho)}{\zeta'(\rho)} x^{\rho-\sigma} + 
   \frac{1}{\zeta(\sigma)} +  \varepsilon_-(\delta,\sigma) x^{-2-\sigma}
   ,\]
   where  $\delta=\pi/2 T$, $\mathcal{Z}_*(T)$ is the set of non-trivial zeros of $\zeta(s)$ with 
   $|\Im s| \leq T$, $w_{\delta,\sigma}(s)$  is as in \eqref{eq:fatso}, 
      $
   \varepsilon_-(\delta,\sigma) = \left(\frac{1}{2+\sigma} + 2\delta\right) \cdot \frac{(2\pi)^2}{\zeta(3)}$, and $\varepsilon(x,T)$ is as in 
   Theorem \ref{thm:mainthmA} with $A(s) = 1/\zeta(s)$ and $a_\infty=1$.
\end{corollary}
If $\sigma=1$, it is understood that $1/\zeta(\sigma) = 0$.
\begin{proof}
We apply Theorem \ref{thm:mainthmA} and Lemma \ref{lem:pommedupe}.
There is just one condition to check: we know that $A(s)$ is bounded on $\partial R \cup L$ for
$R = (-\infty,1] + i [-T,T]$ and 
$L = \cup_{n=1}^\infty ((-2n+1) + i [-T,T])$ because of 
Lemma \ref{lem:zetinvbound}
and the assumption that $\zeta(s)$ has no zeros in $[0,1]\pm i T$.
\end{proof}

\begin{proof}[Proof of Corollary \ref{cor:mertens}]
Apply Cor.~\ref{cor:jolene}. We recall that
\begin{equation}\label{eq:dadumo}
\left|\varepsilon(x,T)\right| \leq \frac{\pi}{4} \frac{L^{-1} + L^{-2} + I}{T^2} + \frac{2}{x}.\end{equation}
Let $\mathbf{c}= \max_{\xi=\pm 1} \max_{\sigma\leq 1} 1/|\zeta(\sigma + i \xi T)|$. By integration by parts,
\begin{equation}\label{eq:patatan}I \leq \mathbf{c}\int_0^\infty t x^{-t} dt = \mathbf{c}\int_0^\infty \frac{x^{-t}}{\log x} dt
= \frac{\mathbf{c}}{(\log x)^2}.\end{equation}
We are assuming $\mathbf{c}\leq \log^2 x$, and so $I \leq 1$. 
Since we also assume $x\geq e^2 T$, we know that $L\geq 2$, and so
$L^{-1} + L^{-2} + I \leq \frac{7}{4}$.
By $T\geq 4\pi$, $x\geq e^2 T$ and $\sigma\geq -1$,
$$\varepsilon_-(\delta,\sigma) x^{-2} \leq 
\left(\frac{1}{2+(-1)} + 2 \cdot \left(\frac{\pi}{2\cdot 4\pi}\right)\right) \frac{(2\pi)^2}{\zeta(3)} \cdot (4 e^2\pi)^{-1} x^{-1} < \frac{0.45}{x} \leq
\frac{0.45}{e^4} \frac{x}{T^2}.$$
By $7/4\cdot \pi/4 +0.45/e^4<\pi/2$, we conclude that the first term from \eqref{eq:dadumo} 
and $\varepsilon_-(\delta,\sigma) x^{-3}$ add up to less than $\frac{\pi}{2 T^2}$. Since
$\frac{\pi}{2 T} + \frac{\pi}{2 T^2} < \frac{\pi}{2 (T-1)}$, we are done.
\end{proof}
\subsection{Computational inputs}\label{subs:compinp}

\begin{lemma}\label{lem:chamlu}
For $-\infty<\sigma\leq 1$, 
\begin{equation}\label{eq:ravech}
\frac{1}{|\zeta(\sigma + i T)|}\leq \begin{cases}
0.894198297&\text{for $T=10^6+1$,}\\
0.591342108 &\text{for $T=1893193.5$,}\\
0.551087906 &\text{for $T=10^{7}+1$,}\\
0.579764046 &\text{for $T=10^{8}+1$,}\\
0.669256578 &\text{for $T=10^{9}$,}\\
0.536165863 &\text{for $T=10^{10}+1$.}
\end{cases}\end{equation}
\end{lemma}
\begin{remark}

\noindent \begin{enumerate}[(a)]
\item Why list $T=1893193.5$? There are $3500000$ zeros of $\zeta(s)$ with $0<\Im(\rho)<1893193.5$; they were computed in \cite{zbMATH03304440}. We will need that value of $T$ for the discussion in \S \ref{subs:Mxbounds}.
\item Why $T=10^9$? Simply because the computation in Lemma \ref{lem:platt109} was done for $T=10^9$ and not for $T=10^9+1$. The maximum of $1/|\zeta(\sigma+i T)|$ for $T=10^9+1$ is in fact fairly large ($4.1475\dotsc$) due to a zero of $\zeta(s)$ nearby, but that would not have represented an obstacle for our results; the condition on $1/|\zeta(\sigma+i T)|$ in Corollaries~\ref{cor:mertens} and \ref{cor:metamert} is rather
relaxed.
\end{enumerate}
\end{remark}
\begin{proof}
For $\sigma\in [-1/64,1]$, we carry out a rigorous computation by means of the bisection method, implementing it in interval and ball arithmetic (MPFI \cite{revol2005motivations}, FLINT/Arb \cite{7891956}) within SageMath 10.7.
The case $T=10^{10}+1$ takes a few days on a single core on a laptop or an old server.

Here are the details. We first subdivide the interval $[-1/64,1]$ into intervals of length $2^{-16}$ or $2^{-18}$, say.
Then we eliminate intervals in which the minimum\footnote{We consider the minimum of $|\zeta(\sigma+i T)|$ rather than the maximum of $1/|\zeta(\sigma+ i T)|$ to avoid division by zero.} of $|\zeta(\sigma+ i T)|$ cannot lie, that is, intervals where the lower bound given by ball arithmetic exceeds the current minimum of the upper bounds on all other intervals being considered.
We bisect and eliminate for a few iterations, so that we have intervals of length $2^{-22}$, say.

Lastly, we bisect
for $23$ iterations more (say) to find roots of $\Re \zeta'(s)/\zeta(s)$, which, of course, equals $(\log |\zeta(s)|)'$. (If we worked with intervals of length much greater than $2^{-22}$, the error interval of
$\zeta'(s)$ for $T\sim 10^{10}$ would be too large.)
At each step, we discard all intervals where
$\Re \zeta'(s)/\zeta(s)$ cannot vanish. (That may be all of them, 
in which case the minimum of $|\zeta(\sigma+iT)|$ over $[-1/64,1]$ must lie at  an endpoint, that is, $-1/64$ or $1$.) In the end,
we return the maximum of $1/|\zeta(\sigma+ i T)|$ over the remaining intervals.
We obtain the bounds in \eqref{eq:ravech} as bounds for $\sigma\in [-1/64,1]$.

For $\sigma < -1/64$ and $T\geq 10^6$, Lemma \ref{lem:zetinvbound} gives us
$$\frac{1}{|\zeta(\sigma + i T)|}\leq 
\left(\frac{2\pi e}{T}\right)^{\frac{1}{2} + \frac{1}{64}} 
    \frac{\sqrt{e}}{|\zeta(1-(\sigma+i T))|}\leq 0.00574\cdot \zeta\left(1 + \frac{1}{64}\right) < 0.371,
    $$
    since, for $\Re s \geq \sigma_0>1$, $|1/\zeta(s)| = |\sum_n \mu(n) n^{-s}|\leq
    \sum_n n^{-\Re s}= \zeta(\Re s) \leq \zeta(\sigma_0)$.
\end{proof}

All non-trivial zeros
    of $\zeta(s)$ with $|\Im s|\leq 3\cdot 10^{12}$ are simple and
    satisfy $\Re s = 1/2$ \cite{zbMATH07381909}.
    \begin{lemma}\label{lem:platt109}
    Write $\delta = \pi/2 T$. Then  
   \begin{equation*}
   \delta
\sum_{\substack{\rho:\; \zeta(\rho)=0 \\0<\Im(\rho)<T}}
\frac{\left|\coth (\delta \rho)\right|}{|\zeta'(\rho)|}
=  \begin{cases}
 2.66161277991001\dotsc&\text{for $T=10^6+1$,}\\
 2.85417779533422\dotsc&\text{for $T=1893193.5$,}\\
 3.36904620179490\dotsc &\text{for $T=10^{7}+1$,}\\
 4.10963816503581\dotsc&\text{for $T=10^{8}+1$,}\\
 4.87936778767100\dotsc&\text{for $T=10^{9}$,}\\
    5.675256\dotsc &\text{for $T=10^{10}+1$.}
\end{cases}\end{equation*}
\end{lemma}
\begin{proof}
This is a rigorous computation due to D. Platt, using interval arithmetic. 
Platt computed
$$\sum_{\substack{\zeta(\rho)=0\\ 0 < \Im \rho < T}}
\left|\frac{f\left(\frac{\rho-1}{i T}\right)}{\rho \zeta'(\rho)}\right|$$
for $f(z) = \frac{\pi z}{2} \cot \frac{\pi z}{2}$. Since $\Re \rho = \frac{1}{2}$ for every $\rho$ in this sum, we have
$\rho-1 = - \overline{\rho}$, and so 
$\coth(\delta(\rho-1)) = \coth(-\delta \overline{\rho}) = 
- \overline{\coth(\delta \rho)}$.
Thus, Platt's sum is the same as ours:
$$\left|\frac{f\left(\frac{\rho-1}{i T}\right)}{\rho \zeta'(\rho)}\right|
= 
\frac{\pi}{2 T} 
\frac{|\rho-1|}{|\rho \zeta'(\rho)|} \left|\cot\left(\frac{\pi}{2 T} \frac{\rho-1}{i}\right)
\right| = \delta \frac{|\coth(\delta(\rho-1))|}{|\zeta'(\rho)|} = 
\delta \frac{|\coth(\delta \rho)|}{|\zeta'(\rho)|}.
$$
\end{proof}

A few words are perhaps in order on how all the zeros $\rho$ of  $\zeta(s)$ with
$|\Im \rho|\leq T$ can be found rigorously given $T$. First of all: given a function $f$ (such as $\zeta$) and a real interval $I$, or for that matter a complex ball or rectangle $I$,
interval and ball arithmetic return a ball (or interval, or rectangle) in which $f(I)$ is guaranteed to be contained. This is part of what is meant by ``rigorous computation''.

The Riemann--Siegel $Z$-function $Z(t)$ is a real-valued function on $\mathbb{R}$ that
is a multiple of $\zeta(1/2+i t)$.
Every sign change of $Z(t)$ gives us a zero on the critical line $\Re s = 1/2$. The question is whether we are missing any zeros.
The classical formula for $N(T)$ has an error term of $O(\log T)$; if the error term were
$<1/2$, we could round, and obtain the exact value of $N(T)$. What is used in practice
(a trick due to Turing) is strong average bounds on the error term for $N(T)$ -- strong 
enough that they would detect even a single missing zero. See \cite{zbMATH05251036} for an exposition.

\begin{lemma}\label{lem:hurst}
For $0<x< 33$, $|M(x)|\leq \sqrt{x}$. For $33\leq x\leq 10^{16}$,    
\begin{equation}\label{eq:boundMco}|M(x)|\leq 0.570591 \sqrt{x}.
\end{equation}
For $0<x<3$, $|m(x)|\leq \sqrt{2/x}$. For $3\leq x\leq 10^{14}$,
\begin{equation}\label{eq:boundmco}
|m(x)|\leq 0.569449/\sqrt{x}.\end{equation}
\end{lemma}
\begin{proof}
    These are computational (``brute-force'') bounds. Both \eqref{eq:boundMco} and \eqref{eq:boundmco} were
    obtained by segmented sieves --
    \cite[\S 2, first table in \S 6.1]{Hurst} and \cite[Lemma 5.10]{Helfbook} (code available upon request),
    respectively, with similar optimizations; 
    \cite[Table 1]{zbMATH02150494} had given \eqref{eq:boundMco} up to
    $10^{14}$.
\end{proof}
\subsection{Clean bounds}\label{subs:cleanmu}
By $\coth(x)-\tanh(y) = \frac{\cosh(x-y)}{\sinh x \cosh y}$, we can rewrite \eqref{eq:fatso}
as
\begin{equation}\label{eq:nipaul}w_{\delta,\sigma}(s) = \frac{1}{\cosh(\delta (1-\sigma))} \cdot\frac{\cosh(\delta  (s-1))}{\sinh(\delta(s-\sigma))}.\end{equation}

 \begin{corollary}\label{cor:mertensimplon}
 Let $M(x) = \sum_{n\leq x} \mu(n)$ and $m(x) = \sum_{n\leq x} \mu(n)/n$.
Then, for $x\geq 1$,
   \begin{equation}\label{eq:garance2}|M(x)|\leq \frac{\pi}{2\cdot 10^{10}}\cdot x + C_1 \sqrt{x},\;\;\;\;\;\;\;\;\;\;
   |m(x)|\leq \frac{\pi}{2\cdot 10^{10}} + \frac{C_1}{\sqrt{x}}
   \end{equation}
   with $C_1 = 11.350514$, 
   \begin{equation}\label{eq:garadiol}\,\,\,\,\,\,\,\,\,\,\,\,\,\,\,|M(x)|\leq \frac{\pi}{2\cdot (10^{9}-1)}\cdot x + C_2 \sqrt{x},\;\;\;\;\;\;\;\;\;\;
   |m(x)|\leq \frac{\pi}{2\cdot (10^{9}-1)} + \frac{C_2}{\sqrt{x}}
   \end{equation}
   with $C_2 = 9.758736$, and
   \begin{equation}\label{eq:garancedos}|M(x)|\leq \frac{\pi}{2\cdot 10^{7}} \cdot x + C_3 \sqrt{x},\;\;\;\;\;\;\;\;\;\;
   |m(x)|\leq \frac{\pi}{2\cdot 10^{7}} + \frac{C_3}{\sqrt{x}}
   \end{equation}
with $C_3 = 6.738093$.   
\end{corollary}
\begin{proof} Let $\mathcal{Z}_*(T)$ be as in Cor.~\ref{cor:mertens}, viz., the set of non-trivial zeros $\rho$
with $|\Im \rho|\leq T = 10^{10}+1$.
Since there are no real zeros of $\zeta(s)$ with $0\leq s\leq 1$, this is the same as the set of
zeros $\rho$ with $0< |\Im \rho|\leq T$.
Let $w_{\delta,\sigma}$ be as in 
\eqref{eq:fatso} or \eqref{eq:nipaul}. Then, by \eqref{eq:nipaul},
for $s$ with $\Re s = 1/2$,
$$|w_{\delta,1}(s)| = |\coth(\delta (s-1))| = |\coth(-\delta \overline{s})| = |\coth(\delta s)|,$$
$$|w_{\delta,0}(s)|
\leq \frac{|\cosh(\delta (s-1))|}{|\sinh(\delta s)|}
= \frac{|\cosh (- \delta \overline{s})|}{|\sinh(\delta s)|} = |\coth(\delta s)|.$$
Hence, by Lemma \ref{lem:platt109}, for $\sigma = 0,1$,

$$\delta
   \left|\sum_{\rho\in \mathcal{Z}_*(T)} \frac{w_{\delta,\sigma}(\rho)}{\zeta'(\rho)} x^{\rho-\sigma}\right| \leq x^{\frac{1}{2}-\sigma} \cdot
   \delta 
\sum_{\rho\in \mathcal{Z}_*(T)} \frac{|\coth(\delta \rho)|}{|\zeta'(\rho)|} \leq
x^{\frac{1}{2}-\sigma}\cdot
2 \cdot 5.6752568. 
$$
We now apply Cor.~\ref{cor:mertens} for $x\geq e^2 T$; the condition  
$\max_{r\leq 1} 1/|\zeta(r\pm i T)|\leq \log^2 x$ there is fulfilled thanks to 
Lemma~\ref{lem:chamlu}.
We know that 
$\frac{1}{\zeta(1)}= 0$ (by convention) and $\frac{1}{\zeta(0)} = -2$.
We conclude that
$$|M(x)|\leq \frac{\pi}{2\cdot 10^{10}} x + C_0\sqrt{x} + 4,\;\;\;\;\;\;\;\;|m(x)|\leq \frac{\pi}{2\cdot 10^{10}} + \frac{C_0}{\sqrt{x}} + \frac{2}{x}$$
for $C_0 = 2 \cdot 5.6752568 = 11.3505136$.
Now we will just absorb the terms $2/x$ and $4$.

Lemma \ref{lem:hurst} gives us 
$|M(x)|\leq \sqrt{x} < C_0 \sqrt{x}$ for $1\leq x\leq 10^{16}$ and
$|m(x)|\leq \sqrt{2/x} < C_0/\sqrt{x}$ for
$1\leq x\leq 10^{14}$.  For $x>10^{16}$, 
$C_0 \sqrt{x} +4\leq (C_0 + 4\cdot 10^{-8}) \sqrt{x}\leq 11.350514 \sqrt{x}$; for $x>10^{14}$, 
$C_0/\sqrt{x}+2/x\leq (C_0 + 2\cdot 10^{-7})/\sqrt{x} = 11.350514/\sqrt{x}$.
Thus, \eqref{eq:garance2} holds. 

We can prove \eqref{eq:garadiol} and \eqref{eq:garancedos} in exactly the same way, using
Lemmas \ref{lem:chamlu} and \ref{lem:platt109} with $T=10^9$ and $T=10^7+1$.
\end{proof}

\section{Square-free numbers}\label{sec:mainsqfr}
We will now show how to prove estimates for the number $Q(x)$ of square-free integers $1\leq n\leq x$. Later, in \S \ref{sec:sqsupp}, we will show how to use them to improve our own bounds on $M(x)$.

What we must bound here is the difference $R(x) = Q(x) - x/\zeta(2)$.
It has been known since at least \cite{Moser_MacLeod_1966} that bounds on $M(x)$ help bound $R(x)$; a general procedure for proving bounds on $R(x)$ using bounds on $M(x)$ can be found in \cite[\S 2]{zbMATH05257415}. We will, however, prefer to give our own derivation.

\subsection{From $M(x)$ to $R(x)$: basic setup}\label{subs:chriyo}
Our approach is based on the following lemma. It expresses $R(x)$ as the difference between an integral and its discrete approximation.

\begin{lemma}\label{lem:difintsq}
For any $x>0$,
\begin{equation}\label{eq:antenor}
R(x) = \sum_{k\leq x} M\left(\sqrt{\frac{x}{k}}\right) - \int_0^x M\left(\sqrt{\frac{x}{u}}\right) du.\end{equation}
Moreover, for any integer $K\geq 0$,
\begin{equation}\label{eq:singdot}\begin{aligned}
R(x) &= \sum_{k\leq K} M\left(\sqrt{\frac{x}{k}}\right)  - 
\int_0^{K+\frac{1}{2}} M\left(\sqrt{\frac{x}{u}}\right) du\\
&-\sum_{K<k\leq x+1} \int_{k-\frac{1}{2}}^{k+\frac{1}{2}} \left(M\left(\sqrt{\frac{x}{u}}\right) -
M\left(\sqrt{\frac{x}{k}}\right)\right) du
\end{aligned}
\end{equation}
\end{lemma}
Since our sums start from $1$, the sum $\sum_{k\leq K}$ is empty for $K=0$.
\begin{proof}
The first equality is immediate from
$$Q(x) = \sum_{n\leq x} \sum_{d: d^2|n} \mu(d) = \sum_{k, d: k d^2\leq x} \mu(d) = \sum_{k\leq x} M\left(\sqrt{\frac{x}{k}}\right)$$
and (by absolute convergence)
$$\int_0^x M\left(\sqrt{\frac{x}{u}}\right) du = \int_0^x \sum_{n\leq \sqrt{\frac{x}{u}}} \mu(n) du
=\sum_n \mu(n) \int_0^{\frac{x}{n^2}} du = x \sum_n \frac{\mu(n)}{n^2} = \frac{x}{\zeta(2)}.$$
In the second equality, it is clear that the terms of the form $M(\sqrt{x/k})$ can all be
grouped as $\sum_{k\leq x+1} M(\sqrt{x/k})$; the term $M(\sqrt{x/k})$ for $k = \lfloor x+1\rfloor > x$ equals $0$. For $f(u) = M(\sqrt{x/u})$,
\[\sum_{K<k\leq x+1} \int_{k-\frac{1}{2}}^{k+\frac{1}{2}} f(u) du = \int_{K+\frac{1}{2}}^{\lfloor x\rfloor + \frac{3}{2}} f(u) du 
= \int_{K+\frac{1}{2}}^x f(u) du,\]
again because $M(t) = 0$ for $t<1$. Here we are assuming that
$K\leq x$. If $K>x$, the second line of \eqref{eq:singdot}
is empty, and the first one equals \eqref{eq:antenor}, by
$M(t)=0$ for $t<1$, so \eqref{eq:singdot} holds.
\end{proof}

\begin{lemma}\label{lem:sugr} For $x>0$ and $k\geq 1$,
$$\left|\int_{k-\frac{1}{2}}^{k+\frac{1}{2}} \left(M\left(\sqrt{\frac{x}{u}}\right) 
-M\left(\sqrt{\frac{x}{k}}\right)\right) du\right|\leq
\frac{1}{2} \left|\left(\sqrt{\frac{x}{k+1/2}},\sqrt{\frac{x}{k-1/2}}\right)\cap \mathcal{Q}\right|,$$
where $\mathcal{Q}$ denotes the set of square-free integers, and
$|S|$ denotes the number of elements of a set $S$.
\end{lemma}
\begin{proof} 


Let $S_1 = \left(\sqrt{\frac{x}{k+\frac{1}{2}}},\sqrt{\frac{x}{k}}\right]\cap \mathcal{Q}$ and
$S_2 = \left(\sqrt{\frac{x}{k}},\sqrt{\frac{x}{k-\frac{1}{2}}}\right)\cap \mathcal{Q}$.
Since, for $u\leq k$, $M(\sqrt{x/u})-M(\sqrt{x/k}) = \sum_{\sqrt{x/k}<n\leq \sqrt{x/u}} \mu(n)$,
\[\begin{aligned}\left|\int_{k-\frac{1}{2}}^{k} \left(M\left(\sqrt{\frac{x}{u}}\right) 
-M\left(\sqrt{\frac{x}{k}}\right)\right) du\right| &= 
\left|\sum_{\sqrt{\frac{x}{k}}<n\leq \sqrt{\frac{x}{k-\frac{1}{2}}}} \mu(n) \int_{k-\frac{1}{2}}^{\frac{x}{n^2}} du\right|\leq \frac{\left|S_2\right|}{2}.\end{aligned}\]
Similarly,
for $u>k$, $M(\sqrt{x/u})-M(\sqrt{x/k}) = - \sum_{\sqrt{x/u}<n\leq \sqrt{x/k}} \mu(n)$, and so
\[\left|\int_k^{k+\frac{1}{2}} \left(M\left(\sqrt{\frac{x}{u}}\right) 
-M\left(\sqrt{\frac{x}{k}}\right)\right) du\right| = 
\left|\sum_{\sqrt{\frac{x}{k+\frac{1}{2}}}<n\leq \sqrt{\frac{x}{k}}} \mu(n) 
\int_{\frac{x}{n^2}}^{k+\frac{1}{2}} du\right|
 \leq \frac{\left|S_1\right|}{2}.\]
\end{proof}

\begin{lemma}\label{lem:trapeze} Assume $|Q(t_2)-Q(t_1)|\leq c_1 |t_2-t_1| + c_2$ for all $t_1,t_2>0$, where $c_1,c_2\geq 0$. Then, for any $x, k\geq 1$,
$$\int_{k-\frac{1}{2}}^{k+\frac{1}{2}} \left|M\left(\sqrt{\frac{x}{u}}\right) 
-M\left(\sqrt{\frac{x}{k}}\right)\right| du\leq \frac{c_1}{4} \sqrt{x}
\left(\frac{1}{\sqrt{k-1/2}}-\frac{1}{\sqrt{k+1/2}}\right)+ c_2.$$
\end{lemma}
\begin{proof} For $t>0$,
$\left|M(t)- M\left(\sqrt{\frac{x}{k}}\right)\right|\leq \left|Q(t)- Q\left(\sqrt{\frac{x}{k}}\right)\right|\leq c_1 \left|t - \sqrt{\frac{x}{k}}\right|+ c_2$. Then
\[\begin{aligned}&\int_{k-\frac{1}{2}}^{k+\frac{1}{2}} 
\left|\sqrt{\frac{x}{u}} - \sqrt{\frac{x}{k}}\right| du = 
\int_{k-\frac{1}{2}}^{k} 
\left(\sqrt{\frac{x}{u}} - \sqrt{\frac{x}{k}}\right) du +
\int_{k}^{k+\frac{1}{2}} 
\left(\sqrt{\frac{x}{k}} - \sqrt{\frac{x}{u}}\right) du\\
&=\sqrt{x} \int_{k-\frac{1}{2}}^{k} f(u) du \leq \frac{\sqrt{x}}{2} \cdot \frac{f(k)+f\left(k-\frac{1}{2}\right)}{2}
= \frac{\sqrt{x}}{4} \left(\frac{1}{\sqrt{k-1/2}}-\frac{1}{\sqrt{k+1/2}}\right)
\end{aligned}\]
for
$f(u)=\frac{1}{\sqrt{u}} - \frac{1}{\sqrt{u+1/2}}$,
because $f$ is convex.
\end{proof}

\begin{proposition}\label{prop:andalas}Assume $|Q(t_2)-Q(t_1)|\leq c_1 |t_2-t_1| + c_2$ for all $t_1,t_2>0$, where $c_1,c_2\geq 0$. 
  For any $x>0$, and any integers $K'\geq K\geq 0$,
\begin{equation}\label{eq:ompet}\begin{aligned}
R(x) &= \sum_{k\leq K} M\left(\sqrt{\frac{x}{k}}\right)  - 
\int_0^{K+\frac{1}{2}} M\left(\sqrt{\frac{x}{u}}\right) du\\
& + O^*\left(\frac{c_1 \sqrt{x}}{4}
\left(\frac{1}{\sqrt{K+1/2}} - \frac{1}{\sqrt{K'+1/2}}\right) + 
c_2 \cdot (K'-K) + \frac{1}{2} \left|Q\left(\sqrt{\frac{x}{K'+1/2}}\right)\right|\right).
\end{aligned}\end{equation}
\end{proposition}
\begin{proof}
Apply Lem.~\ref{lem:difintsq}. Split the sum $\sum_{K<k\leq x+1}$ into sums
$\Sigma_1 = \sum_{K<k\leq K'}$, $\Sigma_2=\sum_{K'<k\leq x+1}$.
Assume $K,K'\leq x$
We bound $\Sigma_1$ by Lemma \ref{lem:trapeze}, and $\Sigma_2$ by Lemma \ref{lem:sugr},
with both sums telescoping:
\[\begin{aligned}|\Sigma_1| &\leq \sum_{K<k\leq K'}
\left(\frac{c_1}{4} \sqrt{x}
\left(\frac{1}{\sqrt{k-1/2}}-\frac{1}{\sqrt{k+1/2}}\right)+ c_2\right)\\
&= \frac{c_1 \sqrt{x}}{4} \left(\frac{1}{\sqrt{K+1/2}} - \frac{1}{\sqrt{K'+1/2}}\right) + c_2\cdot (K'-K),
\end{aligned}\]
\[\begin{aligned}\left|\Sigma_2\right| &\leq \frac{1}{2} \sum_{K'<k\leq x+1}
\left|\left(\sqrt{\frac{x}{k+1/2}},\sqrt{\frac{x}{k-1/2}}\right)\cap \mathcal{Q}\right|
\leq \frac{1}{2} Q\left(\sqrt{\frac{x}{K'+1/2}}\right).\end{aligned}\]
\end{proof}
\subsection{Auxiliary lemmas}

The next bound is already in \cite{Moser_MacLeod_1966}; we include it because it follows from our approach. We use the slightest of inputs on cancellation in $M(x)$; \cite{Moser_MacLeod_1966} used \cite{zbMATH02670421}.
\begin{lemma}\label{lem:oldtrivial}
For any $x\geq 0$, $|R(x)|\leq \sqrt{x}$.
\end{lemma}
\begin{proof}
The statement is easy for $0\leq x<1$.
Apply Prop~\ref{prop:andalas} with $K=K'=0$ and $c_1=c_2=1$. Then, for all $x\geq 1$,
by the trivial bounds $|M(\sqrt{x/u})|\leq \sqrt{x/u}$, $|Q(\sqrt{2 x})|\leq \sqrt{2 x}$,
\begin{equation}\label{eq:golok}|R(x)|\leq \int_0^{\frac{1}{2}} \left|M\left(\sqrt{\frac{x}{u}}\right)\right| du  + \frac{1}{2}\,\big|Q\big(\sqrt{2 x}\big)\big|
\leq \left(\sqrt{2} + \frac{1}{\sqrt{2}}\right) \sqrt{x}.\end{equation}

The injection $n\mapsto 2 n$ takes the set $S_1$ of odd  $n\leq x/2$ onto the set $S_2$ of all $m\leq x$ such that $m\equiv 2\bmod 4$. Moreover, $S_1\cap S_2 = \emptyset$ and $\mu(2 n) = -\mu(n)$ for $n$ odd. Since $\mu(n)=0$ when $4|n$, 
\[|M(x)|\leq x - |\{n\leq x: 4|n\}| - 2 |S_1| \leq
x - \left\lfloor \frac{x}{4}\right\rfloor - 2  \left\lfloor \frac{x}{4}\right\rfloor
\leq \frac{x}{4} + 3.
\]
We apply the first inequality in \eqref{eq:golok} again, using the second inequality in 
\eqref{eq:golok} to bound $|R(\sqrt{2 x})|$:
\[\begin{aligned}|R(x)|&\leq \int_0^{\frac{1}{2}} \left(\frac{1}{4} \sqrt{\frac{x}{u}} + 3\right) du +
\frac{\sqrt{2 x}}{2 \zeta(2)} + \frac{1}{2} \left|R(\sqrt{2 x})\right|\\
&\leq \frac{\sqrt{2 x}}{4} 
 + \frac{\sqrt{2 x}}{2 \zeta(2)}
+ \frac{1}{2} \left(\sqrt{2} + \frac{1}{\sqrt{2}}\right) (2 x)^{\frac{1}{4}} + \frac{3}{2}
.\end{aligned}\]
This implies $|R(x)|\leq \sqrt{x}$ for $x\geq 2400$; we check $R(n)$ for $n\leq 2400$ by a short computation.
\end{proof}

\begin{lemma}\label{lem:ramando}
For every integer $N\geq 2$, 
\begin{equation}\label{eq:synthe}2\sqrt{N+\frac{1}{2}} +  \sum_{n\leq N} \frac{1}{\sqrt{n}} \leq 4 \sqrt{N-\frac{1}{2}}.\end{equation}
For every integer $N\geq 3$,
\begin{equation}\label{eq:synthedos}
\frac{4}{3} \left(N+\frac{1}{2}\right)^{\frac{3}{4}}
+ \sum_{n\leq N} \frac{1}{\sqrt[4]{n}} \leq \frac{8}{3} \left(N-\frac{1}{2}\right)^{\frac{3}{4}}
+ \frac{5}{7}.
\end{equation}
Moreover, for all integers $N\geq 37$, the term $\frac{5}{7}$ can be omitted from \eqref{eq:synthedos}.
\end{lemma}
\begin{proof}
By convexity, $\sum_{n\leq N} \frac{1}{\sqrt{n}}\leq 1 + \int_{\frac{3}{2}}^{N-\frac{1}{2}} \frac{dt}{\sqrt{t}} + \frac{1}{\sqrt{N}}= 1 + 2 \sqrt{N - \frac{1}{2}} - \sqrt{6} +\frac{1}{\sqrt{N}}$ and $\sum_{n\leq N} \frac{1}{\sqrt[4]{n}}\leq 1 + \int_{\frac{3}{2}}^{N - \frac{1}{2}} \frac{dt}{\sqrt[4]{t}} + \frac{1}{\sqrt[4]{N}} = 1 + \frac{4}{3} \left(N - \frac{1}{2}\right)^\frac{3}{4} - \frac{4}{3} \left(\frac{3}{2}\right)^\frac{3}{4}+ \frac{1}{\sqrt[4]{N}}$.
By concavity, $\sqrt{N+\frac{1}{2}}\leq \sqrt{N-\frac{1}{2}} + \frac{1}{2\sqrt{N-\frac{1}{2}}}$ and
$\left(N+\frac{1}{2}\right)^{\frac{3}{4}}\leq 
\left(N-\frac{1}{2}\right)^{\frac{3}{4}} + \frac{3}{4} \left(N-\frac{1}{2}\right)^{-\frac{1}{4}}$. Hence,  \eqref{eq:synthe} holds
when $\frac{1}{\sqrt{N}} + \frac{1}{\sqrt{N-1/2}} \leq \sqrt{6}-1$, \eqref{eq:synthedos} holds
when $\frac{1}{\sqrt[4]{N}} +  \left(N-\frac{1}{2}\right)^{-\frac{1}{4}}\leq \frac{5}{7} + \frac{4}{3} \left(\frac{3}{2}\right)^\frac{3}{4} -1$, and \eqref{eq:synthedos} holds without the term $\frac{5}{7}$ when $\frac{1}{\sqrt[4]{N}} + \left(N-\frac{1}{2}\right)^{-\frac{1}{4}}\leq \frac{4}{3} \left(\frac{3}{2}\right)^\frac{3}{4} -1$. 

Of these three inequalities, the first one holds when $N=3$, and hence, since its left side is decreasing in $N$, it holds for all $N\geq 3$; therefore,  \eqref{eq:synthe} holds for $N\geq 3$. In the same way, since the second and third inequalities hold for $N=4$ and $N=38$, respectively, and their left sides are decreasing in $N$, we know that \eqref{eq:synthedos}
holds with and without the term $\frac{5}{7}$ for $N\geq 4$ and $N\geq 38$, respectively.
We check \eqref{eq:synthe} for $N=2$ and \eqref{eq:synthedos} both for $N=3$ (with the term $\frac{5}{7}$) and for $N=37$ (without that term) by a short rigorous numerical check.
\end{proof}

\begin{corollary}\label{cor:legolas}
Let $F:(0,\infty)\to \mathbb{R}$ be such that $|F(v)|\leq \epsilon v + \kappa \sqrt{v}$ for all $v>0$, where
$\kappa,\epsilon\geq 0$. Then, for any $x>0$ and any integer $K\geq 3$,
\[\left|\sum_{k\leq K} F\left(\sqrt{\frac{x}{k}}\right) \right| +
\left|\int_0^{K+\frac{1}{2}} F\left(\sqrt{\frac{x}{u}}\right) du\right|\leq
4\sqrt{K-\frac{1}{2}} \cdot \epsilon \sqrt{x} +
\left(\frac{8}{3} \left(K-\frac{1}{2}\right)^\frac{3}{4} + \frac{5}{7}\right) \kappa x^\frac{1}{4}.
\]
The term $\frac{5}{7}$ can be removed if $K\geq 37$.
\end{corollary}
\begin{proof} We bound
$\left|\sum_{k\leq K} F\left(\sqrt{\frac{x}{k}}\right)\right|\leq \sum_{k\leq K} \left(\epsilon\sqrt{\frac{x}{k}} + \kappa 
\left(\frac{x}{k}\right)^{\frac{1}{4}}\right)$ and
\[\left|\int_{0}^{K+\frac{1}{2}} F\left(\sqrt{\frac{x}{u}}\right)du\right| \leq \int_{0}^{K+\frac{1}{2}} \left(
\frac{\epsilon\sqrt{x}}{\sqrt{u}} + \frac{\kappa x^{\frac{1}{4}}}{u^{\frac{1}{4}}}\right) du = 2\sqrt{K+\frac{1}{2}}\epsilon \sqrt{x} + \frac{4}{3} \left({K+\frac{1}{2}}\right)^{\frac{3}{4}} \kappa x^{\frac{1}{4}}.\]
Then we apply Lemma~\ref{lem:ramando} and are done.
\end{proof}

We will use a simple short-interval bound.
\begin{lemma}\label{lem:edesmo}
    For each of these pairs $(c_1,c_2)$,
    $|Q(t_2)-Q(t_1)|\leq c_1 |t_2-t_1| + c_2$ for all $t_1,t_2\geq 0$:
    \[(1,1),\;\left(\tfrac{3}{4},\tfrac{3}{2}\right),\;\left(\tfrac{2}{3},\tfrac
{8}{3}\right),\;\left(\tfrac{16}{25},\tfrac{114}{25}\right),\;\left(\tfrac{768}{1225},\tfrac{9458}{1225}\right),\;\left(\tfrac{18432}{29645},\tfrac{361192}{29645}\right),
    \left(\tfrac{442368}{715715},\tfrac{14328304}{715715}\right).\]
\end{lemma}
\begin{proof}
    Let $Q_q(t)$ be the number of $0<n\leq t$ such that $n$ has no factors $p^2$ for $p\leq q$.
    Trivially, \[|Q(t_2)-Q(t_1)|\leq |Q_q(t_2)-Q_q(t_1)|.\] As a function on $(0,\infty)$, $Q_q(t)$ is periodic of period $R = \prod_{p\leq q} p^2$. Moreover, $|Q_q(t_2)-Q_q(t_1)|\leq c_1 |t_2-t_1| + c_2$ holds in general
    if and only if it holds for $t_2$ integer and $t_1\to n^-$, $n$ an integer.
    It is enough, then, to find $\nu = \max_{0\leq n<R} (Q_q(n)-c_1)$,
    where $c_1 = \prod_{p\leq q} (1-1/p^2)$, since then $c_2 = 2\nu$ is valid
    (and in fact optimal).
Determining $\nu$  is a finite computation; it yields
$\nu = 3/2, 8/3, 114/25, 4729/1225, 180596/29645, 7164152/715715$ for $q=2,3,5,7,11,13$.
\end{proof}

\subsection{From $M(x)$ to $R(x)$: parameters and bounds}\label{subs:maldoror}
We start with two bounds on $R(x)$. In practice, the first one will be best for $x$ moderately large, the second one for $x$ very large.
\begin{proposition}\label{prop:gould}
Assume that $|M(v)|\leq \epsilon v + \kappa \sqrt{v}$ for
all $v>0$, where $\kappa,\epsilon \geq 0$, and
$|Q(t_2)-Q(t_1)|\leq c_1 |t_2-t_1| + c_2$ for all $t_1,t_2>0$, where
$0< c_1\leq 1$, $c_2> 0$. Write
$c_3 = \frac{1}{2}\big(\frac{1}{2\zeta(2)} - \frac{c_1}{4}\big)$.

\begin{enumerate}[(a)]
\item\label{it:rogg1} Assume $\kappa>0$.
  For $x>\max\left(64\cdot 10^5 (\kappa/c_1)^4,
  \frac{(c_2/c_3)^5}{(16\kappa/c_1)^6} \right)$,
\begin{equation}\label{eq:chase1}
|R(x)|\leq
\frac{4 \epsilon}{(16\kappa/c_1)^{\frac{2}{5}}} \cdot x^{\frac{3}{5}}  +  
 \dfrac{5\,c_1^{\frac{3}{5}}}{3}\left(\dfrac{\kappa}{2}\right)^{\frac{2}{5}}
 \cdot x^{\frac{2}{5}}
 +3 c_2^{\frac{1}{3}} c_3^{\frac{2}{3}} \cdot x^{\frac{1}{3}}
 + \frac{5\kappa}{7} \cdot x^{\frac{1}{4}}
 + \frac{1}{2} \left(\frac{c_2}{c_3}\right)^{\frac{1}{6}} \cdot x^{\frac{1}{6}}.
\end{equation}
If $x>\max\left(4.25\cdot 10^{12}(\kappa/c_1)^4,\frac{(c_2/c_3)^5}{(16\kappa/c_1)^6} \right)$, we may omit the term $\frac{5\kappa}{7}x^{\frac{1}{4}}$.

\item\label{it:rogg2} Assume $0<\epsilon<\frac{c_1}{584}$. For $x\geq \big(\frac{c_2}{c_3}\big)^2 \big(\frac{c_1}{16 \epsilon}\big)^3$,
\begin{equation}\label{eq:chase2}
|R(x)|\leq  2(\epsilon c_1)^{\frac{1}{2}}\cdot x^{\frac{1}{2}} + 3 c_2^{\frac{1}{3}} c_3^{\frac{2}{3}} \cdot x^{\frac{1}{3}} + \frac{\kappa}{3}\left(\dfrac{c_1}{ \epsilon}\right)^{\frac{3}{4}} \cdot x^{\frac{1}{4}}
+ \frac{1}{2} \left(\frac{c_2}{c_3}\right)^{\frac{1}{6}} \cdot x^{\frac{1}{6}}
.
\end{equation}

\end{enumerate}
\end{proposition}
  We will apply Proposition \ref{prop:andalas} and Corollary \ref{cor:legolas}. We must just choose $K$ and $K'$ and do some minor accounting. The matter is that, when we minimize a function $f(y)$, the best $y$ is generally not an integer. We shall deal with that issue in a very simple way.  
\begin{proof} 
  Let us first choose $K'$ so as to minimize
  terms involving $K'$ in \eqref{eq:ompet}, as that task is the same in cases
  \ref{it:rogg1} and \ref{it:rogg2}. The terms to consider here are
  \begin{equation}\label{eq:adaror}
    -\frac{\frac{c_1\sqrt{x}}{4}}{\sqrt{K'+1/2}} + c_2 K' +
  \frac{1}{2\zeta(2)} \sqrt{\frac{x}{K'+1/2}}.\end{equation}
  For $\alpha,\beta>0$, the function $\alpha \upsilon + \frac{\beta}{\sqrt{\upsilon}}$ reaches its minimum
  $3 \alpha^{\frac{1}{3}} \big(\frac{\beta}{2}\big)^{\frac{2}{3}}$ at a $\upsilon_0=\big(\frac{\beta}{2\alpha}\big)^{\frac{2}{3}}$. We let
  $\alpha = c_2$, $\beta = \left(\frac{1}{2\zeta(2)} - \frac{c_1}{4}\right)\sqrt{x}
  = 2c_3 \sqrt{x}$,
  and set $K' = \left\lceil \upsilon_0-\frac{1}{2}\right\rceil$; then $K'+\frac{1}{2}\geq \upsilon_0$, $K'< \upsilon_0+\frac{1}{2}$,
  and so the expression in \eqref{eq:adaror} is $\leq \alpha\left(\upsilon_0+\frac{1}{2}\right) +
  \frac{\beta}{\sqrt{\upsilon_0}} = \alpha \upsilon_0 + \frac{\beta}{\sqrt{\upsilon_0}} + \frac{\alpha}{2}  = 3 \alpha^{\frac{1}{3}} \left(\frac{\beta}{2}\right)^{\frac{2}{3}} + \frac{\alpha}{2} = 3 c_2^{\frac{1}{3}} c_3^{\frac{2}{3}} x^{\frac{1}{3}} + \frac{c_2}{2}$.
  
  


 For bound \ref{it:rogg1}, we note that for $a,b>0$, the function $a y^{\frac{3}{4}} + \frac{b}{\sqrt{y}}$ reaches its minimum 
 $\frac{5a}{2} \left(\frac{2 b}{3 a}\right)^{\frac{3}{5}}$
 at $y_0=\big(\frac{2b}{3a}\big)^{\frac{4}{5}}$. We let $a=\frac{8\kappa}{3} x^{\frac{1}{4}}$,
 $b=\frac{c_1}{4} \sqrt{x}$ and
 $K = \lceil y_0-\frac{1}{2}\rceil$.  
 Since $y_0 = \frac{x^{1/5}}{(16 \kappa/c_1)^{4/5}}>\frac{(2^{11} 5^5)^{1/5}}{16^{4/5}} =\frac{5}{2}$, we 
 know that  $K\geq 3$.
 Hence, we may apply Corollary~\ref{cor:legolas}.
 We should also check that $y_0\leq \upsilon_0$, since then we will know that
 $K'\geq K$, which is required by Prop.~\ref{prop:andalas}.
 Indeed, by assumption,
 \[\frac{\upsilon_0}{y_0} =  \left(\frac{16\kappa}{c_1}\right)^{\frac{4}{5}}
 c_3^{\frac{2}{3}}
  c_2^{-\frac{2}{3}} x^{\frac{1}{3}- \frac{1}{5}} =
   \left(\left(\frac{16\kappa}{c_1}\right)^6 c_3^{5} c_2^{-5} x\right)^{\frac{2}{15}}\geq 1.
\]
So, since
$K+\frac{1}{2}\geq y_0$ and $K-\frac{1}{2}<y_0$, Prop.~\ref{prop:andalas}, Corollary~\ref{cor:legolas}
  and Lemma \ref{lem:oldtrivial} give us
 \[\begin{aligned}|R(x)|&\leq
 4 \sqrt{y_0}\cdot \epsilon \sqrt{x} +  a y_0^{\frac{3}{4}}
 + \frac{5}{7}\kappa x^{\frac{1}{4}} + \frac{b}{\sqrt{y_0}} + \alpha
 \left(\upsilon_0+\frac{1}{2}-K\right) + \frac{\beta}{\sqrt{\upsilon_0}} + \frac{1}{2} 
 \sqrt[4]{\frac{x}{\upsilon_0}}\\
 &\leq 4 \epsilon \sqrt{y_0 x}  + \frac{5a}{2}\left(\frac{2 b}{3 a}\right)^{\frac{3}{5}} +
 3\alpha^{\frac{1}{3}} \left(\frac{\beta}{2}\right)^{\frac{2}{3}}
  + \frac{5}{7} \kappa x^{\frac{1}{4}}
  + \frac{1}{2}\sqrt[4]{\frac{x}{\upsilon_0}}
\\
 &= \frac{4 \epsilon}{(16\kappa/c_1)^{\frac{2}{5}}} x^{\frac{3}{5}}  +  
 \frac{5}{3\cdot 2^{\frac{2}{5}}} c_1^{\frac{3}{5}} \kappa^{\frac{2}{5}}
 x^{\frac{2}{5}}
 +3 c_2^{\frac{1}{3}} c_3^{\frac{2}{3}} x^{\frac{1}{3}}
 + \frac{5}{7} \kappa x^{\frac{1}{4}}
 + \frac{1}{2} \left(\frac{c_2}{c_3}\right)^{\frac{1}{6}} x^{\frac{1}{6}},
 \end{aligned}\]
where we omit $\frac{5}{7} \kappa x^{\frac{1}{4}}$ if 
$x>2^{11}\cdot 73^5 (\kappa/c_1)^4$, since then $y_0>\frac{73}{2}$, and
so $K\geq 37$. We are bounding $1/2-K$ from above by $0$, that is, we are dropping a negative term of order $x^{\frac{1}{5}}$.

To prove bound \ref{it:rogg2},  we use the fact that, for $a,b>0$, $a y + b/y$ reaches its minimum $2\sqrt{a b}$ at 
$y_0 = \sqrt{b/a}$. We let $a = 4 \epsilon$, $b = \frac{c_1}{4}$ and $K = \lceil y_0^2-\frac{1}{2}\rceil$, so that
$\sqrt{K-\frac{1}{2}}<y_0$ and $\sqrt{K+\frac{1}{2}}\geq y_0$. Since $\epsilon < c_1/584$, we know 
$y_0^2 = c_1/16\epsilon>73/2$, and so $K\geq 37$.
We also check that $\upsilon_0/y_0^2
= \big(\frac{c_3}{c_2}\big)^{\frac{2}{3}} \frac{16\epsilon}{c_1} x^{\frac{1}{3}}\geq 1$, by
$x\geq \big(\frac{c_2}{c_3}\big)^2 \big(\frac{c_1}{16 \epsilon}\big)^3$.
Hence, Prop.~\ref{prop:andalas}, Corollary~\ref{cor:legolas} and Lemma \ref{lem:oldtrivial}  yield
 \[\begin{aligned}|R(x)|&\leq 
 2 \sqrt{a b x} + 
 3\alpha^{\frac{1}{3}} \left(\frac{\beta}{2}\right)^{\frac{2}{3}} +
  \frac{8}{3} y_0^\frac{3}{2}\cdot  \kappa x^{\frac{1}{4}} 
 + \frac{1}{2}\sqrt[4]{\frac{x}{\upsilon_0}} + \alpha (\tfrac{1}{2} - K)\\
&\leq 2 \sqrt{\epsilon c_1}x^{\frac{1}{2}} + 3 c_2^{\frac{1}{3}} c_3^{\frac{2}{3}} x^{\frac{1}{3}} + \frac{\kappa c_1^{\frac{3}{4}}}{3 \epsilon^{\frac{3}{4}}} x^{\frac{1}{4}}
+ \frac{1}{2} \left(\frac{c_2}{c_3}\right)^{\frac{1}{6}} x^{\frac{1}{6}}.
 \end{aligned}\]

\end{proof}

\begin{proposition}\label{prop:qmediano}
Assume 
$|M(v)|\leq \kappa_- \sqrt{v}$ for $v_0\leq v\leq v_1$, and
$|M(v)|\leq \epsilon v + \kappa \sqrt{v}$ for $v>v_1$, where $\kappa_->0$, 
$\epsilon,\kappa \geq 0$, $v_1>v_0>0$. Assume as well that
$|Q(t_2)-Q(t_1)|\leq c_1 |t_2-t_1| + c_2$ for all $t_1,t_2>0$, where
$0< c_1\leq 1$, $c_2> 0$. Write
$c_3 = \frac{1}{2}\big(\frac{1}{2\zeta(2)} - \frac{c_1}{4}\big)$. Then,
 for $\max\left(64\cdot 10^5 (\kappa_-/c_1)^4,
  \frac{(c_2/c_3)^5}{(16\kappa_-/c_1)^6}, \frac{(2 v_0^2)^{\frac{5}{4}}}{16 \kappa_-/c_1}, 2 v_0^2 \right)<x\leq v_1^2$,
\begin{equation}\label{eq:chase3}
|R(x)|\leq  2 \epsilon\cdot x^{\frac{1}{2}}  + 
 \dfrac{5\,c_1^{\frac{3}{5}}}{3}\left(\dfrac{\kappa_-}{2}\right)^{\frac{2}{5}}
 \cdot x^{\frac{2}{5}}
 +3 c_2^{\frac{1}{3}} c_3^{\frac{2}{3}} \cdot x^{\frac{1}{3}} +
 \frac{4}{3} \kappa x^{\frac{1}{4}}
 + \frac{1}{2} \left(\frac{c_2}{c_3}\right)^{\frac{1}{6}} \cdot x^{\frac{1}{6}}.
\end{equation}
\end{proposition}
\begin{proof}
We proceed as in the proof of case \ref{it:rogg1} of Prop.~\ref{prop:gould}, setting $K$ just as there,
only with $\kappa_-$ instead of $\kappa$. Since $\sqrt{x}\leq v_1$, 
$|M(v)|\leq \kappa_- \sqrt{v}$ holds for $v\leq \sqrt{x}$, and we use
$|M(v)|\leq \epsilon v + \kappa \sqrt{v}$ only to bound the tail
$\int_0^1 M(\sqrt{x/u}) du$. We subtract the old tail bound
$\int_0^1 \kappa_- (x/u)^{1/4} du = \frac{4}{3} \kappa_- x^{1/4}$,
and see that it is larger than the obnoxious term $\frac{5\kappa_-}{7} \cdot x^{\frac{1}{4}}$, which we can thus drop. We add
\[\int_0^1 \left|M\left(\sqrt{\frac{x}{u}}\right)\right| du \leq
\int_0^1 \left(\epsilon \sqrt{\frac{x}{u}} + \kappa \sqrt[4]{\frac{x}{u}}\right)du
= 2 \epsilon \sqrt{x} + \frac{4}{3} \kappa x^{\frac{1}{4}}.
\]

It remains to check that all $v\geq \sqrt{\frac{x}{K+\frac{1}{2}}}$ satisfy $v\geq v_0$: for $y_0$ as in the proof of Prop.~\ref{prop:gould}\ref{it:rogg1},
\[K + \frac{1}{2}< y_0 + 1 \leq\frac{x^{\frac{1}{5}}}{(16 \kappa_-/c_1)^{\frac{4}{5}}} + 1\leq \frac{x}{2 v_0^2} + \frac{x}{2 v_0^2} \leq \frac{x}{v_0^2}.\]

 \end{proof}
\subsection{Initial estimates on square-free numbers}\label{subs:octoplectic}
It is time to apply the general results in \S \ref{subs:maldoror} using our bounds on $M(x)$.
First, let us recall there are computational results: by
\cite[Thm.~2]{zbMATH07298460},
\begin{equation}\label{eq:jbond}|R(x)|\leq 1.12543 x^{\frac{1}{4}}\quad\quad\text{for $0<x\leq 10^{18}$.}\end{equation}
\begin{lemma}\label{lem:nopgik} Let $10^{18}\leq x \leq 10^{32}$. Then
\begin{equation*}|R(x)|\leq 0.845 x^{\frac{2}{5}} \leq 0.0134 \sqrt{x}.
\end{equation*}
\end{lemma}
\begin{proof}
By Lemma \ref{lem:hurst}, Cor.~\ref{cor:mertensimplon},
and Lemma \ref{lem:edesmo}, 
we can apply Prop.~\ref{prop:qmediano}
with $\kappa_- = 0.570591$, $v_0=33$, $v_1=10^{16}$, $\epsilon = \pi/(2\cdot 10^{7})$, $\kappa = 6.738093$, $c_1 = 16/25$ 
and  $c_2 = 114/25$. We obtain
\[|R(x)|\leq \frac{\pi}{10^{7}} x^{\frac{1}{2}} +  0.772103 x^{\frac{2}{5}} + 0.860845 x^{\frac{1}{3}} + 8.98413 x^{\frac{1}{4}} + 0.999 x^{\frac{1}{6}}.\]
Since $x_0\leq x\leq x_1$ for $x_0=10^{18}$ and $x_1=10^{32}$, it follows that
\[\begin{aligned}\frac{|R(x)|}{x^{\frac{2}{5}}}&\leq \frac{\pi}{10^{7}} x_1^{\frac{1}{2}-\frac{2}{5}} + 0.772103 + 0.860845 x_0^{\frac{1}{3}-\frac{2}{5}} + 8.98413 x_0^{\frac{1}{4}-\frac{2}{5}} + 0.999 x_0^{\frac{1}{6}-\frac{2}{5}} \leq 0.84491.
\end{aligned}\]
Again by $x\geq 10^{18}$, $0.845 x^{\frac{2}{5}}\leq 0.0134
x^{\frac{1}{2}}$.
\end{proof}
Since Lemma~\ref{lem:nopgik} will be the bottleneck, a much looser bound for $x\geq 10^{32}$ will suffice for now.
\begin{lemma}\label{lem:gopnik} Let $x\geq 10^{32}$. Then
\begin{equation*}
|R(x)|\leq 0.002 \sqrt{x}.
\end{equation*}
\end{lemma}
\begin{proof}
By Lemma \ref{lem:hurst}, Corollary \ref{cor:mertensimplon},
and Lemma \ref{lem:edesmo}, 
we can apply Prop.~\ref{prop:gould}
 with $\epsilon = \pi/(2\cdot 10^{10})$, $\kappa = 11.350514$,
$c_1 = 768/1225$, and $c_2 = 9458/1225$.  Then Prop.~\ref{prop:gould}\ref{it:rogg1} yields
\begin{equation}\label{eq:roan1}|R(x)|\leq 
 \frac{6.51}{10^{11}} x^{\frac{3}{5}} + 2.5222 x^{\frac{2}{5}} + 1.042 x^{\frac{1}{3}} + 1.086 x^{\frac{1}{6}} \leq\frac{6.51}{10^{11}} x^{\frac{3}{5}} + 2.53 x^{\frac{2}{5}} ,\end{equation}
 since the condition in Prop.~\ref{prop:gould}\ref{it:rogg1} for omitting a term is
 fulfilled, while Prop.~\ref{prop:gould}\ref{it:rogg2} gives
 \begin{equation}\label{eq:roan2}|R(x)|\leq 
\frac{1.99}{10^5} \sqrt{x} +
1.042 x^{\frac{1}{3}} +
6.01\cdot 10^7\cdot x^{\frac{1}{4}} + 1.1 x^{\frac{1}{6}}.
 \end{equation}
 Let $x_1 = 10^{50}$. For $10^{32} \leq x\leq x_1$, we apply \eqref{eq:roan1}, and obtain
 \[\frac{|R(x)|}{\sqrt{x}}\leq \frac{6.51}{10^{11}}\cdot x_1^{\frac{3}{5}-\frac{1}{2}} + 
 2.53\cdot \left(10^{32}\right)^{\frac{2}{5}-\frac{1}{2}}
 \leq 0.00161.\]
 For $x>x_1$, we apply \eqref{eq:roan2}, and see that
 $|R(x)|\leq 3.9\cdot 10^{-5} \sqrt{x}$.
\end{proof}
\begin{corollary}\label{cor:bataclar}
    For $x\geq 10^{8}$, $|R(x)|\leq 0.0134 \sqrt{x}$.
\end{corollary}
\begin{proof}
Immediate from \eqref{eq:jbond}, Lemma~\ref{lem:nopgik} and Lemma~\ref{lem:gopnik}.
\end{proof}
We will later be able to prove improved bounds on $R(x)$ for large $x$ by means of the improved bound on $M(x)$ we will prove using Corollary \ref{cor:bataclar}.

There is a different way to improve our bounds on $R(x)$ that we will not pursue
here: we could use Cor.~\ref{cor:mertens} as an explicit formula, that is,
we could use it to approximate $M(\sqrt{x/u})$ by
a linear combination of terms $(\sqrt{x/u})^{\rho}$, and then attempt to obtain cancellation as $u$ varies. For instance, we can consider the difference $(\sqrt{x/u})^\rho - (\sqrt{x/k})^\rho$ for $u$ close to $k$;
then the analogy with quadrature suggested by \eqref{eq:antenor} comes to the fore. Cf.\ \S \ref{subs:iterprod}.

\section{Conclusion}\label{sec:sqsupp}
\subsection{Main results revisited}
We come to a version of Theorem~\ref{thm:mainthmA} using square-free support.
 \begin{proposition}\label{prop:mainsqfprop}
  Let $A(s)=\sum_n a_n n^{-s}$ extend meromorphically to 
  $\mathbb{C}$. Assume
   $a_\infty = \sup_n |a_n|<\infty$ and also that $a_n=0$ for all non-square-free $n$.
  Let $T\geq 50$. Assume $A(s) T^s$ is bounded on some $S$ as in \eqref{eq:sapli}.  
 Assume as well that $|R(\omega)|\leq c \sqrt{\omega}$ for all $\omega\geq (c T)^2$ and some $c>0$.
Then, for any $\sigma\in \mathbb{R}$ and any $x\geq (e c T)^2$,
   \begin{equation*}
   \begin{aligned}\frac{1}{x^{1-\sigma}} \sum_{n\leq x}\dfrac{a_n}{n^{\sigma}} &= O^*\left(\frac{3 a_\infty}{\pi T}
\right) +  \delta
   \sum_{\rho\in \mathcal{Z}_{A}(T) \cup \{\sigma\}} \Res_{s=\rho} \left(
   w_{\delta,\sigma}(s) A(s) x^{s-1}\right)\\
   &+  
    O^*\left(
\frac{\pi}{4} I +
\frac{\pi a_\infty}{4} \left(\frac{1}{L} + \frac{4}{L^2}\right)\right) \frac{1}{T^2} + 
O^*\left(\frac{2.9 c a_\infty}{\sqrt{x}}\right)
   ,\end{aligned}\end{equation*}
   for $\delta = \frac{\pi}{2 T}$. Here $\mathcal{Z}_{A}(T)$ is the set of poles $\rho$
   of $A(s)$ with $|\Im \rho|\leq T$, and
   \[
   w_{\delta,\sigma}(s) =
    \coth(\delta (s-\sigma)) - \tanh(\delta (1-\sigma)),\]
\[
L= \log \frac{x}{(c T)^2},\quad
I = \frac{1}{2} \sum_{\xi = \pm 1}\int_0^\infty t |A(1-t +i\xi T)| x^{-t} dt.\]
 \end{proposition}
 We could write $\frac{3}{\pi T} \frac{\tanh((\sigma-1) \delta)}{(\sigma-1) \delta}$ instead of $\frac{3}{\pi T}$ in the first term.
\begin{proof}
We will apply Prop.~\ref{prop:summsec3} with
$\varphi=\varphi_\lambda$ and  $A/a_\infty$ in place of $A$, where
$\varphi_\lambda$ is as in \eqref{eq:sonnenblum} and
$\lambda = 2\pi (\sigma-1)/T$.
We will use the bound \eqref{eq:alicen1} in the proof of Prop.~\ref{prop:gendau}, 
with $\kappa=1/16\pi$ as before, and also
\eqref{eq:alicen2}, but with $\alpha = \pi/T$ and 
$c_1<2.85$, $c_2<\frac{2}{9}$, since now we are applying Lemma \ref{lem:artanor} with
$\alpha = \pi/T\leq 1/6$.
Of course $\|\widehat{\varphi}-I_\lambda\|_1$ is still
$\tanh(\lambda/4)/\lambda$. We also use the first bound from Lemma \ref{lem:artanor}:
\[\|\left. e^{-\frac{\pi}{T} y} (\widehat{\varphi}(y) - I_\lambda(y))\right|_{(y_0,\infty)}\|_1\leq \frac{2}{3} + \frac{1}{11 \pi} \cdot \frac{e^{-\frac{\pi}{T} y_0}}{\frac{\pi}{T} y_0^2}
\]
since $T\geq 50>6\pi$. We obtain
\begin{equation}\label{eq:ramones}\begin{aligned}
 x^\sigma S_\sigma(x) &= 
\frac{1}{i T} \int_{1-i T}^{1+i T} \varphi\left(\frac{s-1}{i T}\right)
	A(s) x^s ds +O^*\left( \frac{12/\pi}{T} \frac{\tanh(\lambda/4)}{\lambda}\right)\cdot
    a_\infty x\\
&+   a_\infty\cdot O^*\left(\frac{x}{8 T |y_0|} 
+  \left(\frac{\pi}{T} \left(\frac{2}{3} +  \frac{e^{-\frac{\pi}{T} y_0}}{11 \frac{\pi^2}{T} y_0^2}\right)
+ \frac{\frac{1}{16\pi}}{e^{\frac{\pi}{T} y_0} y_0^2}
+ c_2 \frac{e^{-\frac{\pi}{T} y_0}}{\frac{\pi}{T} y_0^2} + c_1
\right) c \sqrt{x}\right).
\end{aligned}\end{equation}
Letting $y_0 = -\rho T/\pi$, we see that the terms within 
$O^*(\dotsc)$ in the second line  equal
\begin{equation}\label{eq:adanto}A(\rho) = \frac{\pi x}{8 \rho T^2} +
\left(\left(\frac{c_2 \pi}{T} + \frac{\frac{\pi}{11} + \frac{\pi}{16}}{T^2}\right)
\frac{e^\rho}{\rho^2}\right) c \sqrt{x} + \left(c_1 + \frac{2\pi}{3 T}\right) c \sqrt{x},
\end{equation}
and so we should make $K/\rho + e^\rho/\rho^2$ small for 
$K = \frac{(\pi/8)\sqrt{x}}{c\cdot \left(c_2 \pi T + \frac{\pi}{11}+\frac{\pi}{16}\right)}$.
The minimum is reached for $\rho$ close to $\log K$.
We choose $\rho = L/2 = \log \frac{\sqrt{x}}{c T}\geq 1$, and so $y_0\leq -T/\pi$, fulfilling a condition in 
Prop.~\ref{prop:summsec3}; moreover, $e^{\frac{2\pi}{T} y_0} x = (c T)^2$, and so our condition on
$R(\omega)$ guarantees the condition on $R(\omega)$ in Prop.~\ref{prop:summsec3}. 
Note also $(e^\rho/\rho^2)\cdot c \sqrt{x} = 4 x/(L^2 T)$.
By \eqref{eq:adanto},
\[A(\rho) = A(L/2) = 
\frac{\pi x}{4 T^2 L} \left(1 + 
\frac{16}{L}\left(c_2 + \frac{\frac{1}{11} + \frac{1}{16}}{T} \right)\right) +
\left(c_1 + \frac{2\pi}{3 T}\right) c \sqrt{x}.
\]
We simplify: $4\left(c_2 + \frac{1}{T}\left(\frac{1}{11}+\frac{1}{16}\right)\right)
< 1$, $c_1 + \frac{2\pi}{3 T}\leq c_1 + \frac{2\pi}{3\cdot 50} < 2.9$.

The rest of the argument is exactly as in the proofs of Prop.~\ref{prop:gendau} and Thm.~\ref{thm:mainthmA}: we shift our integral in \eqref{eq:ramones}, which is the same integral as in
\eqref{eq:raman}, and, in the cases $\sigma>1$ and $\sigma=1$, we flip signs as in 
the proof of Thm.~\ref{thm:mainthmA}.
\end{proof} 


 \begin{corollary}\label{cor:metamert}
 Assume that all zeros of $\zeta(s)$ with $|\Im s|\leq T$ are simple,
 where $T\geq 50$.  Assume as well that for some $c>0$, $|R(\omega)|\leq c \sqrt{\omega}$ for all $\omega\geq (c T)^2$.
      Then, for any $x\geq \max(e^3 (c T)^2,4 T)$ such that $\max_{r\leq 1} 1/|\zeta(r\pm i T)|\leq (\log^2 x)/3$,
\begin{equation}\label{eq:metagarance}\left|\sum_{n\leq x} \dfrac{\mu(n)}{n^{\sigma}} -
\left(\delta
   \sum_{\rho\in \mathcal{Z}_*(T)} \frac{w_{\delta,\sigma}(\rho)}{\zeta'(\rho)} x^{\rho-\sigma} + 
   \frac{1}{\zeta(\sigma)}\right)\right|
\leq \frac{3/\pi}{T-1}\cdot x^{1-\sigma} +  2.9 c x^{\frac{1}{2}-\sigma},
   \end{equation}
where $\delta = \frac{\pi}{2 T}$,
$\mathcal{Z}_*(T)$ is the set of non-trivial zeros $\rho$
of $\zeta(s)$ with $|\Im \rho|\leq T$, and
 $w_{\delta,\sigma}$ is as in \eqref{eq:mainsmooth}. 
\end{corollary}
There is no deeper meaning to whether we make some of the conditions $T\geq 6\pi$, $x> e^3 (c T)^2$, $\max_{r\leq 1} \dotsc \leq (\log^2 x)/3$ stricter or looser;
we are just fitting things into the neat inequality \eqref{eq:metagarance}. 
\begin{proof}
The proof of Cor.~\ref{cor:jolene} (an application of
Lemma \ref{lem:pommedupe}, and a check on boundedness on
$\partial R\cup L$) goes through without any change.
It is only the estimation of the main error term that is slightly different
from that in the proof of Cor.~\ref{cor:mertens}, in that
we now have
\begin{equation}\label{eq:excali}\frac{\pi}{4} \frac{L^{-1} + 4 L^{-2} + I}{T^2} +
\frac{2.9 c}{\sqrt{x}}\quad\text{instead of}\quad
 \frac{\pi}{4} \frac{L^{-1} + L^{-2} + I}{T^2} + \frac{2}{x},\end{equation}
and $L = \log \frac{x}{(c T)^2}$ rather than $L = \log \frac{x}{T}$.
The meaning of $I$ is as before, though now \eqref{eq:patatan} gives us $I\leq 1/3$, since we are assuming that $\mathbf{c}$ in \eqref{eq:patatan} is
$\leq (\log^2 x)/3$. 
Thanks to  $x\geq e^3 (c T)^2$,
we have $L\geq \log e^3 = 3$, and so $L^{-1} + 4 L^{-2} + I \leq \frac{1}{3} + \frac{4}{9} + \frac{1}{3} = \frac{10}{9}$. As for the trivial zeros,
\[\varepsilon_-(\delta,\sigma) x^{-2}\leq \left(\frac{1}{2+(-1)} + 2\cdot \frac{\pi}{2\cdot 6\pi}\right) \cdot \frac{(2\pi)^2}{\zeta(3)} x^{-2}   < \frac{39}{x^2} \leq \frac{39}{64} \frac{x}{T^3}.\]
Since
$\frac{3/\pi}{T} + \frac{\pi}{4} \frac{10/9}{T^2} + \frac{39}{64 T^3} < \frac{3/\pi}{T-1}$, we are done.
\end{proof}

\begin{corollary}\label{cor:mertosimp}
 Let $M(x) = \sum_{n\leq x} \mu(n)$ and $m(x) = \sum_{n\leq x} \mu(n)/n$.
Then, for $x\geq 3.61\cdot 10^{17}$,
   \begin{equation}\label{eq:metagarance2}|M(x)|\leq \frac{3}{\pi\cdot 10^{10}}\cdot x + 11.39 \sqrt{x},\;\;\;\;\;\;\;\;\;\;
   |m(x)|\leq \frac{3}{\pi\cdot 10^{10}} + \frac{11.39}{\sqrt{x}}.
   \end{equation}
\end{corollary}
\begin{proof}
Let $T=10^{10}+1$. We proceed
exactly as in the proof of Corollary \ref{cor:mertensimplon} with
$T=10^{10}+1$, except we
apply Cor.~\ref{cor:metamert} instead of Cor.~\ref{cor:mertens}, and there are no term $2/x$ to absorb, and the bound for $|M(x)|$ has to absorb $1/|\zeta(0)|$, that is, $2$, rather than $4$.
By Cor.~\ref{cor:bataclar}, the condition  $|R(\omega)|\leq c\sqrt{\omega}$ holds with $c=0.0134$ for all $\omega\geq (c T)^2$.
The term $2.9 c x^{\frac{1}{2}-\sigma} \leq 0.0389 x^{\frac{1}{2}-\sigma}$ from Cor.~\ref{cor:metamert} is added to the
term $C x^{\frac{1}{2}-\sigma}$ with $C = 11.350514$. Condition
$x\geq e^3 (c T)^2$ holds by $3.61\cdot 10^{17}> e^3 (c T)^2$.
\end{proof}

\begin{proof}[Proof of Corollary \ref{cor:mertensimple}]
If $x\geq 3.61\cdot 10^{17}$, apply Corollary~\ref{cor:mertosimp}.
If $x< 3.61\cdot 10^{17}$, use \eqref{eq:garadiol} in Corollary~\ref{cor:mertensimplon}; we are then done by
\[\frac{\pi}{2\cdot (10^9-1)} + \frac{9.7588}{\sqrt{x}} 
- \left(\frac{3}{\pi\cdot 10^{10}} + \frac{11.39}{\sqrt{x}}\right) <
\frac{\pi}{2\cdot (10^9-1)} - \frac{3}{\pi\cdot 10^{10}} -  \frac{11.39-9.7588}{\sqrt{3.61\cdot 10^{17}}}
< 0.
\]
\end{proof}

\subsection{Coda on $R(x)$}\label{subs:coda}

Given that we have improved our bounds on $M(x)$ using estimates on square-free numbers, it would be almost churlish not to use our improved bounds on $M(x)$ to strengthen our estimates on square-free numbers.

\begin{proposition}
Let $Q(x)$ be the number of square-free $1\leq n\leq x$. Let $R(x) = 
Q(x) - \frac{x}{\zeta(2)}$. 
\begin{enumerate}[(a)]
\item For $0< x\leq 10^{32}$,
\begin{equation}\label{eq:fish}|R(x)|\leq 0.845 x^{\frac{2}{5}}.\end{equation}
\item For $x>0$,
\begin{equation}\label{eq:shark}
|R(x)|\leq \frac{1.537}{10^5} \sqrt{x} + 2.515 x^{\frac{2}{5}}.
\end{equation}
\end{enumerate}
\end{proposition}
\begin{proof}
The bounds here follow from \eqref{eq:jbond} for $7\leq x\leq 10^{18}$,
and hold for $0<x\leq 7$ by a direct check.
For $10^{18}\leq x\leq 10^{32}$, \eqref{eq:fish} holds by Lemma \ref{lem:nopgik},
and hence \eqref{eq:shark} holds as well. It remains to prove that \eqref{eq:shark} holds for $x>10^{32}$.

By Corollary \ref{cor:mertensimple}, Lemma \ref{lem:hurst}
and Lemma \ref{lem:edesmo}, 
we can apply Prop.~\ref{prop:gould}
 with $\epsilon = 3/(\pi\cdot 10^{10})$, $\kappa = 11.39$,
     $c_1=\tfrac{442368}{715715}$ and $c_2=\tfrac{14328304}{715715}$. Then Prop.~\ref{prop:gould}\ref{it:rogg1} yields
\begin{equation}\label{eq:groan1}|R(x)|\leq 
 \frac{3.9282}{10^{11}} x^{\frac{3}{5}} + 2.50422 x^{\frac{2}{5}} + 1.446 x^{\frac{1}{3}} + 1.27 x^{\frac{1}{6}} \leq\frac{3.9282}{10^{11}} x^{\frac{3}{5}} + 2.5149 x^{\frac{2}{5}} ,\end{equation}
 since the condition in Prop.~\ref{prop:gould}\ref{it:rogg1} for omitting a term is
 fulfilled, while Prop.~\ref{prop:gould}\ref{it:rogg2} gives
 \begin{equation}\label{eq:groan2}|R(x)|\leq 
\frac{1.5366}{10^5} \sqrt{x} +
1.446 x^{\frac{1}{3}} +
8.664\cdot 10^7\cdot x^{\frac{1}{4}} + 1.3 x^{\frac{1}{6}}\leq\frac{1.537}{10^5} \sqrt{x} + 8.664\cdot 10^7\cdot x^{\frac{1}{4}}
 \end{equation}
 for $x\geq 10^{52}$. If $x\leq 7.98\cdot 10^{52}$, then $3.9282\cdot 10^{-11} \cdot x^{\frac{3}{5}}\leq 7.7\cdot 10^{-6}\sqrt{x}$, and \eqref{eq:shark} follows from
 \eqref{eq:groan1}; when $x> 7.98\cdot 10^{52}$, then
 $8.664\cdot 10^7\cdot x^{\frac{1}{4}} \leq 1.01 x^{\frac{2}{5}}$
 and \eqref{eq:shark} follows from
 \eqref{eq:groan2}.
\end{proof}
We could go further, as discussed at the end of \S \ref{subs:octoplectic}, but this isn't a paper on square-free numbers. 

  \section{Final remarks}\label{sec:finremark}

Large parts of the basic toolkit of explicit analytic number theory will have to be redone in view of the results and methods here
and in \cite{Nonnegart}. That includes most of the first chapter of \cite{TMEEMT}.


  \subsection{Prior work on $M(x)$}\label{subs:compmeth}
\subsubsection{Bounds of the form $|M(x)|\leq \epsilon x$}\label{subs:Mxbounds}
The basic idea of previous bounds on $M(x)$ goes back to 
\cite{zbMATH02670421} (1898), whose method is a variant of
Chebyshev's \cite{JMPA_1852_1_17__366_0}. In brief: 
let $F(N) = \sum_{k=1}^N \mu(k) \lfloor N/k\rfloor$. For $N\geq 1$,
$F(N) = 1$
by Möbius inversion. Thus, for $n\geq 6$,
\begin{equation*}
-2 = F(n) - F\left(\left\lfloor \frac{n}{2}\right\rfloor\right)
- F\left(\left\lfloor \frac{n}{3}\right\rfloor\right) - F\left(\left\lfloor \frac{n}{6}\right\rfloor\right) = 
 \sum_{k=1}^n \mu(k) f\left(\frac{n}{k}\right)
\end{equation*} where
$f(n) = n - \lfloor n/2\rfloor - \lfloor n/3\rfloor
- \lfloor n/6\rfloor$. Now, $f(n)=1$ for $n\not\equiv 0,5\bmod 6$,
and $|f(n)-1|=1$ for $n\equiv 0,5\bmod 6$. In other words, 
$\sum_{k=1}^n \mu(k) f\left(\frac{n}{k}\right)$ is an approximation to $M(n)$; a simple
bound on their difference was enough for
von Sterneck to obtain $|M(n)|\leq n/9 + 8$ for all $n>0$.

The same idea, with $f(n)$ replaced by increasingly complicated linear combinations, gave:
\begin{center}\small
\begin{tabular}{@{}llr@{}}
\toprule
bound on $|M(x)|$ & valid for $x$ at least & reference\\ \midrule
 $x/9+8$ & $0$   & \cite{zbMATH02670421}  \\
 $x/26+155$ & $0$  & \cite{zbMATH02636882} \\
 $x/80$    & $1{\,}119$  & \cite{zbMATH03255840}  \\ 
 $x/105+C$ & $0$  &\cite{zbMATH03749090}\\
 $x/1036$ & $120{\,}727$  & \cite{zbMATH04139828} \\
$x/2360$& $617{\,}973$ & \cite{zbMATH00530101}   \\
 $x/4345$ & $2{\,}160{\,}535$ & \cite{zbMATH05257415}  \\ \bottomrule
\end{tabular}
\end{center}
It was already proved in \cite{zbMATH03749090}, assuming
$M(x) = o(x)$, that, for any $\epsilon>0$, one can in principle find linear combinations of $F(n/k)$ showing
that $|M(x)|/x \leq \epsilon$ for $x$ large;
cf. \cite{zbMATH03708490}. All linear combinations since \cite{zbMATH03255840} were found with computer help; even the one in \cite{zbMATH03255840} -- the last combination to actually appear in print -- took $3$ pages, and the one in \cite{zbMATH05257415} had 63951 terms.

 Lastly, \cite{Daval2} has obtained
$|M(x)|\leq x/160383$ for $x\geq 8.4\cdot 10^9$ by combining
\cite{zbMATH05257415} with a method previously used (see \S \ref{subs:iterprodhist}) 
to derive
bounds $|M(x)|\leq \epsilon_\alpha/(\log x)^\alpha$ from $|M(x)|\leq \epsilon x$.

Had our method been known earlier, it would have given
results far better than those in the table above, even with technology available at the time: computing $\zeta'(s)$ is no harder than computing
$\zeta(s)$, by any 
method commonly used (Euler--Maclaurin generalizes trivially, and versions
of Riemann--Siegel for $\zeta'(s)$ exist; FFT-based amortization strategies
\cite{zbMATH04158743} for computing many values of $\zeta(s)$ also allow computing many values of $\zeta'(s)$), and so it seems fair to compare, say, all bounds above from 1969 onwards with $$|M(x)|\leq \frac{x}{1\,206\,036} + 5.708356\sqrt{x}\;\;\;\;\;\;\;\;\;\;\text{(valid for all $x\geq 0$)},$$ which is what Cor.~\ref{cor:mertens} yields if given values of
$\zeta'(\rho)$ for $\Im \rho\leq 1\,894\,438.5$, the height reached
by the RH verification in \cite{zbMATH03304440}.

\subsubsection{Deriving explicit bounds on $M(x)=o(x)$ from $|M(x)|\leq \epsilon x$ and PNT}\label{subs:iterprodhist}

It is possible to prove in an elementary way that the prime number theorem
$\psi(x) = (1+o(1)) x$ implies $M(x) = o(x)$ 
\cite[Ch. XLI, \S 155]{zbMATH02636836}, \cite[\S 8.1]{MR2378655}; the proof can be made to yield
explicit bounds on $M(x)$, though they are much poorer than the bounds on
$\psi(x)-x$ one starts from.

It is better to take the basic idea
from the proof -- namely, to use the identity $\mu\cdot \log = -\Lambda\ast \mu$
or some variant thereof -- and take both an explicit PNT and bounds of
the form $|M(x)|\leq \epsilon x$ as inputs: one can bound
$\sum_{n, m: n m \leq x} \Lambda(n) \mu(m)$ by bounding $M(x/m)$
 for $n$ less than some $y$ and $\psi(x/m)$ for $m\leq x/y$.
 In this way, we obtain a bound on $S_1 = \sum_{n\leq x} \mu(n) \log(n)$;
we add the sum $S_2 = \sum_{n\leq x} \mu(n) \log(x/n)$, which is bounded
by $\sum_{n\leq x} \log(x/n) \leq x$, and thus obtain a sum on 
$M(x) = (S_1 + S_2)/x$. This has been the standard procedure since
\cite{zbMATH03281848}.

One can iterate this procedure, that is, one can use the explicit bounds $M(x)=o(x)$ thus obtained as inputs, in place of $|M(x)|\leq \epsilon x$. In this way, one obtains bounds on $M(x)$ that are asymptotically better but, after a couple of iterations, become unusable in practice.

The best explicit results of this kind to date are as follows:
\[|M(x)|\leq \frac{0.006688 x}{\log x}\:\:\text{for $x\geq 1079974$ \cite{ramarseb}},
\quad |m(x)|\leq \frac{0.0130073}{\log x}
\:\:\text{for $x\geq 97603$ \cite{zbMATH08050272}}
\]
\[|M(x)|\leq \frac{362.7 x}{\log^2 x}\:\:\text{for $x\geq 2$ \cite{MR1378588}},
\quad |m(x)|\leq \frac{362.84}{\log^2 x}
\:\:\text{for $x\geq 2$ \cite{zbMATH08050272}}.
\]


\subsection{Generalizations}
Let us sketch several matters as an aid to future work.
\subsubsection{Sums of other Dirichlet series}
We have chosen $\mu$ as our main example, but there are other natural applications of our main results.
For instance, take 
Dirichlet characters $\chi(n)$. We can apply Thm.~\ref{thm:mainthmA} to
$a_n = \chi(n) \mu(n)$. Then Platt's
verification of $GRH$ for $\chi \bmod q$ for $q\leq 400\,000$ up to height $H_q=10^8/q$ \cite{zbMATH06605790} will just need to be supplemented with a smaller
computation to bound $\max_{\sigma\in [-1/64,1]} 1/L(\sigma + i H_q,\chi)$ for each $\chi$.

One can do better if one aims at sums of $\mu(n)$ 
on arithmetic progressions $\bmod\:q$.
Arithmetic progressions are obviously sparse; gaining a factor of 
$q$ in the main term of Prop.~\ref{prop:sumwiz} by restricting to $a + q\mathbb{Z}$ is very easy. 
Thus, one should be able to replace the term $O^*(\delta) = O^*(\pi/(2 H_q))$ in Thm.~\ref{thm:mainthmA}  by $O^*(\pi/(2 q H_q)) = O^*(\pi/(2\cdot 10^8))$.

Much the same goes for $L$-functions of higher degree (in the sense of
Selberg class $S$). In particular: if $L(s) = \sum_n a_n n^{-s}$ is a primitive function in $S$ such that (i) $a_{p^k}=O(1)^k$ for all $p$, $k$ and (ii) Selberg's
first conjecture ($\sum_{p\leq x} |a_p|^2/p = \log \log x + O(1)$) holds,
it is easy to see that, for $1/L(s) = \sum_n b_n n^{-s}$,
conditions (i) and (ii) are also satisfied for $b_n$ instead of $a_n$,
and so $b_n$ will be bounded on average over short intervals \cite{zbMATH03639723}\footnote{We thank O. Gorodetsky for this reference.}; thus, Thm.~\ref{thm:mainthmA} 
should generalize easily to $\{b_n\}$.

\subsubsection{Sums with continuous weights}\label{subs:contweight}

Weighted sums, e.g., $\check{M}(x) = \sum_{n\leq x} \mu(n) \log(x/n)$
and $\check{m}(x) = \sum_{n\leq x} \mu(n) \log(x/n)/n$, have been
carefully estimated (\cite{ramare2015explicit}, \cite{Ramare19corrig},
\cite{zbMATH08050272}), in part because of their applications
\cite{Helfbook}, and in part because of their role in iteration (\S \ref{subs:iterprodhist}). Existing estimates were all derived from the bounds on $M(x)$ listed in \S \ref{subs:compmeth}.

We could also derive bounds on sums with continuous weights from
our bounds on $M(x)$, but it is better to estimate such sums
directly following our method. 
For $m(x)-M(x)/x$, the problem reduces (by Lemma \ref{lem:edge} and Prop.~\ref{prop:sumwiz}) to
finding $\varphi$ with support on $[-1,1]$ such that
$\|\widehat{\varphi} - 1_{[0,\infty)}(t)\cdot (e^{-\lambda t} - 1)\|_1$
  is minimal (a problem already solved in \cite[Thms.~1 and 2]{zbMATH06384942});
  for $\check{m}$, what is to be minimized is 
  $\|\widehat{\varphi} - 1_{[0,\infty)}(t)\cdot t e^{-\lambda t}\|_1$
    (\cite[Thms.~2.6, 2.8]{zbMATH05639041}; some work may remain).

    As a result, we should obtain analogues of our main results
    for $m(x)-M(x)/x$ and $\check{m}(x)$, among other sums.
    The constant term in the bounds on either sum
    will be proportional to $1/T^2$, rather than $1/T$.
    The contribution of the zeros of $\zeta(s)$ up to height $T$ will also be
    smaller than it is for $m(x)$; it will be of the form $c_T \sqrt{x}$,
    where $c_T$ should in fact be bounded by a small constant.
    
    
  

    \subsubsection{Sums with coprimality conditions}
    Let $M_q(x) = \sum_{n\leq x: (n,q)=1} \mu(n)$, and define
    $m_q$, $\check{m}_q$ similarly. 
    Since, for any function $h:\mathbb{Z}_{>0}\to \mathbb{C}$,
    \begin{equation}\label{eq:mahalia}
      \sum_{n: (n,q)=1} \mu(n) h(n) = \sum_{d|q^\infty} \sum_n \mu(n) h(d n),
      \end{equation}
     it is easy to deduce bounds on
    $M_q$, $m_q$, etc., from bounds on $M$, $m$, etc., such as
     Corollary \ref{cor:mertensimple}. In this way, we get, for instance,
     $|M_q(x)|\leq \frac{\pi}{2\cdot 10^{10}} \frac{q x}{\phi(q)} + C \sqrt{x}
     /\prod_{p|q} (1-1/\sqrt{p})$, where $C=C_1$ is as in Cor.~\ref{cor:mertensimplon}.
     One can do better by proceeding as follows: apply
     Prop.~\ref{prop:sumwiz} to bound
     the difference between $M_q(x)$ (say) and
     $\sum_{n\leq x} \mu(n) (x/n) \widehat{\varphi}((T/2\pi) \log(x/n))$;
     then apply \eqref{eq:mahalia} to the latter sum, and then use Lemma \ref{lem:edge} to estimate the inner sum in \eqref{eq:mahalia}. In this way, the term $(\pi/2 T) x$ does not get multiplied
     by $q/\phi(q)$.

     What one should {\em not} do is try to apply Thm.~\ref{thm:mainthmA}
     directly to \[\sum_{n: (n,q) = 1} \mu(n) n^{-s}
     = (1/\zeta(s))/\prod_{p|q} (1-p^{-s}),\]
     as the residues on $\Re s = 0$ are hard to control. (One could shift the
     line of integration only to $\Re s = \delta$ for some $\delta>0$, but the above approach seems superior.)

     Imposing coprimality conditions starting from estimates of the form
     $M(x) \leq c x/\log x$, $m(x) \leq c/\log x$, etc. (as in
     \S \ref{subs:iterprod}) is a little trickier; there are at least two
     approaches in the literature (\cite[\S 8--16]{ramare2015explicit} and \cite[\S 5.3.3--5.3.4]{Helfbook}).   

     \subsubsection{Square-free numbers}\label{subs:sqfrdiscuss}
As before, let $Q(x)$ be the number of square-free integers $1\leq n\leq x$, and let $R(x) = Q(x) - x/\zeta(2)$.
In \S \ref{sec:mainsqfr}, we showed how to derive and apply new bounds on $R(x)$.

{\bf \em Short-interval estimates.}  
An alternative is to apply an explicit short-interval estimate
on $R(x)$, that is, a bound on $|R(x+y)-R(x)|$ for $y$ much smaller than $x$.
The idea here is that the difference between $\widehat{\varphi}\left(\frac{T}{2\pi} \log \frac{x}{n}\right)$ and $I\left(\frac{T}{2\pi} \log \frac{n}{x}\right)$  (see \S \ref{subs:diffweights}) is large mainly for $n$ in a short interval
around $x$, so we want to bound the number of square-free numbers in short intervals.

There are explicit short-interval bounds on $R(x)$ in the literature \cite{zbMATH04198111}. 
One can easily derive    \begin{equation*}
       |R(x+y)-R(x)|\leq 1.6749 y^{\frac{1}{2}} + 1.4327 x^{\frac{1}{3}}\end{equation*}
     for all $x,y\geq 0$ simply by combining 
Theorems 1 and 2 in \cite{zbMATH04198111}.



There are, however, issues with \cite{zbMATH04198111}: it is 
an extended abstract, without full proofs, and its main intermediate result contains numerical mistakes.
A student supervised by one of us has worked out and corrected \cite{roudymasters} the results in
\cite{zbMATH04198111}\footnote{We are deeply grateful to H. Cohen and F. Dress, who kindly shared with us their unpublished notes on \cite{zbMATH04198111}
.}, but that is currently an unpublished master's thesis.
Incidentally, before now, the literature in the field relied on 
\cite{zbMATH04198111}, via \cite{zbMATH05257415}. 

\subsection{An improved approach to bounds 
$|M(x)|\leq \epsilon_k x/(\log x)^k$}\label{subs:iterprod}
We could apply the iterative procedure in \S \ref{subs:iterprodhist} using
Corollary \ref{cor:mertensimple} as an input.
However, it seems better to use 
Corollary \ref{cor:mertens} or \ref{cor:metamert}, that is, one
of our finite explicit formulas,
and examine the contribution of each non-trivial zero $\rho$ of $\zeta(s)$. In fact,
this is an example of why it is important that we have
explicit formulas, and not just bounds. Let us sketch matters, leaving details for a later paper.

We may start from the identity in 
\cite[\S 3]{zbMATH03281848}, \cite[II.7,(44)--(45)]{zbMATH03208366}: for any $1\leq y\leq x$,
\begin{equation}\label{eq:utome}
-\sum_{n\leq x} \mu(n) \log n = \sum_{k\leq y} (\Lambda(k)-1) M\left(\frac{x}{k}\right) +
\sum_{j\leq x/y} \mu(j) \sum_{k\leq x/j} (\Lambda(k)-1) - 
M\left(\frac{x}{y}\right) \sum_{k\leq y} (\Lambda(k)-1) + 1.
\end{equation}
(There are alternatives: \cite[(1)]{ramarseb}, \cite[Lemme 1]{Daval2}.)
Let us look into the
first sum on the right.

Applying Cor.~\ref{cor:mertens} with $\sigma = 0$, we obtain
$$\begin{aligned}
 \sum_{k\leq y} (\Lambda(k)-1) M\left(\frac{x}{k}\right) &= 
 \frac{\pi}{2 T} \sum_{k\leq y} (\Lambda(k)-1)
 \sum_{\rho \in \mathcal{Z}_*(T)} 
\frac{w_{\delta,0}(\rho)}{\zeta'(\rho)} \left(\frac{x}{k}\right)^\rho +  O^*(\text{err}(y))\\
 &=  \frac{\pi}{2 T} \sum_{\rho \in \mathcal{Z}_*(T)} 
\frac{w_{\delta,0}(\rho)}{\zeta'(\rho)} x^\rho \sum_{k\leq y}
\frac{\Lambda(k)-1}{k^\rho} + O^*(\text{err}(y)),
\end{aligned}$$
where $\text{err} = \sum_{k\leq y} |\Lambda(k)-1| \left(\frac{\pi/2}{T-1} \frac{x}{k} + 4\right)$ has
$\frac{\pi/2}{T} x \log y$ as its main term.

We now estimate $\sum_{k\leq y} \frac{\Lambda(k)}{k^\rho}$,
 applying Theorem 1.1 in \cite{Nonnegart}.
While, like Theorem \ref{thm:mainthmA} in this paper,
it gives an estimate for
$\sum_{k\leq y} \frac{\Lambda(k)}{k^\sigma}$ with $\sigma$ real,
not complex,
we can apply it
with $\sigma = \frac{1}{2}$ to
\begin{enumerate}[(a)]
\item
$a_n = (1 + \Re n^{-i \gamma}) \Lambda(n)$, that is, $A(s) = F(s) + (F(s+i \gamma) + F(s-i\gamma))/2$,
\item
$a_n = (1 + \Im n^{-i \gamma}) \Lambda(n)$, and so $A(s) = F(s) + (F(s+i \gamma) - F(s-i\gamma))/2 i$,
\end{enumerate}
where $\gamma = \Im \rho$ and $F(s) =  - \zeta'(s)/\zeta(s)$.

Verifications of RH typically go up to $T'$  higher
than the height $T$ up to which we have residue computations,
and so we know that the shifts $F(s\pm i \gamma)$ have no poles $s$ with $\Re s>1/2$ and $|\Im s|\leq T'-T$
except for $1\mp i\gamma$. We thus obtain a bound uniform
over $\rho$ of the form
\begin{equation*}
\left|\sum_{k\leq y} \frac{\Lambda(k)}{k^\rho} - \frac{y^{1-\rho}}{1-\rho}\right|\leq 
\frac{\pi}{T'-T}\cdot 2 \sqrt{y} + C',\end{equation*}
and so we get
a bound on $|\sum_{k\leq y} (\Lambda(k)-1) M(x/k)|$ whose main terms
are of the form \begin{equation}\label{eq:turando}
\frac{\pi}{T'-T} 2 C \sqrt{x y} + C C' \sqrt{x} + 
\frac{\pi}{2 T} x \log y,
\end{equation}
where $C = \frac{\pi}{2 T}
   \sum_{\rho\in \mathcal{Z}_*(T)}
\left|\frac{\coth(\delta \rho)}{\zeta'(\rho)}\right|$.

To estimate the double sum in \eqref{eq:utome}, we will need an explicit version of 
$\psi(x) = (1+o(1)) x$.
We can apply \cite[Cor.~1.4]{zbMATH07723301}, which is of the form $|\psi(x) - x| \leq K (\log x)^{3/2} e^{-c\sqrt{\log x}} x$. Then
$$\begin{aligned}\left|\sum_{j\leq x/y} \mu(j) 
\sum_{k\leq x/j} (\Lambda(k)-1)\right|&\lesssim K \int_1^{x/y} \left(\log
\frac{x}{t}\right)^{3/2} e^{-c\sqrt{\log \frac{x}{t}}} \cdot \frac{x}{t} dt\\
&= \frac{2 K}{c^5} x \int_{c \sqrt{\log y}}^{c \sqrt{\log x}} u^4 e^{- u} du\leq \ P_4\left(c \sqrt{\log y}\right) e^{- c\sqrt{\log y}} x,
\end{aligned}$$
where $P_4$ is a polynomial of degree $4$.
One can of course save a factor of $\zeta(2)$ in the first step.

We obtain a bound of type $\varepsilon x$ here
by setting $y$ equal to a large constant; adding this bound to 
\eqref{eq:turando}, we will be able to obtain a bound on
$\sum_{n\leq x} \mu(n) \log n$ of type $\varepsilon' x + C'' \sqrt{x}$.
We recall that we will have very good bounds on $\sum_{n\leq x} \mu(n) \log \frac{x}{n}$ (\S
\ref{subs:contweight}), and so, by
$$M(x) \log x = \sum_{n\leq x} \mu(n) \log \frac{x}{n}
+ \sum_{n\leq x} \mu(n) \log n,$$
we will obtain a result of the form 
$|M(x)|\leq \varepsilon' x/\log x + C''' \sqrt{x}$, where $\varepsilon'$ is very small and $C'''$ is actually not so large (since $C''$ gets divided by $\log x$), and similarly for $|m(x)|$.

Instead of iterating, we could use identities like \cite[(1)]{ramarseb}
to obtain bounds of the form $|M(x)|\leq \varepsilon_2 x/(\log x)^2 + C_2 \sqrt{x}$ and
so forth, with $\varepsilon_2$ small enough for the bound to be useful.

The main point is that Theorem \ref{thm:mainthmA}, taken together
with estimates on $\psi(x)$, will serve as the basis on which
to build a succession of estimates on $|M(x)|$ that are better and
better asymptotically.











  \subsection{Computational-analytic bounds}\label{subs:companal}


The constant $C$ in Cor.~\ref{cor:mertensimple} is both a little
bothersome and really there, or at least $C$ in Cor.~\ref{cor:mertens} is really there: 
for some very rare, extremely large $x$ 
the arguments of all or most terms $x^{1/2+i\gamma}$ will line up, and give
us a sum of size $C \sqrt{x}$. We do not, however, expect this
to happen for $x\leq 10^{30}$, say, beyond which point the leading term in
Corollary~\ref{cor:mertensimple}
is clearly dominant.

How do we find cancellation in the sum over non-trivial zeros in
Cor.~\ref{cor:mertens} in a range $x_0\leq x\leq x_1$, then? Here $x_0$ would be the end of the brute-force range (currently $x_0 = 10^{16}$).

The basic strategy is known (\cite[\S 4.4]{odlyzko1020}; see also the implementation in \cite{zbMATH06864192}): we can see the finite sum $\sum_\gamma x^{i \gamma}$ as the Fourier
transform of a linear combination of point measures
$\delta_{\frac{\gamma}{2\pi}}$, evaluated at $\log x$.
To bound that transform throughout the range
$[\log x_0,\log x_1]$, it is enough, thanks to a
Fourier interpolation formula (Shannon-Whittaker\footnote{\cite{odlyzko1020} recommends \cite{zbMATH03895606} for a historical overview.}), to evaluate
it at equally spaced points. Actually, we first split the range of $\gamma$ into segments of
length $L$; then we need to evaluate the transform only at integer multiples of $2\pi/L$. That one does by
applying a Fast Fourier Transform.

We propose what may be an innovation: do not split
the range brutally into segments; rather,
express the constant function as a sum\footnote{One can think of pinking shears closing perfectly. This analogy, proposed by G. Kuperberg, seems more precise than the jaws of an idealized vertical crocodile.} of triangular
functions $t\to \tri(t/L+n)$, where $\tri(t) =   (1_{[-1/2,1/2]}\ast 1_{[-1/2,1/2]})(t)$. Since 
$\widehat{\tri}(x) = \frac{\sin^2 \pi x}{(\pi x)^2}$,
it is not hard to obtain, in effect, an interpolation
formula with non-negative weights of fast decay.\footnote{D. Radchenko suggests partitioning the constant function using $(1_{[-1/2,1/2]})^{\ast 2 m}$ instead, as then, for $m>1$, decay is even faster than for $m=1$.}

  \appendix

\section{Norms and expressions for extremal functions}\label{sec:appnorms}
\subsection{Norms} 
We need bounds on the Fourier transform of our optimal weight function $\varphi_\lambda$.

\begin{lemma}\label{lem:arborio}
Let $I_\lambda$, $\lambda\ne 0$, be as in \eqref{eq:truncexp}.
Let $\varphi_\lambda$ be as in \eqref{eq:sonnenblum}. 
 Then, for $u$ real with $|u|\geq 1/2$,
\begin{align} \label{12_16pm}
    |(\widehat{\varphi_\lambda}-I_\lambda)(u)|\leq \dfrac{1}{16\pi u^2}, \,\,\,\,\,\, \mbox{and}\,\,\,\,\,\, |(\widehat{\varphi_\lambda}-I_\lambda)'(u)|\leq \dfrac{3}{22 u^2}.
\end{align}
     For
    $0<|u|<\frac{1}{2}$, $u\sgn(\lambda) (\widehat{\varphi_\lambda}-I_\lambda)(u)<0$ and $\sgn(\lambda)(\widehat{\varphi_\lambda}-I_\lambda)'(u)>0$. Moreover, $(\widehat{\varphi_\lambda}-I_\lambda)(0^+)-(\widehat{\varphi_\lambda}-I_\lambda)(0^-)=-\sgn(\lambda)$, and  $(\widehat{\varphi_\lambda}-I_\lambda)(\pm\frac{1}{2})=0$. 
 \end{lemma}
\begin{proof}
Clearly
$\widehat{\varphi_\lambda}(u)-I_\lambda(u) = K_{|\lambda|/2}(v) -I_{|\lambda|/2}(v)
= K_{|\lambda|/2}(v) - E_{|\lambda|/2}(v)$
for $v = 2\sgn(\lambda) u \ne 0$, so we shall focus on 
$K_{|\lambda|/2} - E_{|\lambda|/2}$. 

{\bf \em Decay.}
By \cite[Lemma~5]{zbMATH06384942},
\begin{equation} \label{eq:monthmay}
	K_{|\lambda|/2}(u) - E_{|\lambda|/2}(u)  = \left\{
	\begin{array}{ll}\vspace{0.3cm}
		\dfrac{\sin\pi u}{\pi}\displaystyle\int_{0}^\infty(b(|\lambda|/2+w)-b(|\lambda|/2))e^{uw}dw \,\,\,\,\,\, \mathrm{if\ } u<0,  \\
		\dfrac{\sin\pi u}{\pi}\displaystyle\int_{-\infty}^0(b(|\lambda|/2)-b(|\lambda|/2+w))e^{uw}dw \,\,\,\,\,\, \mathrm{if\ } u>0, \
	\end{array}
	\right.
\end{equation}
where $b(w)=1/(1+e^w)$. Using integration by parts in \eqref{eq:monthmay}, we obtain
\begin{equation}     \label{eq:monthmayending}
	K_{|\lambda|/2}(u) - E_{|\lambda|/2}(u)  = \left\{
	\begin{array}{ll}\vspace{0.3cm}
		\dfrac{\sin\pi u}{\pi u}\displaystyle\int_0^\infty -b'(|\lambda|/2+w) e^{u w} dw \,\,\,\,\,\, \mathrm{if\ } u<0,  \\
		\dfrac{\sin\pi u}{\pi u} \displaystyle\int_{-\infty}^0 b'(|\lambda|/2+w) e^{u w} dw \,\,\,\,\,\, \mathrm{if\ } u>0. \
	\end{array}
	\right.
\end{equation}
Since $b'(w)=-e^w/(1+e^w)^2$ we get that $|b'(w)|\leq 1/4$ for all $w\in \R$. Hence, for $u\neq 0$,
\begin{align} \label{eq:decaigo}
	|K_{|\lambda|/2}(u) - E_{|\lambda|/2}(u)| \leq \left|\dfrac{\sin\pi u}{4\pi u}\right|\displaystyle\int_{0}^\infty e^{-|u|w}dw 
    = \frac{|\sin \pi u|}{4\pi|u|^2}.
\end{align}
Furthermore, taking derivatives in \eqref{eq:monthmayending}, for $u\neq 0$,
\begin{align} \label{eq:rachma2}
|K'_{|\lambda|/2}(u) - E'_{|\lambda|/2}(u)| & 
\leq \left|\left(\dfrac{\sin \pi u}{\pi u}\right)^{\!\!'}\right|\displaystyle\int_{0}^\infty \frac{e^{-|u|w}}{4} dw  + \left|\dfrac{\sin \pi u}{\pi u}\right|\displaystyle\int_{0}^\infty \frac{w e^{-|u|w}}{4} dw =  \dfrac{\kappa(u)}{4|u|^2},
\end{align}
where $\kappa(u) =\left|\cos\pi u -\dfrac{\sin \pi u}{\pi u}
\right| + \left|\dfrac{\sin \pi u}{\pi u}
\right|$.
Let us bound $\kappa(u)$ for $|u|\geq 1$. Since $\kappa$ is even, we may assume $u\geq 1$. We can see that $\kappa(u)
\leq \max\{|\cos\pi u|, |\cos \pi u-\frac{2\sin \pi u}{\pi u}|\}\leq \max\{1, |\cos \pi u-\frac{2\sin \pi u}{\pi u}|\}$. Let $h(u)=\cos \pi u-\frac{2\sin \pi u}{\pi u}$. For $1\leq u\leq \frac{3}{2}$, we see that $-1\leq \cos \pi u \leq h(u)\leq \frac{2\sin \pi u}{\pi u}\leq \frac{4}{3\pi}$, and so $\kappa(u)\leq 1$.
Now, assume that $u\geq \frac{3}{2}$. By Cauchy -Schwarz, since $\cos^2x + \sin^2x = 1$, we have that $|h(u)|\leq (1+4(u\pi)^{-2})^{1/2}
\leq (1+4(3\pi/2)^{-2})^{1/2}<12/11$. Therefore, for $|u|\geq 1$, we have obtained that $\kappa(u)< 12/11$. From \eqref{eq:decaigo} and \eqref{eq:rachma2} we deduce \eqref{12_16pm}.

{\bf \em Behavior on $[-\frac{1}{2},\frac{1}{2}]$.} Consider \eqref{eq:monthmayending}. The function $(\sin \pi u)/(\pi u)$ is positive on $(-1,1)$, strictly increasing on
$(-1,0)$, and strictly decreasing on $(0,1)$, as can be seen easily from
$(\tan \pi u)/(\pi u) > 1$ for $0<u<1/2$.
Moreover, for $u<0$,
$g_1(u) = \int_0^\infty -b'(|\lambda|/2+w) e^{u w} dw$ is positive with $g_1'(u)>0$, whereas, for $u>0$,
$g_2(u) = \int_{-\infty}^0 b'(|\lambda|/2+w) e^{u w} dw$ is negative with
$g_2'(u)>0$.  Hence,
$K_{|\lambda|/2}-E_{|\lambda|/2}$ is positive and strictly increasing on
$(-1,0)$, but negative and strictly increasing on $(0,1)$. Finally, 
$$
(K_{|\lambda|/2}-E_{|\lambda|/2})(0^+)-(K_{|\lambda|/2}-E_{|\lambda|/2})(0^-) = E_{|\lambda|/2}(0^-) - E_{|\lambda|/2}(0^+) = -1.    
$$
The values at $u=\pm\frac{1}{2}$ follow from 
$K_{|\lambda|/2}(\pm 1)=E_{|\lambda|/2}(\pm 1)$, which is a special case of
\eqref{eq:apprinter}. 
\end{proof}

\begin{lemma}\label{lem:barbar}
Let $\alpha>0$ and $u_0 \leq -\frac{1}{2}$. Then 
$$\int_{u_0}^{-\frac{1}{2}} \frac{e^{-\alpha u}}{u^2} du \leq \dfrac{16}{11}\cdot \frac{e^{-\alpha u_0}}{\alpha u_0^2} + \kappa_\alpha,
$$
where $\kappa_\alpha=2e^{\alpha/2}-\alpha\Ei(\alpha/2)+2.525\alpha$, and $\Ei(x)$ is the exponential integral function.
\end{lemma}
\begin{proof}
Clearly $\int_{u_0}^{-1/2} \frac{e^{-\alpha u}}{u^2} du = \alpha \int_{\alpha u_0}^{-\alpha/2} \frac{e^{-t}}{t^2} dt$. Then,
\begin{align} \label{10_32pm}
\int_{\alpha u_0}^{-\alpha/2} \frac{e^{-t}}{t^2} dt = -\left. \frac{e^{-t}}{t} \right|_{\alpha u_0}^{-\alpha/2} - \int_{\alpha u_0}^{-\alpha/2} \frac{e^{-t}}{t} dt = \frac{e^{-\alpha u_0}}{\alpha u_0} + \frac{e^{\alpha/2}}{\alpha/2} + \Ei(-\alpha u_0) - \Ei(\alpha/2).
\end{align}
Let $\beta>2$, and define $g(x)=\frac{e^x}{x} \big(1+\frac{\beta}{(\beta-2)x}\big)-\Ei(x)$. Then $g'(x)=\frac{2e^x(x-\beta)}{(\beta-2)x^3}$, which implies that $g(x)\geq g(\beta)$ for all $x>0$. Thus,
$$
\Ei(x) \leq \frac{e^x}{x} \left(1 + \frac{\beta}{(\beta-2)x}\right) - \frac{e^\beta(\beta-1)}{\beta(\beta-2)}  +\Ei(\beta).
$$
Using this in \eqref{10_32pm} with $x=-\alpha u_0$ we get that, for any $\beta>2$,
\begin{align*}
\int_{u_0}^{-\frac{1}{2}} \frac{e^{-\alpha u}}{u^2} du \leq \dfrac{\beta}{\beta-2}\cdot \frac{e^{-\alpha u_0}}{\alpha u_0^2} + 2e^{\alpha/2}-\alpha\Ei(\alpha/2)+\alpha\left(\Ei(\beta)-\frac{e^\beta(\beta-1)}{\beta(\beta-2)}\right).
\end{align*}
Finally, by setting $\beta=6.4$ and employing rigorous numerics, we obtain the result.
\end{proof}

\begin{lemma}\label{lem:artanor}
Let $I_\lambda$, $\lambda\ne 0$, be as in \eqref{eq:truncexp} and let $\varphi_\lambda$ as in \eqref{eq:sonnenblum}. Let $\alpha>0$. Then
\begin{align} \label{12_41am}
\left\|e^{-\alpha u}\big(\widehat{\varphi_\lambda}(u)-I_\lambda(u)\big)\right\|_{L^1([u_0,\infty))}\leq c_{0}(\alpha) + \frac{1}{11\pi}\cdot \frac{e^{-\alpha u_0}}{\alpha u_0^2},
\end{align}
and
\begin{align} \label{12_41am2}
\left\|e^{-\alpha u}\big(\widehat{\varphi_\lambda}(u)-I_\lambda(u)\big)\right\|_{\text{$\TV$ on $[u_0,\infty)$}}\leq c_{1}(\alpha) + c_{2}(\alpha)\cdot \frac{e^{-\alpha u_0}}{\alpha u_0^2},
\end{align}
where $c_0(\alpha)=\frac{e^{\alpha/2}}{2}+\frac{e^{-\alpha/2}}{8\pi}+\frac{\kappa_\alpha}{16\pi}$, $c_1(\alpha)=2e^{\alpha/2}+\left(\frac{\alpha}{8\pi} + \frac{3}{11}\right)e^{-\alpha/2} + \left(\frac{\alpha}{16\pi} + \frac{3}{22}\right)\kappa_\alpha$, and $c_2(\alpha)=\frac{16}{11}\left(\frac{\alpha}{16\pi} + \frac{3}{22}\right)$, where $\kappa_\alpha$ is defined in Lemma \ref{lem:barbar}. In particular, for $\alpha\in (0,\frac{1}{2}]$, $c_{0}(\alpha)<\frac{16}{21}$, $c_{1}(\alpha)<\frac{17}{5}$, and $c_2(\alpha)<\frac{2}{9}$, and, for $\alpha\in (0,\frac{1}{6}]$, $c_0(\alpha)\leq \frac{2}{3}$ and $c_1(\alpha)\leq 2.85$.
\end{lemma}
\begin{proof}
Let $f=\widehat{\varphi_\lambda}-I_\lambda$. Then $\left\|e^{-\alpha u}f(u)\right\|_{L^1([u_0,\infty))}$ is 
\begin{align*} 
\int_{u_0}^{-\frac{1}{2}}|e^{-\alpha u} f(u)|du + \left\|e^{-\alpha u}f(u)\right\|_{L^1([-\frac{1}{2},\frac{1}{2}])}+ \int_{\frac{1}{2}}^\infty|e^{-\alpha u} f(u)|du.
\end{align*}
We bound the second integral here by Lemma \ref{lem:arborio}, followed by a simple inequality:
\begin{align} \label{int1}
\int_{\frac{1}{2}}^\infty|e^{-\alpha u} f(u)|du \leq \dfrac{1}{16\pi}\int_{\frac{1}{2}}^\infty \frac{e^{-\alpha u}}{u^2} du \leq \dfrac{e^{-\alpha/2}}{16\pi}\int_{\frac{1}{2}}^\infty \frac{1}{u^2} du = \dfrac{e^{-\alpha/2}}{8\pi}.
\end{align}
By Lemma \ref{lem:barbar}, for $u_0\leq -\frac{1}{2}$,
\begin{align} \label{int2}
\int_{u_0}^{-\frac{1}{2}}|e^{-\alpha u} f(u)|du \leq \dfrac{1}{16\pi}\int_{u_0}^{-\frac{1}{2}}\dfrac{e^{-\alpha u}}{u^2}du \leq \dfrac{1}{11\pi}\cdot\frac{e^{-\alpha u_0}}{\alpha u_0^2} + \dfrac{\kappa_\alpha}{16\pi}.
\end{align}
If $u_0> -\frac{1}{2}$, there is no term to bound here and the bound holds trivially.

It only remains to bound $\left\|e^{-\alpha u}f(u)\right\|_{L^1([-\frac{1}{2},\frac{1}{2}])}$. For convenience, we consider $\xi f(u)$ instead of $f(u)$, where $\xi = \sgn(\lambda)$; the $L^1$ norm is not affected. By Lemma \ref{lem:arborio}, the function $e^{-\alpha u}\cdot \xi f(u)$ is negative and increasing on $(0,\frac{1}{2})$, and $\xi f(u)$ is positive and increasing on $(-\frac{1}{2},0)$. Thus, $-\xi f(0^+)=|f(0^+)|$ and $\xi f(0^-)=|f(0^-)|$. In consequence, $\left\|e^{-\alpha u}f(u)\right\|_{L^1([-\frac{1}{2},\frac{1}{2}])}$ equals
$$\begin{aligned}\int_{-\frac{1}{2}}^{\frac{1}{2}} |e^{-\alpha u} \cdot \xi f(u)| du & = 
\int_{-\frac{1}{2}}^{0} e^{-\alpha u} \cdot \xi f(u) du +
\int_{0}^{\frac{1}{2}} e^{-\alpha u} \cdot -\xi f(u) du \\
&\leq \frac{\xi f(0^-)e^{\alpha/2}}{2} -\dfrac{\xi f(0^+)}{2}
= \dfrac{|f(0^-)|e^{\alpha/2}+|f(0^+)|}{2}.
\end{aligned}$$
By Lemma \ref{lem:arborio}, $|f(0^-)|+|f(0^+)|= 1$.
Hence,
$\left\|e^{-\alpha u}f(u)\right\|_{L^1([-\frac{1}{2},\frac{1}{2}])} \leq \frac{e^{\alpha/2}}{2}$. Combining these bounds we get \eqref{12_41am}.

On the other hand, since $f$ is differentiable in $\mathbb{R}\setminus\{0\}$,
$\left\|e^{-\alpha u}f(u)\right\|_{\text{$\TV$ on $[u_0,\infty)$}}$ is 
\begin{align*}
\int_{u_0}^{-\frac{1}{2}}|(e^{-\alpha u} f(u))'|du + \|e^{-\alpha u} f(u)\|_{\text{$\TV$ on $[-\frac{1}{2},\frac{1}{2}]$}}+ \int_{\frac{1}{2}}^\infty|(e^{-\alpha u} f(u))'|du.
\end{align*}
By Lemma \ref{lem:arborio}, for $|u|\geq \tfrac{1}{2}$,
$$|(e^{-\alpha u} f(u))'| \leq \alpha e^{-\alpha u} |f(u)| + e^{-\alpha u} |f'(u)|
\leq \left(\frac{\alpha}{16\pi} + \frac{3}{22}\right)\frac{e^{-\alpha u}}{u^2}.$$
Thus, we bound the integrals as in \eqref{int1} and \eqref{int2}, obtaining
\begin{align*}
\int_{\frac{1}{2}}^\infty|(e^{-\alpha u} f(u))'|du & \leq \left(\dfrac{\alpha}{8\pi} + \frac{3}{11}\right)e^{-\alpha/2},
\end{align*}
and
\begin{align*}
\int_{u_0}^{-\frac{1}{2}}|(e^{-\alpha u} f(u))'|du \leq \left(\dfrac{\alpha}{16\pi} + \frac{3}{22}\right)\left(\frac{16}{11} \frac{e^{-\alpha u_0}}{\alpha u_0^2} + \kappa_\alpha\right).
\end{align*}
It only remains to bound $\|e^{-\alpha u}f(u)\|_{\text{$\TV$ on $[-\frac{1}{2},\frac{1}{2}]$}}$ (considering $\xi f(u)$ instead of $f(u)$, since the total variation is not affected). By Lemma \ref{lem:arborio}, 
$$\|e^{-\alpha u} \cdot \xi f(u)\|_{\text{$\TV$ on $[-\frac{1}{2},\frac{1}{2}]$}} =
1 + \|e^{-\alpha u} \cdot \xi f(u)\|_{\text{$\TV$ on $[-\frac{1}{2},0)$}}+\|e^{-\alpha u} \cdot \xi f(u)\|_{\text{$\TV$ on $(0,\frac{1}{2}]$}}.$$
By the previous monotonicity arguments,  $\|e^{-\alpha u} \sgn(\lambda)f(u)\|_{\text{$\TV$ on $(0,\frac{1}{2}]$}}$ is 
$e^{-\alpha/2}\cdot \xi f(1/2) - \xi f(0^+)$, which is at most $|f(0^+)|$. 
Finally, since $\xi f(u)$ is positive and increasing on $(-\frac{1}{2},0)$, and continuous at $\tfrac{1}{2}$,
the total variation on $[-\frac{1}{2},0)$ is
$$\begin{aligned}\int_{-\frac{1}{2}}^0 |(e^{-\alpha u} \cdot \xi f(u))'| du &\leq
\int_{-\frac{1}{2}}^0 |(e^{-\alpha u})' \cdot \xi f(u)| du +
\int_{-\frac{1}{2}}^0 |e^{-\alpha u} \cdot (\xi f)'(u)| du\\
&\leq \xi f(0^-) \cdot (e^{\alpha/2}-1) + e^{\alpha/2} \xi f(0^-)
= (2 e^{\alpha/2}-1)|f(0^-)|.
\end{aligned}$$
By $|f(0^-)|+|f(0^+)|= 1$, we obtain that $\|e^{-\alpha u} \cdot \xi f(u)\|_{\text{$\TV$ on $[-\frac{1}{2},\frac{1}{2}]$}} \leq 1 + |f(0^+)| + |f(0^-)| (2 e^{\alpha/2}-1) \leq 1 + (2 e^{\alpha/2} - 1) =
2 e^{\alpha/2}$. This implies \eqref{12_41am2}. Since $\kappa_\alpha$ is increasing for $\alpha\in (0,\frac{1}{2}]$, $c_0(\alpha), c_1(\alpha)$ and $c_2(\alpha)$ are increasing for $\alpha\in (0,\frac{1}{2}]$.
We evaluate $c_i\left(\frac{1}{2}\right)$  and $c_i\left(\frac{1}{6}\right)$ by rigorous numerics.
\end{proof}

\subsection{Expressions in terms of special functions}

Here is a reasonably ``closed-form'' expression for the approximants
from \S \ref{sec:modidon}. We have used it only for plotting Figure \ref{fig:carnlittplot}.

\begin{lemma}
Let $K_\nu$, $\nu>0$, be as in \eqref{eq:ombroso}. Then
\begin{equation}\label{eq:argor1}
K_\nu(z) = -\frac{\sin \pi z}{\pi} \left(\Phi(-e^{-\nu},1,-z) +
\frac{1/z}{1+e^{-\nu}}\right),
\end{equation}
where $\Phi(z,s,\alpha)$ is the Lerch transcendent. Consequently,
for $\varphi_\lambda$ as in \eqref{eq:sonnenblum},
\begin{equation}\label{eq:argor2}
\widehat{\varphi_\lambda}(z) = -\sgn(\lambda) \frac{\sin 2 \pi z}{\pi}
\left(\Phi(-e^{-|\lambda|/2},1,- 2 \sgn(\lambda) z) +
\frac{\sgn(\lambda)/2 z}{1+e^{-|\lambda|/2}}\right)
.\end{equation}
\end{lemma}
The {\em Lerch transcendent} $\Phi(z,s,\alpha)$ (not to be confused with $\Phi_\lambda$) is a special function defined by
\begin{equation*}
\Phi(z,s,\alpha) = \sum_{n=0}^\infty \frac{z^n}{(n+\alpha)^s}\end{equation*}
for $|z|<1$, provided that $s\in \mathbb{Z}_{>0}$ and
$\alpha\not\in \mathbb{Z}_{\leq 0}$  (or some other conditions that we need not
worry about)
\cite[\S 25.14]{zbMATH05765058}. For $|z|<1$ and $s\in \mathbb{Z}_{>0}$, $\sin \pi z \cdot \Phi(z,s,\alpha)$ tends
to a limit as $\alpha$ approaches a non-positive integer, and so
\eqref{eq:argor1} and \eqref{eq:argor2} still make sense for
$\alpha\in\mathbb{Z}_{\leq 0}$. 
\begin{proof}
We may write $K_\nu(z) = \frac{\sin \pi z}{\pi} k_\nu(z)$,
where 
$$\begin{aligned}k_\nu(z) &= 
\sum_{n}(-1)^n\left(\dfrac{e^{-\nu n}}{z-n}-\dfrac{e^{-\nu n}}{z}\right)
\\ &= - \sum_{n=0}^\infty \frac{(-e^{-\nu})^n}{n-z} - 
\sum_{n=0}^\infty \frac{(-e^{-\nu})^n}{z} = 
 - \Phi(-e^{-\nu},1,-z) - \frac{1/z}{1+e^{-\nu}}.
\end{aligned}$$
\end{proof}

\section{An explicit estimate on $\zeta(s)$}

Let us make explicit a well-known application of the functional equation.

\begin{lemma} \label{lem:zetinvbound}
For $s\in \mathbb{C}$ with $\Re s\leq 1/2$,     
	\begin{equation}\label{eq:fromfunceq} 
		\bigg|\dfrac{1}{\zeta(s)}\bigg|\leq \left(\frac{2\pi e}{|1-s|}\right)^{\!\frac{1}{2}-\Re s} 
             \frac{e^{\frac{\pi |\Im s|}{2}}}{2 \left|\sin \frac{\pi s}{2}\right|}
        \frac{\sqrt{e}}{|\zeta(1-s)|}.
	\end{equation}
    If $\Re s \leq 0$ and $|\Im s|\geq 1$, then
\begin{equation}\label{eq:fromfunceq2}
\bigg|\dfrac{1}{\zeta(s)}\bigg|\leq\left(\frac{2\pi e}{|\Im s|}\right)^{\frac{1}{2} - \Re s} 
    \frac{\sqrt{e}}{|\zeta(1-s)|}.\end{equation}
\end{lemma}
Note that $\frac{e^{\frac{\pi |\Im s|}{2}}}{\left|\sin \frac{\pi s}{2}\right|}$ is bounded
outside any union of balls of constant radius around $0, \pm 2, \pm 4,\dotsc$.
\begin{proof}
    Recall the functional equation
	\begin{equation}\label{eq:funceq}
    \zeta(s) = (2\pi)^{s-1}\cdot 2 \sin \left(\frac{\pi s}{2}\right) \Gamma(1-s) \zeta(1-s),\end{equation}
	valid in all of $\mathbb{C}$, and Stirling's formula
    \begin{equation}\label{eq:stirling}\log \Gamma(z) = \left(z -\frac{1}{2}\right)\log z - z + \log \sqrt{2 \pi}+ \frac{1}{12 z} + O^*\left(\frac{\sqrt{2}}{180 |z|^3}\right)
    \end{equation} for $\Re z\geq 0$ (see [GR, 8.344]).
    Let us take real parts: by \eqref{eq:funceq} and
    \eqref{eq:stirling}, respectively,
    \begin{equation*}    \Re \log \zeta(s) = 
		(\log 2\pi) (\Re s - 1) + \frac{\pi |\Im{s}|}{2} + \log \frac{2 \left|\sin \frac{\pi s}{2}\right|}{e^{\pi |\Im s|/2}} + 
		\Re \log \Gamma(1-s) + \Re \log \zeta(1-s),\end{equation*}
\begin{equation*}\Re \log \Gamma(z) = 
\left(\Re z -
		\frac{1}{2}\right) \log |z| - \Im z\cdot\arg z - \Re z + \log \sqrt{2 \pi} +
        \frac{\Re z}{12 |z|^2} + 
        O^*\left(\frac{\sqrt{2}}{180 |z|^3}\right).
\end{equation*}
For $z\in \mathbb{C}\setminus (-\infty,0]$, $\Im z \cdot \arg z = \frac{\pi |\Im z|}{2} - |\Im z| \arctan \frac{\Re z}{|\Im z|}$. Hence
\[\Re \log \Gamma(z) = \left(\Re z -
		\frac{1}{2}\right) \log |z| - \frac{\pi |\Im z|}{2}  + \log \sqrt{2 \pi} + r_1(z),
\]
where $r_1(z) = |\Im z| \left(\arctan \frac{\Re z}{|\Im z|} - \frac{\Re z}{|\Im z|}\right) + \frac{\Re z}{12 |z|^2} + O^*\left(\frac{\sqrt{2}}{180 |z|^3}\right)$.
Letting $z=1-s$, we obtain
$$\begin{aligned}\Re \log \zeta(s)
&= \left(\frac{1}{2}-\Re s\right) \log \frac{|1-s|}{2\pi} + \log \frac{2 \left|\sin \frac{\pi s}{2}\right|}{e^{\pi |\Im s|/2}} 	 + \Re \log \zeta(1-s) + r_1(1-s),\end{aligned}$$ or, what is the same,
    \begin{equation*}|\zeta(s)| = 
 \left(\frac{|1-s|}{2\pi}\right)^{\frac{1}{2}-\Re s} \cdot \frac{2 \left|\sin \frac{\pi s}{2}\right|}{e^{\pi |\Im s|/2}} e^{r_1(1-s)} |\zeta(1-s)|.\end{equation*}

  For $\Re z \geq 1/2$ (say),
  $(\Re z)/(12 |z|^2) + O^*(\sqrt{2}/(180 |z|^3)
 > 0$, and so $r_1(z)\geq - \Re z$; thus
\eqref{eq:fromfunceq} follows.
    Now, $\left|\sin \frac{\pi s}{2}\right|  = \left|\frac{1}{2 i} \left(e^{\frac{\pi i s}{2}} -
    e^{\frac{-\pi i s}{2}}\right)\right|\geq e^{\pi |\Im s|/2}
    (1 - e^{- \pi |\Im s|})/2$. For $t\geq 1$, 
    $(1-e^{-\pi t}) (|1+ i t|/t)^{1/2}\geq ((1 - 4 e^{-\pi t}) (1+1/t^2))^{1/4}
    \geq 1$, since, still for $t\geq 1$, $(1 + 1/t^2) (1- 1/5 t^2)\geq 1$ and $e^{-\pi t}\leq e^{-\pi}/t^2
    < 1/20 t^2$. Hence, for $\Re s \leq 0$ and $|\Im s|\geq 1$,
    \[|\zeta(s)| \geq \left(\frac{|\Im s|}{2\pi e}\right)^{\frac{1}{2} - \Re s} 
    \cdot \frac{|\zeta(1-s)|}{\sqrt{e}}.\]

\end{proof}

\bibliographystyle{alpha}
\bibliography{chirhelf}
\end{document}